\documentclass[11pt,a4paper]{article}
\usepackage{amsmath,amsthm,amscd,amssymb,graphicx}
\usepackage{mathtools,bm,mathrsfs}
\usepackage{enumitem}
\usepackage{microtype}
\usepackage{booktabs,array}
\usepackage[T1]{fontenc}
\usepackage{lmodern}
\usepackage{xcolor}

\newtheorem{theorem}{Theorem}[section]
\newtheorem{lemma}[theorem]{Lemma}
\newtheorem{proposition}[theorem]{Proposition}
\newtheorem{corollary}[theorem]{Corollary}
\newtheorem{assumption}[theorem]{Assumption}
\theoremstyle{definition}
\newtheorem{definition}[theorem]{Definition}

\newtheorem{remark}[theorem]{Remark}
\numberwithin{equation}{section}

\usepackage[text={178mm,254mm},centering]{geometry}
\usepackage{xurl}
\usepackage{hyperref}
\newcolumntype{P}[1]{>{\raggedright\arraybackslash}p{#1}}

\hypersetup{
  colorlinks=true,linkcolor=blue,citecolor=blue,urlcolor=blue,
  pdfauthor={Jiangwei Zhang and Jianhua Huang},
  pdftitle={Unique Ergodicity for Strongly Damped Stochastic Wave Equations under Finite-Rank Pure-Jump Levy Forcing}
}

\title{Unique Ergodicity for Strongly Damped Stochastic Wave Equations under Finite-Rank Pure-Jump L\'{e}vy Forcing\thanks{This work is supported by the National Natural Science Foundation of China (Nos. 12371198).}}
\author{%
  Jiangwei Zhang\\[2pt]
  {\small Institute of Applied Physics and Computational Mathematics}\\
  {\small Beijing 100088, P.R. China}\\[7pt]
  Jianhua Huang\thanks{Corresponding author. Email: %
  	{\texttt{jhhuang32@hotmail.com}}.}\\[2pt]
  {\small College of Science, National University of Defense Technology}\\
  {\small Changsha 410073, P.R. China}%
}
\date{}

\begin{document}
\maketitle

\begin{abstract}
{\small
We study a strongly damped stochastic wave equation on the three-dimensional torus driven by finite-rank pure-jump L\'{e}vy noise obtained from subordinated Brownian motion. Lyapunov estimates and an asymptotic compactness decomposition yield invariant measures, and a small-noise event gives weak irreducibility. To prove uniqueness, we combine conditional Malliavin calculus with terminal observability and estimates for the full variational flow starting from high modes. A control constructed on alternating stopping blocks yields the e-property under a fixed-rank balance condition involving high-mode decay, spectral leakage, and the random clock. Consequently, the invariant measure is unique and ergodic, and the Ces\`aro averages of the transition laws from every initial state converge weakly to it. We verify this condition for sufficiently small positive cubic coefficients at each fixed forcing operator of sufficiently large rank. Under a separate explicit smallness condition, synchronous coupling proves the e-property and weak convergence of the transition laws, without a lower bound on the forcing rank.
}
\end{abstract}



\setlength{\abovedisplayskip}{9pt plus 2pt minus 4pt}
\setlength{\belowdisplayskip}{9pt plus 2pt minus 4pt}
\setlength{\abovedisplayshortskip}{0pt plus 2pt}
\setlength{\belowdisplayshortskip}{5pt plus 2pt minus 2pt}

\section{Introduction}\label{DeSubintro}

In this paper, we are concerned with the strongly damped stochastic wave equation on the three-dimensional torus
\begin{equation}\label{eq:wave-second-order}
\mathrm{d}(\partial_t u)-\Delta(\partial_t u){}\mathrm{d} t-\Delta u{}\mathrm{d} t+P_0g(u){}\mathrm{d} t
=Q{}\mathrm{d} L_t,
\qquad (u(0),\partial_t u(0))=(u_0,v_0),
\end{equation}
where $P_0$ is the zero-mean projection and $Q$ acts only on finitely many Fourier modes of the velocity component.  The noise is given by
\begin{equation*}
L_t=W_{S_t},
\end{equation*}
where $W$ is a finite-dimensional Brownian motion and $S$ is an independent driftless subordinator with infinite activity and an exponential moment, as specified in Assumption~\ref{ass:levy}. Hence $L$ is a finite-dimensional pure-jump L\'{e}vy process.

The main problem is to find sufficient conditions under which a fixed finite-dimensional pure-jump forcing determines the long-time statistical behavior of the infinite-dimensional wave equation.  We begin with existence of invariant measures and weak irreducibility, and then derive a fixed-rank balance condition yielding the e-property and uniqueness of the invariant measure.  The estimates separate strong damping, the coupling between low and high modes, and the random clock. For a small cubic coefficient, we also give an independent proof by synchronous coupling.

The long-time behavior of stochastic damped wave equations has been studied by energy estimates, compactness arguments, coupling methods, and Malliavin calculus; see \cite{BarbuDaPratoTubaro2007,BrzezniakOndrejatSeidler2016,LiLi2024,LiLiangSun2025,Martirosyan2014,Tolomeo2020,ForlanoTolomeo2024} and \cite{DaPratoZabczyk2014,PeszatZabczyk2007}.  A closely related deterministic-time result is \cite{Li2026}.  In that work, finite-dimensional Gaussian forcing is used to prove uniqueness and exponential mixing for the three-dimensional strongly damped wave equation.  For equations driven by jump noise, we refer to \cite{BoShiWang2010,JiangWangWang2012,Majka2017,WangXu2018,SunXieXu2019,WangYangZhaiZhang2025}.  Time-changed Malliavin methods for degenerate stochastic fluid equations were developed in \cite{PengZhaiZhang2024,HuangPengWangZhang2025}.

The jump forcing requires estimates on random operational-time intervals and introduces inverse-clock occupation costs. These are absent from the deterministic-clock Gaussian problem in \cite{Li2026}; our general result gives Ces\`aro convergence rather than exponential mixing. A separate difficulty comes from the wave equation: multiplication by $g'(u)$ can transfer a tangent vector from high to low modes, so the control must estimate the full variational flow.

The proof has three components. Terminal wave observability and strict clock increments give a lower-tail estimate for the low-mode Malliavin covariance, uniformly on bounded energy sets. Coupled high--low estimates control the full tangent flow and quantify the spectral leakage by $d_m\to0$. Finally, Tikhonov control on alternating stopping blocks combines low-mode regularization with high-mode decay. Conditioning on the clock allows the Skorohod integral to be constructed on each controlled Gaussian factor. The resulting balance condition is stated below; the covariance tolerance and the Tikhonov parameter are chosen afterwards. For $g(s)=as^3$, we retain the coefficient $a$ in the variational estimates and show that this condition holds when $a>0$ is sufficiently small.

We now state the main results.
Let $\{P_t\}_{t\ge0}$ be the Markov semigroup generated by \eqref{eq:first-order},
\begin{equation*}
P_t\varphi(z)=\mathbb{E}[\varphi(Z_t^z)],
\qquad \varphi\in B_b(\mathbb{H}).
\end{equation*}

\begin{theorem}[Invariant measure and weak irreducibility]\label{thm:main}
	Suppose Assumptions \ref{ass:g} and \ref{ass:levy} hold, and let $Q=Q_m$ be as in \eqref{eq:Q}.  Then the Markov semigroup generated by \eqref{eq:first-order} possesses at least one invariant measure $\mu$, and every invariant measure satisfies the moment bound \eqref{eq:mu-V} below.
	Moreover, the transition semigroup is weakly irreducible in the sense of Proposition \ref{prop:weak-irr}.
\end{theorem}

Let $p_0\ge40$ be the moment exponent in Proposition~\ref{prop:high-branch}, and set $\theta_0=4/p_0$. With the constants specified in Section~\ref{subsec:residual-control}, define
\begin{equation*}
\mathfrak q_m(\kappa,\chi,\tilde\eta,\widehat\delta)
:=\mathscr C\exp\{2\theta_0\widehat c\widehat\mu\}
\bigl(\widehat q(\chi)^{\theta_0}+d_m^{\theta_0}\bigr),
\qquad \widehat\mu=\overline\Xi+\widehat\delta.
\end{equation*}

\begin{theorem}[Fixed-rank criterion for unique ergodicity]\label{thm:conditional-ergodicity}
	Assume the hypotheses of Theorem \ref{thm:main}.  Let $\widehat m$ be a threshold beyond which
	Proposition \ref{prop:high-branch} holds, and fix $m\ge \widehat m$.  Suppose there exist a centering margin $\widehat\delta>0$ and admissible stopping/control parameters $\kappa,\chi,\tilde\eta$, as specified in Sections \ref{sec:stopping} and \ref{sec:control}, such that
	\begin{equation*}
	\mathfrak q_m(\kappa,\chi,\tilde\eta,\widehat\delta)<\frac14.
	\end{equation*}
	Equivalently, the fixed-rank balance \eqref{eq:fixed-rank-admissible} holds.  Here $\widehat q(\chi)$ is the deterministic contraction along the high branch, $d_m\to0$ is the wave-specific spectral leakage generated by the high--low coupling, and $\widehat\mu$ is the mean normalized clock cost plus the centering margin.  Then $\{P_t\}_{t\ge0}$ has the e-property on $\mathbb{H}$ in the bounded-set-uniform form \eqref{eq:e-property} and possesses a unique invariant measure $\mu_*$.  In particular, $\mu_*$ is ergodic and satisfies \eqref{eq:mu-V}.  Moreover, for every $z\in\mathbb{H}$,
	\begin{equation}\label{eq:cesaro-convergence}
	\frac1T\int_0^T P_t^*\delta_z\,\mathrm dt
	\Longrightarrow \mu_*,
	\quad\text{as }T\to\infty.
	\end{equation}
\end{theorem}

\begin{corollary}[Cubic nonlinearity]\label{cor:cubic-ergodicity}
Let $m\ge m_0$ as in Proposition~\ref{prop:cubic-realization}, fix $Q_m$ as in
\eqref{eq:Q}, and suppose Assumption~\ref{ass:levy} holds.
If $g(s)=as^3$ with $0<a<a_*(m,Q_m,S)$, then the semigroup has the e-property
and a unique invariant measure. This measure is ergodic, and
\eqref{eq:cesaro-convergence} holds for every initial state in $\mathbb H$.
\end{corollary}

Theorem~\ref{thm:conditional-ergodicity} requires the fixed-rank balance in addition to Assumptions~\ref{ass:g}--\ref{ass:levy}. Since the clock and moment constants depend on $Q_m$, the limit $d_m\to0$ alone does not verify it; see Remark~\ref{rem:fixed-rank-interpretation}. The criterion quantifies covariance regularization, full-flow leakage, and the clock cost for nonlinearities satisfying Assumption~\ref{ass:g}. Proposition~\ref{prop:cubic-realization} verifies the balance for the cubic family at each fixed forcing operator of sufficiently large rank. It yields unique ergodicity and Ces\`aro convergence for sufficiently small positive cubic coefficients, without a quantitative mixing rate.

For a sufficiently small cubic coefficient, Proposition~\ref{prop:synchronous} gives an independent result: weak convergence of $P_t^*\delta_z$ itself, with no minimum forcing rank, under a separate explicit smallness condition. The two smallness thresholds depend on different estimates, and no ordering between them is established. The cubic realization provides a concrete verification of the general fixed-rank criterion; it does not establish a parameter range beyond synchronous coupling.

The role of degeneracy also differs from sparse-forcing results such as
\cite{PengZhaiZhang2024,HuangPengWangZhang2025}: the general criterion here
forces every mode of a finite forced low-mode subspace and does not use a
H\"ormander bracket condition. Jump coupling provides another approach to
degenerate SPDEs \cite{SunXieXu2019}, but that result assumes a bounded
globally Lipschitz nonlinearity and a self-adjoint dissipative generator.
Those hypotheses do not directly cover the cubic wave equation.

The remainder of this paper is organized as follows. Section~\ref{sec:setting} introduces the functional setting and assumptions and proves well-posedness. Section~\ref{sec:lyapunov} establishes Lyapunov estimates and the existence of invariant measures. Section~\ref{sec:stopping} develops stopping-time estimates, conditional Malliavin calculus, and covariance bounds. Section~\ref{sec:control} proves the e-property under the fixed-rank balance condition and verifies this condition for the cubic nonlinearity. Section~\ref{sec:ergodicity} proves unique ergodicity and gives the independent synchronous coupling argument.

\section{Functional setting, assumptions, and noise}\label{sec:setting}

All function spaces are real, and the eigenfunctions below form a real Fourier basis. The symbol $C$ denotes a positive constant that may change from line to line. Constants in the contraction criterion are fixed before choosing the covariance tolerance and the Tikhonov parameter. Uniformity is asserted only for the parameters specified in each result.

\subsection{Phase space and abstract equation}
Let
\begin{equation*}
H_0=\left\{u\in L^2(\mathbb{T}^3):\int_{\mathbb{T}^3}u(x){}\mathrm{d} x=0\right\},
\qquad A=-\Delta,
\end{equation*}
with periodic boundary conditions.  We denote by $P_0:L^2(\mathbb{T}^3)\to H_0$ the zero-mean
orthogonal projection,
\begin{equation*}
P_0f=f-\frac1{|\mathbb{T}^3|}\int_{\mathbb{T}^3}f(x){}\mathrm{d} x.
\end{equation*}

 Let $(e_k)_{k\ge1}$ be an orthonormal basis of $H_0$ such that
$Ae_k=\lambda_k e_k$,
$0<\lambda_1\le\lambda_2\le\cdots$, with $\lambda_k\to\infty$.
For $r\ge0$, set $H^r=D(A^{r/2})$ with norm $\|u\|_{H^r}=\|A^{r/2}u\|$; in particular,
$H^0=H_0$. We also use $H^{-1}=(H^1)'$ with pivot space $H_0$.
The extension $A:H^1\to H^{-1}$ is defined by
$\langle Au,h\rangle_{H^{-1},H^1}=(A^{1/2}u,A^{1/2}h)$.
The phase space is
\begin{equation*}
\mathbb{H}=H^1\times H_0,
\qquad
\|(u,v)\|_{\mathbb{H}}^2=\|A^{1/2}u\|^2+\|v\|^2.
\end{equation*}
For $N\ge1$, let $P_Nf=\sum_{k=1}^N(f,e_k)e_k$ on $H_0$, and set
\[
 \Pi_N=\operatorname{diag}(P_N,P_N)\quad\text{on }\mathbb H.
\]
The same spectral formula extends $P_N$ to $H^1$ and $H^{-1}$; these projections
are contractions in the corresponding norms.
Writing $v=\partial_t u$ almost everywhere in time, equation \eqref{eq:wave-second-order} becomes
\begin{equation}\label{eq:first-order}
\begin{cases}
\,\mathrm{d} u(t)=v(t){}\mathrm{d} t,\\
\,\mathrm{d} v(t)+Av(t){}\mathrm{d} t+Au(t){}\mathrm{d} t+P_0g(u(t)){}\mathrm{d} t
=Q{}\mathrm{d} L_t,\\
(u(0),v(0))=z_0\in\mathbb{H}.
\end{cases}
\end{equation}
We use the equivalent time notations $u_t=u(t)$ and $v_t=v(t)$, and write
$Z_t=(u_t,v_t)$.

\begin{remark}
	Note that the projection in the nonlinear term is essential: for a general nonlinear $g$, zero mean of
	$u$ does not imply zero mean of $g(u)$.  Since every test function used below belongs to $H_0$,
	\begin{equation*}
	(P_0f,h)=(f,h),\qquad h\in H_0,
	\end{equation*}
	so the projection does not alter any of the energy pairings. 
\end{remark}
\subsection{Assumptions on the nonlinearity}
Let
\begin{equation*}
G(s)=\int_0^s g(r){}\mathrm{d} r.
\end{equation*}
\begin{assumption}\label{ass:g}
	The function $g\in C^2(\mathbb{R})$ satisfies $g(0)=0$, and there exist constants
	$\alpha_1,\alpha_2,\alpha_3,\iota>0$ such that for every $s\in\mathbb{R}$,
	\begin{align}
	|g'(s)|&\le \alpha_1(1+|s|^2),\label{eq:g-growth}\\
	|g''(s)|&\le \alpha_3(1+|s|),\label{eq:g-second-growth}\\
	G(s)&\ge \alpha_2|s|^4,\label{eq:G-coercive}\\
	g(s)s&\ge \iota G(s).\label{eq:g-diss}
	\end{align}
\end{assumption}
By the fundamental theorem of calculus and \eqref{eq:g-second-growth}, we derive
\begin{equation}\label{eq:g-prime-remainder}
|g'(s)-g'(0)|\le C\bigl(|s|+|s|^2\bigr),\qquad s\in\mathbb{R}.
\end{equation}

\subsection{Subordinated Brownian forcing}
Let $(\Omega,\mathcal{F},\mathbb{P})$ be a complete probability space.
Let $W=(W^1,\dots,W^m)$ be an $m$-dimensional standard Brownian motion.
Let $S=(S_t)_{t\ge0}$ be a \emph{driftless} pure-jump subordinator, independent of $W$,
with L\'evy measure $\nu_S$.
For $t\ge0$, define
\begin{equation*}
\mathcal{G}_t
=
\sigma\bigl(S_r:0\le r\le t\bigr)
\vee
\sigma\bigl(W_{s\wedge S_t}:s\ge0\bigr).
\end{equation*}
We denote by $\{\mathcal{F}_t\}_{t\ge0}$ the usual augmentation of
$\{\mathcal{G}_t\}_{t\ge0}$, that is, its complete and right-continuous
augmentation. Put
\begin{equation*}
L_t=W_{S_t},\qquad t\ge0.
\end{equation*}
Then $L=\{L_t\}_{t\ge0}$ is adapted to $\{\mathcal{F}_t\}_{t\ge0}$.
Unless otherwise stated, all stopping times $\tau_n$ introduced below
are understood with respect to this filtration.  Moreover, $L$ is an $m$-dimensional symmetric pure-jump L\'evy process.  Its L\'evy measure is
\begin{equation*}
\nu_L(B)=\int_0^\infty\int_B
(2\pi r)^{-m/2}\exp\!\left(-\frac{|y|^2}{2r}\right)
\,\mathrm{d} y\,\nu_S(\,\mathrm{d} r),
\qquad B\in\operatorname{Bor}(\mathbb{R}^m\setminus\{0\}).
\end{equation*}

\begin{assumption}\label{ass:levy}
	There exists $\zeta_S>0$ such that
	\begin{equation*}
	\nu_S((0,\infty))=\infty,
	\qquad
	\int_0^\infty (e^{\zeta_S r}-1)\,\nu_S(\,\mathrm{d} r)<\infty.
	\end{equation*}
\end{assumption}

We write
\[
 \bar s=\mathbb E S_1=\int_0^\infty r\,\nu_S(\mathrm dr),\qquad
 \Psi_S(\theta)=\int_0^\infty(e^{\theta r}-1)\nu_S(\mathrm dr),
 \quad 0\le\theta\le\zeta_S.
\]
Then $\mathbb E e^{\theta S_t}=e^{t\Psi_S(\theta)}$ and $\bar s<\infty$.

For each fixed forced dimension $m$, let $\{\mathbf{e}_j^{(m)}\}_{j=1}^m$ be the canonical basis of
$\mathbb{R}^m$ and define the finite-rank operator
\begin{equation}\label{eq:Q}
Q_m\mathbf{e}_j^{(m)}=q_j e_j,
\qquad q_j\ne0,
\qquad 1\le j\le m.
\end{equation}
For fixed $m$, we drop the subscript and write $Q=Q_m$.  Constants may depend on the chosen $Q_m$; the admissibility condition below is formulated directly at that fixed rank, with no uniform dependence on $Q_m$ left implicit.  We use the two noise parameters
\begin{equation}\label{eq:BQ-qmin}
B_Q=\|Q\|_{\mathbb{L}_2(\mathbb{R}^m,H_0)}^2=\sum_{j=1}^mq_j^2,
\qquad
q_{\min}=\min_{1\le j\le m}|q_j|>0.
\end{equation}
On the wave phase space, we write
\begin{equation*}
\mathbf{Q} a=(0,Qa),\qquad a\in\mathbb{R}^m.
\end{equation*}
Although the operator in \eqref{eq:Q} has finite rank and is therefore degenerate on the infinite-dimensional phase space, it forces every mode in the forced low-mode subspace
$H_m=\operatorname{span}\{e_1,\dots,e_m\}$.  Thus, the uniqueness theorem below is a finite-rank determining-space result rather than a sparse-forcing H\"ormander theorem.

\subsection{Lyapunov functional}
We introduce the modified energy used throughout the well-posedness, stopping-time, and invariant-measure arguments.
Fix
\begin{equation}\label{eq:delta}
0<\delta<\min\left\{\frac14,\frac{\lambda_1}{4}\right\}.
\end{equation}
For $z=(u,v)\in\mathbb{H}$, we define
\begin{equation}\label{eq:V}
V(z)=\frac12\bigl(\|A^{1/2}u\|^2+\|v\|^2\bigr)
+\delta(u,v)+\int_{\mathbb{T}^3}G(u(x)){}\mathrm{d} x.
\end{equation}
By Poincar\'e's inequality, \eqref{eq:delta} and \eqref{eq:G-coercive}, we have
\begin{equation*}
\delta |(u,v)|
\le \frac{\delta}{2\sqrt{\lambda_1}}
\bigl(\|A^{1/2}u\|^2+\|v\|^2\bigr)
\le \frac18\bigl(\|A^{1/2}u\|^2+\|v\|^2\bigr).
\end{equation*}
Moreover, \eqref{eq:g-growth} and $g(0)=0$ imply
$|G(s)|\le C(|s|^2+|s|^4)$.  Thus, there are constants
$c_{0},C_0>0$ such that
\begin{equation}\label{eq:V-equivalence}
c_{0}\bigl(\|z\|_{\mathbb{H}}^2+\|u\|_{L^4}^4\bigr)
\le V(z)
\le C_0\bigl(1+\|z\|_{\mathbb{H}}^2+\|u\|_{L^4}^4\bigr).
\end{equation}
In particular, we have
\begin{equation}\label{eq:V-nonnegative}
V(z)\ge0,\qquad z\in\mathbb{H}.
\end{equation}

\subsection{Well-posedness, e-property and weak irreducibility}\label{subsec:markov-framework}
We specify the solution concept used in the energy and variational estimates.
Here $D([0,T];E)$ denotes the space of c\`adl\`ag paths with values in $E$.

\begin{definition}\label{def:integral-solution}
An adapted process $Z=(u,v)$ is a solution of \eqref{eq:first-order} with
initial state $z_0=(u_0,v_0)$ if, for every $T>0$, almost surely,
\[
 u\in C([0,T];H^1),\qquad
 v\in D([0,T];H_0)\cap L^2(0,T;H^1),
\]
and, for all $t\in[0,T]$,
\begin{equation}\label{eq:integral-solution}
 \begin{aligned}
 u_t&=u_0+\int_0^t v_r\,\mathrm dr,\\
 v_t&=v_0-\int_0^t(Av_r+Au_r+P_0g(u_r))\,\mathrm dr+QL_t.
 \end{aligned}
\end{equation}
Here the first identity holds in $H^1$ and the second in $H^{-1}$, with Bochner
integrals. 
\end{definition}

The following well-posedness statement provides the Markov framework used throughout the paper; see also \cite{PeszatZabczyk2007,BoShiWang2010}.

\begin{proposition}\label{prop:wellposed}
	Under Assumptions \ref{ass:g}--\ref{ass:levy}, for every $z_0\in\mathbb{H}$, equation \eqref{eq:first-order} has a unique pathwise solution in the sense of Definition~\ref{def:integral-solution}
	\begin{equation*}
	u\in C([0,T];H^1),\qquad v\in D([0,T];H_0)\cap L^2(0,T;H^1),
	\qquad T>0,
	\end{equation*}
	and $Z=(u,v)$ belongs to $D([0,T];\mathbb{H})$.  Moreover, for every $T>0$, every $z_0\in\mathbb{H}$, and every $\varepsilon>0$,
	\begin{equation}\label{eq:continuous-dependence}
	\lim_{z_0'\to z_0}\mathbb{P}\left(
	\sup_{0\le t\le T}\|Z_t^{z_0'}-Z_t^{z_0}\|_{\mathbb{H}}>\varepsilon
	\right)=0.
	\end{equation}
	In addition, for every integer $p\ge1$ and every $T>0$,
	\begin{equation}\label{eq:finite-horizon-sup}
	\mathbb{E}_{z_0}\left[\sup_{0\le s\le T}(1+V(Z_s))^p\right]
	\le C_{T,p}\bigl(1+V(z_0)^p\bigr).
	\end{equation}
	Thus $Z$ is a strong Markov process and $\{P_t\}$ is Feller.
\end{proposition}

\begin{proof}
Let
\begin{equation*}
 \zeta_t=QL_t \text{~~~and~~~} y_t=v_t-\zeta_t.
\end{equation*}
For almost every sample path, $\zeta\in D([0,T];H_m)$ and is bounded in every Sobolev norm on compact time intervals.  The pair $(u,y)$ solves the random deterministic system
\begin{equation}\label{eq:shift-wellposed}
 \partial_t u=y+\zeta,
 \qquad
 \partial_t y=-Ay-A\zeta-Au-P_0g(u),
 \qquad (u(0),y(0))=(u_0,v_0).
\end{equation}
For $N\ge m$, solve the spectral Galerkin system with initial state $\Pi_Nz_0$.  Applying the modified energy $V(u,y)$ from
\eqref{eq:V}, \eqref{eq:g-diss}, and Young's inequality gives, for almost every $t$,
\begin{equation}\label{eq:pathwise-galerkin-energy}
 \frac{\,\mathrm{d}}{\,\mathrm{d} t}V(u^N,y^N)
 +c\bigl(\|A^{1/2}y^N\|^2+\|A^{1/2}u^N\|^2+\|u^N\|_{L^4}^4\bigr)
 \le C\bigl(1+\|\zeta_t\|_{H^1}^4\bigr),
\end{equation}
with constants independent of $N$.

Indeed, relative to the unforced energy identity, the additional terms are
\[
 (A^{1/2}u^N,A^{1/2}\zeta)-(A^{1/2}y^N,A^{1/2}\zeta)
 +(g(u^N),\zeta)+\delta(\zeta,y^N)
 -\delta(A^{1/2}u^N,A^{1/2}\zeta).
\]
Since $|g(s)|\le C(|s|+|s|^3)$ and $H^1\hookrightarrow L^4$, their sum in absolute value is at most
\[
 \epsilon\bigl(\|y^N\|_{H^1}^2+\|u^N\|_{H^1}^2+\|u^N\|_{L^4}^4\bigr)
 +C_\epsilon(\|\zeta\|_{H^1}^2+\|\zeta\|_{H^1}^4).
\]
Choose $\epsilon$ below the unforced dissipation coefficients. This also gives
\eqref{eq:pathwise-galerkin-energy} without the constant $1$ on its right-hand side.

Since $\zeta$ is bounded on $[0,T]$, \eqref{eq:pathwise-galerkin-energy} prevents explosion and yields
\begin{equation*}
 u^N\ \hbox{bounded in }L^\infty(0,T;H^1),\qquad
 y^N\ \hbox{bounded in }L^\infty(0,T;H_0)\cap L^2(0,T;H^1).
\end{equation*}
Moreover $\partial_t u^N=y^N+\zeta$ is bounded in $L^2(0,T;H^1)$ and
$\partial_t y^N$ in $L^2(0,T;H^{-1})$. Aubin--Lions compactness gives, along a subsequence,
\begin{equation*}
 u^N\to u\quad\hbox{strongly in }L^2(0,T;L^q(\mathbb{T}^3))
 \quad\text{for every }q<6,
\end{equation*}
and weakly in the energy spaces above. After a further subsequence, $u^N\to u$ almost everywhere. The cubic growth and uniform $L^\infty(0,T;L^6)$ bound imply boundedness of $g(u^N)$ in $L^2(0,T;L^2)$; almost everywhere convergence identifies its weak limit as $g(u)$. Passing to the Galerkin limit gives \eqref{eq:shift-wellposed}. For time continuity,
\begin{equation*}
 y\in L^2(0,T;H^1),\qquad
 \partial_t y=-Ay-A\zeta-Au-P_0g(u)\in L^2(0,T;H^{-1}),
\end{equation*}
so the Lions--Magenes lemma gives $y\in C([0,T];H_0)$.  Moreover,
$\partial_t u=y+\zeta\in L^2(0,T;H^1)$ because $\zeta$ has finite spectral range.  Hence
$u\in W^{1,2}(0,T;H^1)\subset C([0,T];H^1)$.  Thus
\begin{equation*}
 u\in C([0,T];H^1),\qquad y\in C([0,T];H_0).
\end{equation*}
Since $\zeta$ is c\`adl\`ag in the finite-dimensional space $H_m$, $v=y+\zeta$ is c\`adl\`ag in $H_0$, and therefore $Z=(u,v)\in D([0,T];\mathbb{H})$.

For uniqueness, let $(u^{(i)},y^{(i)})$, $i=1,2$, solve \eqref{eq:shift-wellposed} with the same forcing path.  The difference contains no random forcing.  On an interval on which
$\|u^{(1)}\|_{H^1}+\|u^{(2)}\|_{H^1}\le R$, the mean-value theorem, \eqref{eq:g-growth}, and $H^1(\mathbb{T}^3)\hookrightarrow L^6(\mathbb{T}^3)$ give
\begin{equation*}
 \|g(u^{(1)})-g(u^{(2)})\|_{L^2}
 \le C_R\|u^{(1)}-u^{(2)}\|_{H^1}.
\end{equation*}
The modified difference energy then satisfies
\begin{equation*}
 \frac{\,\mathrm{d}}{\,\mathrm{d} t}\|(u^{(1)}-u^{(2)},y^{(1)}-y^{(2)})\|_{\mathbb{H}}^2
 \le C_R\|(u^{(1)}-u^{(2)},y^{(1)}-y^{(2)})\|_{\mathbb{H}}^2.
\end{equation*}
Gronwall's inequality gives pathwise uniqueness.

For strong Galerkin convergence, put $v^N=y^N+\zeta$,
$Z^N=(u^N,v^N)$, and $(h_N,k_N)=Z^N-Z$. On each fixed forcing path,
\eqref{eq:pathwise-galerkin-energy} bounds $\|u^N\|_{H^1}+\|u\|_{H^1}$ by
some $M<\infty$, uniformly in $N$ and $t\le T$. The mean-value formula and
$L^{6/5}\hookrightarrow H^{-1}$ give
\[
 \|P_0(g(u^N)-g(u))\|_{H^{-1}}
 \le C(1+\|u^N\|_{L^6}^2+\|u\|_{L^6}^2)\|h_N\|
 \le C_M\|h_N\|.
\]
Since $N\ge m$, the common noise cancels and
$\partial_t h_N=k_N$, $\partial_t k_N=-Ak_N-Ah_N-f_N$, where
$f_N=P_NP_0(g(u^N)-g(u))-(I-P_N)P_0g(u)$.
Testing by $(Ah_N,k_N)$ in the energy pairing and applying Young's inequality yield
\[
 \frac{\mathrm d}{\mathrm dt}\frac{\|(h_N,k_N)\|_{\mathbb H}^2}{2}
 +\frac12\|k_N\|_{H^1}^2
 \le C_M\|(h_N,k_N)\|_{\mathbb H}^2
      +C\|(I-P_N)P_0g(u)\|_{H^{-1}}^2.
\]
Gronwall's inequality therefore gives
\begin{equation}\label{eq:state-galerkin-convergence}
 \begin{aligned}
 &~~\sup_{t\le T}\|Z_t^N-Z_t\|_{\mathbb H}^2
   +\int_0^T\|v_t^N-v_t\|_{H^1}^2\,\mathrm dt\\
 &\le C_{M,T}\left(\|(I-\Pi_N)z_0\|_{\mathbb H}^2
   +\int_0^T\|(I-P_N)P_0g(u_t)\|_{H^{-1}}^2\,\mathrm dt\right)
   \longrightarrow0.
 \end{aligned}
\end{equation}
Here $g(u)\in L^2(0,T;L^2)$, so dominated convergence removes the projection
remainder. The same argument applies to a flow restarted at any fixed time.

For the maximal estimate, apply the L\'{e}vy It\^o formula to $V$ and discard the positive dissipation. Up to a bounded stopping time $T_*$,
\begin{equation*}
 V(Z_{t\wedge T_*})\le V(z_0)+C(t\wedge T_*)+M_{t\wedge T_*},
\end{equation*}
where the compensated martingale has jump integrand
$H(z,y)=D_vV(z)[Qy]+\|Qy\|^2/2$. For $p\ge2$, Assumption~\ref{ass:levy} gives
\begin{equation}\label{eq:maximal-jump-increment}
 \int|H(z,y)|^2\nu_L(\mathrm dy)\le C(1+V(z)),\qquad
 \int|H(z,y)|^p\nu_L(\mathrm dy)\le C_p(1+V(z)^{p/2}).
\end{equation}
Writing $V_* =\sup_{s\le T_*}V(Z_s)$, the Burkholder--Davis--Gundy and Young inequalities imply, for any $\epsilon>0$,
\begin{align*}
 \mathbb E\sup_{s\le T_*}|M_s|^p
 &\le C_p\mathbb E\left[\left(\int_0^{T_*}(1+V(Z_s))\,\mathrm ds\right)^{p/2}
             +\int_0^{T_*}(1+V(Z_s)^{p/2})\,\mathrm ds\right]\\
 &\le C_p\mathbb E\left[T_*^{p/2}(1+V_*)^{p/2}
                      +T_*(1+V_*^{p/2})\right]\\
 &\le\epsilon\mathbb E V_*^p+C_{p,\epsilon}\mathbb E(1+T_*)^p.
\end{align*}
Choosing $\epsilon$ small enough in the energy inequality and removing the martingale localization yields
\begin{equation}\label{eq:maximal-random-horizon}
 \mathbb E\sup_{s\le T_*}V(Z_s)^p
 \le C_p\left(1+V(z_0)^p+\mathbb E(1+T_*)^p\right).
\end{equation}
Taking $T_*=T$ proves \eqref{eq:finite-horizon-sup} for $p\ge2$; the case $p=1$ follows by Cauchy--Schwarz. This gives
\begin{equation*}
 \sup_{V(z_0)\le R}\mathbb P\left(\sup_{s\le T}V(Z_s)>M\right)
 \le\frac{C_{T,p}(1+R^p)}{(1+M)^p}\longrightarrow0.
\end{equation*}
The localized difference estimate gives \eqref{eq:continuous-dependence} and the Feller property. The measurable solution map, uniqueness, and independent increments of $L$ give the Markov property. For a finite stopping time $\sigma$, approximate it from above by $\sigma_k=2^{-k}\lceil2^k\sigma\rceil$. For $f\in C_b(\mathbb H)$, the Markov identity at $\sigma_k$ passes to $\sigma$ by right continuity of $Z$ and continuity of $P_tf$; a monotone-class argument extends it to bounded Borel $f$. Thus $Z$ is strong Markov. Right continuity at zero gives stochastic continuity.
\end{proof}

Let $\{P_t\}_{t\ge0}$ be the Markov semigroup generated by \eqref{eq:first-order}.  We use the following two notions.

\begin{definition}\label{def:ergodic-notions}
	The semigroup $\{P_t\}$ has the \emph{e-property} if, for every bounded Lipschitz function
	$\varphi$ and every $z\in\mathbb{H}$,
	\begin{equation*}
	\lim_{z'\to z}\sup_{t\ge0}|P_t\varphi(z')-P_t\varphi(z)|=0.
	\end{equation*}
	It is \emph{weakly irreducible at the origin} if, for every $R,r>0$, there exists
	$T=T(R,r)>0$ such that
	\begin{equation*}
	\inf_{V(z)\le R}P_T(z,B_{\mathbb{H}}(0,r))>0.
	\end{equation*}
\end{definition}

\section{Lyapunov estimates and the existence of invariant measures}\label{sec:lyapunov}
Set
\begin{equation}\label{eq:EL}
 \mathcal{E}_L=\int_{\mathbb{R}^m\setminus\{0\}}\|Qy\|^2\,\nu_L(\,\mathrm{d} y)
 =B_Q\bar s<\infty.
\end{equation}
The special form of $V$ is particularly convenient for jumps because a jump changes only the velocity component and
\begin{equation}\label{eq:jump-remainder}
 V(u,v+Qy)-V(u,v)-D_vV(u,v)[Qy]
 =\frac12\|Qy\|^2.
\end{equation}
We use the dissipation functional
\begin{equation*}
 \mathfrak{F}(u,v)=\|A^{1/2}v\|^2+\|A^{1/2}u\|^2
 +\int_{\mathbb{T}^3}G(u)\,\mathrm{d}x.
\end{equation*}
Set $\mathfrak F(u,v)=+\infty$ when $v\notin H^1$.

\begin{lemma}\label{lem:basic-energy}%
There exist constants $c_1,c_2,C>0$ such that for every solution of \eqref{eq:first-order},
\begin{equation}\label{eq:mean-energy}
 \mathbb{E}[V(Z_t)]+c_1\mathbb{E}\left[\int_0^t e^{-c_2(t-s)}V(Z_s){}\mathrm{d} s\right]
 \le C+e^{-c_2t}V(z_0),
 \qquad t\ge0.
\end{equation}
In particular,
\begin{equation}\label{eq:uniform-second}
 \sup_{t\ge0}\mathbb{E}[\|Z_t\|_{\mathbb{H}}^2]
 \le C(1+V(z_0)).
\end{equation}
\end{lemma}

\begin{proof}
Applying It\^o's formula to the spectral Galerkin approximations and passing to the limit give the corresponding energy identity. The It\^o formula for square-integrable L\'evy processes is standard; see, for instance, \cite{Applebaum2009}. The drift terms are obtained by the usual deterministic modified-energy calculation:
\begin{align}\label{eq:D-controls-3.6}
 \begin{split}
 &\,\mathrm{d} V(Z_t)
 +\|A^{1/2}v_t\|^2{}\mathrm{d} t
 +\delta\|A^{1/2}u_t\|^2{}\mathrm{d} t
 +\delta(P_0g(u_t),u_t){}\mathrm{d} t\\
 &=\delta\|v_t\|^2{}\mathrm{d} t
 -\delta(A^{1/2}u_t,A^{1/2}v_t){}\mathrm{d} t
 +\,\mathrm{d} M_t+\frac12\mathcal{E}_L{}\mathrm{d} t,
 \end{split}
\end{align}
where $M$ is a square-integrable martingale.  
The compensator is exactly \eqref{eq:EL} because of \eqref{eq:jump-remainder}.  
By Young's inequality we have
\begin{equation*}
 \delta(A^{1/2}u,A^{1/2}v)
 \le \frac12\|A^{1/2}v\|^2+\frac{\delta^2}{2}\|A^{1/2}u\|^2.
\end{equation*}
Since $\delta<\lambda_1/4$, using Poincar\'e's inequality yields
\begin{equation*}
 \delta\|v\|^2\le\frac14\|A^{1/2}v\|^2.
\end{equation*}
By the preceding Young estimate, the three dissipative coefficients are bounded below by $1/4$, $\delta(1-\delta/2)$, and $\delta\iota$, respectively.  Set
\begin{equation*}
 a_0=\min\left\{\frac14,\delta\left(1-\frac\delta2\right),\delta\iota\right\}>0.
\end{equation*}
By Poincar\'e's inequality,
\eqref{eq:G-coercive}, and the cross-term estimate obtained before \eqref{eq:V-equivalence}, there exists a constant $\hat{c}>0$, depending only on the structural coefficients, such that
\begin{equation}\label{eq:D-controls-V}
 \mathfrak{F}(z)\ge \hat{c} V(z),\qquad z\in\mathbb{H}.
\end{equation}
Indeed, the velocity term is bounded via  $\|v\|^2\le\lambda_1^{-1}\|A^{1/2}v\|^2$, the $u$-energy is already included in
$\mathfrak{F}$, and $\int G(u)$ controls $\|u\|_{L^4}^4$; the cross term $\delta(u,v)$ is absorbed by the two quadratic terms.  Writing
$a_0\mathfrak{F}=(a_0/2)\mathfrak{F}+(a_0/2)\mathfrak{F}$, and using
 \eqref{eq:D-controls-3.6}  and \eqref{eq:D-controls-V}, we obtain
\begin{equation}\label{Diffineq}
 \frac{\,\mathrm{d}}{\,\mathrm{d} t}\mathbb{E}[V(Z_t)]
 +\frac{a_0\hat{c}}2\mathbb{E}[V(Z_t)]
 +\frac{a_0}2\mathbb{E}[\mathfrak{F}(Z_t)]
 \le \frac12\mathcal{E}_L.
\end{equation}
Take $c_1=c_2=a_0\hat c/2$. Multiplying \eqref{Diffineq} by $e^{c_2t}$, integrating, and using \eqref{eq:D-controls-V} yields
\begin{equation*}
 \mathbb{E}[V(Z_t)]+c_1\int_0^t e^{-c_2(t-s)}\mathbb{E}[V(Z_s)]\,\mathrm ds
 \le e^{-c_2t}V(z_0)+\frac{\mathcal{E}_L}{2c_2},
\end{equation*}
which proves \eqref{eq:mean-energy}.
Moreover,
\eqref{eq:uniform-second} follows from \eqref{eq:V-equivalence}. %
\end{proof}

\begin{lemma}\label{lem:poly-moments}%
For every integer $p\ge1$, there exists $C_p>0$ such that
\begin{equation}\label{eq:poly-moments}
 \sup_{t\ge0}\mathbb{E}[(1+V(Z_t))^p]
 \le C_p\bigl(1+V(z_0)^p\bigr).
\end{equation}
Consequently,
\begin{equation}\label{eq:poly-time-average}
 \sup_{T\ge1}\frac1T\int_0^T\mathbb{E}[(1+V(Z_s))^p]{}\mathrm{d} s
 \le C_p\bigl(1+V(z_0)^p\bigr).
\end{equation}
\end{lemma}
\begin{proof} 
	We apply the L\'evy-It\^o formula to $(1+V)^p$ at the Galerkin level.  Since the jumps affect only $v$ and $V$ is quadratic in $v$, using Taylor's formula and \eqref{eq:jump-remainder} yields
	\begin{equation*}
	\begin{split}
	&(1+V(u,v+Qy))^p-(1+V(u,v))^p\\
	&\quad-p(1+V(u,v))^{p-1}D_vV(u,v)[Qy]\\
	&\le C_p(1+V(u,v))^{p-1}\|Qy\|^2+C_p\|Qy\|^{2p}.
	\end{split}
	\end{equation*}
	Assumption \ref{ass:levy} implies
	$\int(\|Qy\|^2+\|Qy\|^{2p})\nu_L(\,\mathrm{d} y)<\infty$.
	Let $x=1+V$. The drift estimate from \eqref{eq:D-controls-3.6}, the jump
	remainder above, and \eqref{eq:D-controls-V} yield at the Galerkin level
	\begin{align*}
	\mathcal Lx^p
	&\le -pa_0x^{p-1}\mathfrak F+C_px^{p-1}+C_p\\
	&\le -pa_0\hat c\,x^{p-1}(x-1)+C_px^{p-1}+C_p
	\le -c_px^p+C'_p.
	\end{align*}
	In the last step, the term of order $x^{p-1}$ is absorbed outside a
	fixed bounded interval. Gronwall's inequality at the Galerkin level, removal of localization, and Fatou's lemma yield
\begin{equation*}
\mathbb E[(1+V(Z_t))^p]
\le e^{-c_pt}(1+V(z_0))^p+\frac{C'_p}{c_p}(1-e^{-c_pt}).
\end{equation*}
This proves \eqref{eq:poly-moments} and \eqref{eq:poly-time-average}.
\end{proof}

\begin{lemma}\label{lem:compactness}%
Let $\mathbb{S}(t)$ be the linear strongly damped wave semigroup associated with
$\partial_t (u,v)=(v,-Au-Av)$.  Write 
\begin{equation*}
 Z_t=\hat{Z}_t+\widetilde{Z}_t,
 \qquad \hat{Z}_t=\mathbb{S}(t)z_0,
\end{equation*}
where $\widetilde{Z}_0=0$.  Then there exist $c,C>0$ such that
\begin{equation}\label{eq:lin-decay}
 \|\hat{Z}_t\|_{\mathbb{H}}\le Ce^{-ct}\|z_0\|_{\mathbb{H}},
 \qquad t\ge0.
\end{equation}
Furthermore, for every $R>0$,
\begin{equation}\label{eq:rem-compact-average}
 \sup_{T\ge1}\frac{1}{T}\mathbb{E}\left[\int_0^T
 \bigl(\|\widetilde{Z}_s^{(1)}\|_{H^2}^2+\|\widetilde{Z}_s^{(2)}\|_{H^1}^2\bigr){}\mathrm{d} s\right]
 \le C_{R,m,Q},
 \qquad V(z_0)\le R.
\end{equation}
The operator $Q=Q_m$ is fixed. Consequently, the measures
\begin{equation*}
 \mu_T=\frac1T\int_0^TP_s^*\delta_{z_0}{}\mathrm{d} s
\end{equation*}
form a tight family as $T\to\infty$ in $\mathbb{H}$.
\end{lemma}
\begin{proof}
	For each eigenvalue $\lambda_k$, the two roots of
	$r^2+\lambda_kr+\lambda_k=0$ have real parts bounded above by a negative constant independent of $k$. Hence, the standard modified energy estimate for the linear system yields \eqref{eq:lin-decay}.
	
	The remainder $\widetilde{Z}=Z-\hat{Z}$ has zero initial data and can be decomposed as
	\begin{equation*}
\widetilde{Z}_t=Z_t^{g}+Z_t^{L},
	\qquad
	Z_t^{g}=\int_0^t\mathbb{S}(t-s)(0,-P_0g(u_s)){}\mathrm{d} s,
	\qquad
	Z_t^{L}=\int_0^t\mathbb{S}(t-s)\mathbf{Q}{}\mathrm{d} L_s.
	\end{equation*}
	Since $u_t\in H^1\hookrightarrow L^6$, \eqref{eq:g-growth} implies
	$\|P_0g(u_t)\|_{L^2}\le \|g(u_t)\|_{L^2}\le C(1+\|u_t\|_{H^1}^3)$. It follows from Lemma \ref{lem:poly-moments} that the deterministic forcing belongs to
	$L^2_{\mathrm{loc}}(0,\infty;H_0)$ in mean square. For the deterministic problem with zero initial data,
	\begin{equation*}
	\partial_t x_1=x_2,\qquad
	\partial_t x_2=-Ax_1-Ax_2+f,
	\end{equation*}
	an eigenmode calculation gives
	\begin{equation}\label{eq:linear-max-reg}
	\int_0^T\bigl(\|x_1(s)\|_{H^2}^2+\|x_2(s)\|_{H^1}^2\bigr){}\mathrm{d} s
	\le C\int_0^T\|f(s)\|^2{}\mathrm{d} s.
	\end{equation}
	Indeed, after diagonalizing $A$, each coefficient solves
	\begin{equation}\label{eq:linear-mode-forcing}
	x_k''+\lambda_kx_k'+\lambda_kx_k=f_k,
	\qquad x_k(0)=x_k'(0)=0.
	\end{equation}
	Multiplying \eqref{eq:linear-mode-forcing} by $\lambda_kx_k'+\varepsilon\lambda_kx_k$, where
	$0<\varepsilon<\varepsilon_0(\lambda_1)$ is fixed independently of $k$, and collecting the total derivative terms, we obtain
	\begin{align}\label{eq:linear-mode-energy}
	\begin{split}
	&~~\frac{\mathrm d}{\mathrm dt}\left[
	\frac{\lambda_k}{2}|x_k'|^2
	+\varepsilon\lambda_k x_kx_k'
	+\frac{(1+\varepsilon)\lambda_k^2}{2}|x_k|^2\right]
	+(\lambda_k^2-\varepsilon\lambda_k)|x_k'|^2
	+\varepsilon\lambda_k^2|x_k|^2  \\
	&=\lambda_k f_kx_k'+\varepsilon\lambda_k f_kx_k.
	\end{split}
	\end{align}
	For this choice of $\varepsilon$, the quantity in brackets is nonnegative and uniformly equivalent to
	$\lambda_k|x_k'|^2+\lambda_k^2|x_k|^2$. Moreover, Young's inequality gives, with constants independent of
	$k$,
	\begin{equation}\label{eq:linear-mode-young}
	\lambda_k|f_kx_k'|+\varepsilon\lambda_k|f_kx_k|
	\le \frac14(\lambda_k^2|x_k'|^2+\varepsilon\lambda_k^2|x_k|^2)
	+C|f_k|^2.
	\end{equation}
	Integrating \eqref{eq:linear-mode-energy} in $t$ from $0$ to $T$,  using
	$\lambda_k\ge\lambda_1$, \eqref{eq:linear-mode-young}, and summing over $k$, we obtain
	\eqref{eq:linear-max-reg}.  In particular, the constant in this estimate is uniform with respect to the eigenmodes.
	
	We treat the stochastic convolution separately. Since $Q$ has finite spectral range,
	$Z_t^L$ takes values in the finite-dimensional space $H_m\times H_m$.  The linear semigroup is
	exponentially stable on this space, and the L\'evy isometry gives
	\begin{equation*}
	\sup_{t\ge0}\mathbb{E}[\|Z_t^L\|_{H^2\times H^1}^2]
	\le C_{m,Q}\int_{\mathbb{R}^m\setminus\{0\}}|y|^2\nu_L(\,\mathrm{d} y)<\infty.
	\end{equation*}
	Thus, no stochastic maximal regularity estimate for a distributional-in-time forcing term is required. Applying \eqref{eq:linear-max-reg} only to $Z^g$, using
	$\|P_0g(u)\|^2\le C(1+V(Z)^3)$, and combining this with the preceding finite-dimensional estimate, we obtain
	\begin{equation*}
	\frac1T\mathbb{E}\left[\int_0^T\bigl(\|A\widetilde{Z}_s^{(1)}\|^2+\|A^{1/2}\widetilde{Z}_s^{(2)}\|^2\bigr){}\mathrm{d} s\right]
	\le C\left(1+\frac1T\int_0^T\mathbb{E}[V(Z_s)^3]{}\mathrm{d} s\right)+C_{m,Q}.
	\end{equation*}
	Lemma \ref{lem:poly-moments} with $p=3$ now yields \eqref{eq:rem-compact-average}.  The decomposition is needed because the rough part $\mathbb{S}(t)z_0$ need not enter $H^2\times H^1$ at a fixed positive time, although it decays exponentially in the energy space.
	
	Let $B_M$ be a bounded ball in $H^2\times H^1$, and let $K_M$ denote its closure in
	$\mathbb{H}=H^1\times H_0$.  Since the embedding $H^2\times H^1\hookrightarrow\mathbb{H}$ is compact,
	$K_M$ is compact in $\mathbb{H}$.  By \eqref{eq:rem-compact-average} and Chebyshev's inequality, for every
	$\varepsilon>0$, one can choose $M$ sufficiently large such that
	\begin{equation}\label{eq:rem-tight-average}
	\sup_{T\ge1}\frac1T\int_0^T
	\mathbb{P}(\widetilde{Z}_s\notin K_M){}\mathrm{d} s<\varepsilon.
	\end{equation}
Set
\begin{equation*}
 \widehat K=\overline{\{\mathbb S(s)z_0:s\ge0\}}^{\,\mathbb H}.
\end{equation*}
This set is compact because $s\mapsto\mathbb S(s)z_0$ is continuous and tends to zero. Hence $K_M+\widehat K$ is compact and, by \eqref{eq:rem-tight-average},
\begin{equation*}
 \sup_{T\ge1}\mu_T((K_M+\widehat K)^c)
 \le\sup_{T\ge1}\frac1T\int_0^T
 \mathbb P(\widetilde Z_s\notin K_M)\,\mathrm ds<\varepsilon.
\end{equation*}
This proves tightness.
\end{proof}

\begin{proposition}\label{prop:existence}
Under Assumptions \ref{ass:g}--\ref{ass:levy}, $\{P_t\}$ has at least one invariant measure $\mu$. Moreover, every invariant measure satisfies
\begin{equation}\label{eq:mu-V}
 \int_{\mathbb{H}}V(z)\,\mu(\,\mathrm{d} z)<\infty.
\end{equation}
\end{proposition}

\begin{proof}
The Feller property and tightness follow from Proposition~\ref{prop:wellposed} and Lemma~\ref{lem:compactness}. The Krylov--Bogoliubov argument therefore gives an invariant measure; see \cite{DaPratoZabczyk1996}.

To prove the moment bound for an arbitrary invariant measure $\mu$, set $V_N=V\wedge N$. By \eqref{eq:mean-energy} and Jensen's inequality,
\begin{equation*}
 \int V_N\,\mathrm d\mu
 =\int P_tV_N\,\mathrm d\mu
 \le\int\bigl(C+e^{-c_2t}V(z)\bigr)\wedge N\,\mu(\mathrm dz).
\end{equation*}
Since $V(z)<\infty$ for every $z\in\mathbb H$, dominated convergence as $t\to\infty$ gives $\int V_N\,\mathrm d\mu\le C$. Monotone convergence as $N\to\infty$ proves \eqref{eq:mu-V}.
\end{proof}

\section{Stopping times and conditional Malliavin calculus}\label{sec:stopping}

\subsection{Stopping blocks and centered clock localization}
For a fixed subordinator path $\ell$ and
$\varepsilon\in(0,1]$, we set
\begin{equation}\label{eq:clock-regularization}
 \ell_t^\varepsilon=\frac1\varepsilon\int_t^{t+\varepsilon}\ell_r\mathrm{d} r+\varepsilon t.
\end{equation}
Then $\ell^\varepsilon$ is continuous and strictly increasing and
$\ell_t^\varepsilon\rightarrow\ell_t$ for every $t$.  Since $\ell_0^\varepsilon$ need not vanish, we define the centered continuous approximation
\begin{equation*}
 L_t^\varepsilon=W_{\ell_t^\varepsilon}-W_{\ell_0^\varepsilon},\qquad L_0^\varepsilon=0,
\end{equation*}
and let $Z^\varepsilon$ denote the solution of  \eqref{eq:first-order} with $L$ replaced by $L^\varepsilon$.  By It\^o's formula, there exist constants
$c_V,c_D,C_V>0$, independent of $\varepsilon$, such that
\begin{equation}\label{eq:approx-energy-structure}
 \,\mathrm{d} V(Z_t^\varepsilon)+c_VV(Z_t^\varepsilon)\mathrm{d} t
 +c_D\mathfrak{F}(Z_t^\varepsilon)\mathrm{d} t
 \le C_V\mathrm{d} t+\,\mathrm{d}M_t^\varepsilon
 +C_VB_Q\mathrm{d}\ell_t^\varepsilon,
\end{equation}
where
\begin{equation*}
 M_t^\varepsilon
 =\int_0^t(v_s^\varepsilon+\delta u_s^\varepsilon,Q\mathrm{d} W_{\ell_s^\varepsilon})
\end{equation*}
and
\begin{equation}\label{eq:approx-energy-bracket}
 \,\mathrm{d}\langle M^\varepsilon\rangle_t
 \le C B_Q\bigl(1+V(Z_t^\varepsilon)\bigr)\mathrm{d}\ell_t^\varepsilon.
\end{equation}

\begin{lemma}\label{lem:clock-stability}
Fix a subordinator path $\ell$ and a Brownian path $W$ such that the quantities below are finite on every compact interval. Set
\begin{equation*}
 \zeta_t=QW_{\ell_t},
 \qquad
 \zeta_t^\varepsilon
 =Q\bigl(W_{\ell_t^\varepsilon}-W_{\ell_0^\varepsilon}\bigr).
\end{equation*}
Then, for every $T>0$ and $r\ge0$,
\begin{equation}\label{eq:zeta-L2-conv}
 \int_0^T\|\zeta_t^\varepsilon-\zeta_t\|_{H^r}^2\mathrm{d} t
 \longrightarrow0.
\end{equation}
Let $Z^\varepsilon=(u^\varepsilon,v^\varepsilon)$ and $Z=(u,v)$ be two solutions of \eqref{eq:first-order} with the same initial condition, and put
$y^\varepsilon=v^\varepsilon-\zeta^\varepsilon$, $y=v-\zeta$.  Then we have
\begin{align}
 &~~\sup_{0\le t\le T}
 \bigl(\|u_t^\varepsilon-u_t\|_{H^1}^2
       +\|y_t^\varepsilon-y_t\|^2\bigr)
 +\int_0^T\|y_t^\varepsilon-y_t\|_{H^1}^2\,\,\mathrm{d} t\le C_T
 \int_0^T\|\zeta_t^\varepsilon-\zeta_t\|_{H^1}^2\,\,\mathrm{d} t.
 \label{eq:shifted-stability}
\end{align}
Moreover, if $t_\varepsilon\to t$ and $\ell_{t_\varepsilon}^\varepsilon\to\ell_t$, then $Z_{t_\varepsilon}^\varepsilon\to Z_t$ in $\mathbb H$ for these fixed paths.
\end{lemma}
\begin{proof}
	First, for \(t\le T\) and \(\varepsilon\le1\), the quantity \(\ell_t^\varepsilon\) is bounded by a pathwise finite constant that depends only on \(\ell_{T+1}\) and \(T\).  Since \(Q\) has finite spectral range, all \(H^r\) norms are equivalent on \(H_m\), and the Brownian path is bounded on the relevant compact operational-time interval. Hence, by dominated convergence, \eqref{eq:zeta-L2-conv} follows.
	
	The shifted variables solve
	\begin{equation*}
	\partial_t u=y+\zeta,
	\qquad
	\partial_t y=-Ay-A\zeta-Au-P_0g(u),
	\end{equation*}
	and the analogous equation with $\zeta^\varepsilon$.  The pathwise energy bound in \eqref{eq:pathwise-galerkin-energy} is uniform in $\varepsilon$ on each compact interval. Thus \eqref{eq:g-growth} and $H^1\hookrightarrow L^6$ give
	\begin{equation*}
	\|P_0(g(u^\varepsilon)-g(u))\|_{L^2}
	\le C_T\|u^\varepsilon-u\|_{H^1}.
	\end{equation*}
	The two forcing terms containing $\zeta^\varepsilon-\zeta$ are bounded by a small multiple of the $H^1$ dissipation of $y^\varepsilon-y$ plus
	$C_T\|\zeta^\varepsilon-\zeta\|_{H^1}^2$.  By Gronwall's inequality we deduce \eqref{eq:shifted-stability}. 
	
	 Finally, $u$ and $y$ are continuous in time, while
	\begin{equation*}
	\zeta_{t_\varepsilon}^\varepsilon
	=Q\bigl(W_{\ell_{t_\varepsilon}^\varepsilon}-W_{\ell_0^\varepsilon}\bigr)
	\longrightarrow QW_{\ell_t}=\zeta_t
	\end{equation*}
	whenever $t_\varepsilon\to t$ and
	$\ell_{t_\varepsilon}^\varepsilon\to\ell_t$.  Together with \eqref{eq:shifted-stability}, this proves the asserted pathwise convergence.
\end{proof}

Fix $c_Q>0$ sufficiently large compared with the constants in
\eqref{eq:approx-energy-structure}--\eqref{eq:approx-energy-bracket}.  We also fix a block height
$\chi>1$.  The value of $\chi$ will be chosen in Section \ref{sec:control} so that one uncontrolled block
produces a genuine deterministic high-mode contraction.  
For $\kappa>0$, define the stopping times \(\tau_n\) by
\begin{equation}\label{eq:stopping-times}
 \tau_0=0,
 \qquad
 \tau_{n+1}=\inf\left\{t\ge\tau_n:
 c_V(t-\tau_n)-c_Q\kappa B_Q(S_t-S_{\tau_n})\ge\chi\right\}.
\end{equation}
The quantity inside the braces has positive continuous drift and jumps only downward. For \(\kappa\) sufficiently small, it drifts to \(+\infty\). We use the closed threshold "\(\ge\chi\)" so that the process stays below \(\chi\) before \(\tau_{n+1}\); since upward motion is continuous, the first passage is a continuous contact with \(\chi\). Therefore,
\begin{equation}\label{eq:clock-block-bound}
 c_V(\tau_{n+1}-\tau_n)-c_Q\kappa B_Q
 (S_{\tau_{n+1}}-S_{\tau_n})=\chi,
\end{equation}
and in particular
\begin{equation*}
 S_{\tau_{n+1}}-S_{\tau_n}
 \le \frac{c_V}{c_Q\kappa B_Q}(\tau_{n+1}-\tau_n).
\end{equation*}

For the regularized clock \eqref{eq:clock-regularization}, we define recursively
\begin{equation}\label{eq:tau-eps}
 \tau_0^\varepsilon=0,\qquad
 \tau_{n+1}^\varepsilon=\inf\left\{t\ge\tau_n^\varepsilon:
 c_V(t-\tau_n^\varepsilon)-c_Q\kappa B_Q
 (\ell_t^\varepsilon-\ell_{\tau_n^\varepsilon}^\varepsilon)\ge\chi\right\}.
\end{equation}

\begin{lemma}\label{lem:clock-hitting-stability}
Choose $\kappa>0$ so small that
\begin{equation}\label{eq:positive-clock-drift}
 c_V-c_Q\kappa B_Q\,\mathbb{E}[S_1]>0.
\end{equation}
Then, for almost every subordinator path and every fixed $n\ge0$, as $\varepsilon\rightarrow0$,
\begin{equation}\label{eq:tau-eps-conv}
 \tau_n^\varepsilon\longrightarrow\tau_n
 \text{~~~and~~~}
 \ell_{\tau_n^\varepsilon}^\varepsilon\longrightarrow S_{\tau_n}.
\end{equation}
\end{lemma}
\begin{proof}
Put $d_\kappa=c_Q\kappa B_Q$. The strong law and
\eqref{eq:positive-clock-drift} imply $c_Vt-d_\kappa S_t\to\infty$.
For sufficiently small $\varepsilon$, $c_V-d_\kappa(\bar s+\varepsilon)>0$; since $\ell_t^\varepsilon/t\to\bar s+\varepsilon$, the regularized hitting times are finite.
Almost surely, no jump time $r$ satisfies
\begin{equation}\label{eq:no-prejump-contact}
 c_V(r-a)-d_\kappa(S_{r-}-S_a)=\chi,\qquad a=\tau_n.
\end{equation}
To justify this assertion, first take $a=0$ and write
$X_{r-}=c_Vr-d_\kappa S_{r-}$. For each deterministic $r>0$,
$S_{r-}=S_r$ almost surely; the infinite L\'evy measure makes $S_r$
atomless \cite[Theorem~27.4]{Sato1999}. Consequently, Fubini's theorem
gives $\mathbb E\int_0^T\mathbf1_{\{X_{r-}=\chi\}}\,\mathrm dr=0$.
For each $k\ge1$, the compensation formula
for jumps with size greater than $1/k$ therefore gives
\[
 \mathbb E\sum_{r>0}\mathbf1_{\{r\le T,\,\Delta S_r>1/k,\,X_{r-}=\chi\}}
 =\nu_S((1/k,\infty))\,
   \mathbb E\int_0^T\mathbf1_{\{X_{r-}=\chi\}}\,\mathrm dr=0.
\]
Taking the union over $k$ and integer $T$ excludes every jump at such a
prejump contact. The strong Markov property gives the same conclusion for
each $a=\tau_n$; a countable intersection covers all $n$.

Suppose $a_\varepsilon=\tau_n^\varepsilon\to a=\tau_n$ and
$\ell_{a_\varepsilon}^\varepsilon\to S_a$. Write
\[
 X(t)=c_V(t-a)-d_\kappa(S_t-S_a),\qquad
 X_\varepsilon(t)=c_V(t-a_\varepsilon)
             -d_\kappa(\ell_t^\varepsilon-\ell_{a_\varepsilon}^\varepsilon).
\]
At $\tau=\tau_{n+1}$, upward passage is continuous, so $X(\tau)=\chi$ and
$\Delta S_\tau=0$. Exclusion of \eqref{eq:no-prejump-contact} implies
\begin{equation}\label{eq:strict-prehit-gap}
 \sup_{a\le t\le\tau-d}X(t)<\chi\qquad(d>0,\ a<\tau-d).
\end{equation}
Otherwise a maximizing sequence would produce an earlier contact or a left limit
equal to $\chi$ at a jump. Monotonicity of $S$ gives
\begin{equation}\label{eq:clock-sandwich}
 S_t\le\ell_t^\varepsilon\le S_{t+\varepsilon}+\varepsilon t,
\end{equation}
and hence
\[
 X_\varepsilon(t)\le X(t)+c_V|a_\varepsilon-a|
                  +d_\kappa|\ell_{a_\varepsilon}^\varepsilon-S_a|.
\]
Thus $\liminf\tau_{n+1}^\varepsilon\ge\tau$. After $\tau$, the process
$c_Vh-d_\kappa(S_{\tau+h}-S_\tau)$ has bounded variation and positive drift
coefficient. Zero is regular for $(0,\infty)$, so in every $(\tau,\tau+d)$
there is a continuity point $t_+$ with $X(t_+)>\chi$.
By \eqref{eq:clock-sandwich}, $X_\varepsilon(t_+)\to X(t_+)$; hence
$\limsup\tau_{n+1}^\varepsilon\le\tau$.
Finally, applying \eqref{eq:clock-sandwich} at $\tau_{n+1}^\varepsilon$ and using
continuity at $\tau$ proves convergence of the clock values. Induction starts at zero.
\end{proof}

For $s\in[\tau_n,\tau_{n+1}]$, we set
\begin{equation*}
 \omega_n(s)=
 \exp\left\{-c_V(\tau_{n+1}-s)
 +c_Q\kappa B_Q(S_{\tau_{n+1}}-S_s)\right\}.
\end{equation*}
By the definition of $\tau_{n+1}$, we find that $0<\omega_n(s)\le1$ and $\omega_n(\tau_n)=e^{-\chi}$.

\begin{lemma}\label{lem:stopped-energy}
There exist constants \(c_*>0\) and \(\mathscr K>0\), depending only on the modified-energy and martingale-bracket coefficients and independent of \(m,\chi,\kappa\). Choose
$p_*>\max\{100,\mathscr K/c_Q\}$ and then \(\kappa_0>0\) such that, for every \(\kappa\in(0,\kappa_0]\), the following holds.
There are
\(C_{{\rm wt},\kappa}<\infty\), \(\mathscr S_{\kappa,\chi}<\infty\), and
\(\widehat C_{\kappa,\chi}<\infty\) such that, for every \(n\ge0\),
\begin{align}
 &\mathbb{E}[e^{p_*c_V\tau_1}]\le \widehat C_{\kappa,\chi},\label{eq:tau-exp}\\
 &\mathbb{E}\left[e^{\eta_* V(Z_{\tau_{n+1}})}\mid\mathcal{F}_{\tau_n}\right]
 \le \mathscr S_{\kappa,\chi}
 \exp\!\left(\eta_*e^{-\chi}V(Z_{\tau_n})\right),
 \label{eq:stopped-recursion}\\
 &\mathbb{E}\Bigg[
 \exp\Bigg\{
 \begin{aligned}
 &\qquad \kappa V(Z_{\tau_{n+1}})-\kappa e^{-\chi}V(Z_{\tau_n})\\
 &+c_*\kappa\int_{\tau_n}^{\tau_{n+1}}\omega_n(s)\mathfrak{F}(Z_s)\,\,\mathrm{d} s
 -\mathscr K\kappa B_QY_n
 \end{aligned}
 \Bigg\}\,\Bigg|\,\mathcal{F}_{\tau_n}\Bigg]
 \le C_{{\rm wt},\kappa},
 \label{eq:weighted-stopped-energy}
\end{align}
where $\eta_*=\frac{\kappa}{2}$ and
\begin{equation}\label{eq:Yn}
 Y_n=S_{\tau_{n+1}}-S_{\tau_n}.
\end{equation}
Moreover, for any $p\ge1$,
\begin{equation}\label{eq:stopped-sup}
 \mathbb{E}_{z_0}\left[\sup_{0\le s\le\tau_1}V(Z_s)^p\right]
 \le \widetilde C_{p,\kappa,\chi}(1+V(z_0)^p),
\end{equation}
here the constants $\mathscr S_{\kappa,\chi}$ and $\widetilde C_{p,\kappa,\chi}$ may depend on $\chi$.  In particular, the same estimates hold after the shift $\tau_n$ with initial state $Z_{\tau_n}$.
\end{lemma}

\begin{proof}
Decrease $\kappa_0$, if necessary, so that \eqref{eq:positive-clock-drift} holds for every $\kappa\le\kappa_0$; the cumulant choice at
the end of the proof further decreases the same threshold.  Let $\tau_n^\varepsilon$ be the
regularized first-passage blocks from
\eqref{eq:tau-eps}.  Lemma \ref{lem:clock-hitting-stability} gives \eqref{eq:tau-eps-conv}; in
particular, the passage to the discontinuous clock uses convergence of the first-passage endpoints,
not a monotonicity assertion in $\varepsilon$.
Fix the subordinator path $\ell$.  The regularized clock $\ell^\varepsilon$ uses future values of $\ell$, so $\tau_n^\varepsilon$ need not be a stopping time for the original joint filtration.  It is, however, a deterministic physical time once $\ell$ is frozen.  Put
\begin{equation*}
 a_n^\varepsilon=\ell_{\tau_n^\varepsilon}^\varepsilon-\ell_0^\varepsilon,
 \qquad
 W_r^\varepsilon=W_{\ell_0^\varepsilon+r}-W_{\ell_0^\varepsilon},
 \qquad
 \mathcal{F}_r^{W,\varepsilon}=\sigma\{W_q^\varepsilon:0\le q\le r\}.
\end{equation*}
For each fixed clock, $W^\varepsilon$ is a standard Brownian motion and
$L_t^\varepsilon=W_{\ell_t^\varepsilon-\ell_0^\varepsilon}^\varepsilon$.  On
$[\tau_n^\varepsilon,\tau_{n+1}^\varepsilon]$ apply
\eqref{eq:approx-energy-structure} conditionally on the fixed clock.  With
\begin{equation*}
 \omega_n^\varepsilon(s)=
 \exp\{-c_V(\tau_{n+1}^\varepsilon-s)
 +c_Q\kappa B_Q(\ell_{\tau_{n+1}^\varepsilon}^\varepsilon-
 \ell_s^\varepsilon)\},
\end{equation*}
write $a=\tau_n^\varepsilon$ and $T=\tau_{n+1}^\varepsilon$.  Then
$\omega_n^\varepsilon(T)=1$, $\omega_n^\varepsilon(a)=e^{-\chi}$ and
\begin{equation*}
 \mathrm d\omega_n^\varepsilon
 =c_V\omega_n^\varepsilon\,\mathrm ds
 -c_Q\kappa B_Q\omega_n^\varepsilon\,\mathrm d\ell_s^\varepsilon.
\end{equation*}
Consequently,
\begin{equation*}
 c_V\int_a^T\omega_n^\varepsilon(s)\,\mathrm ds
 =1-e^{-\chi}
 +c_Q\kappa B_Q\int_a^T\omega_n^\varepsilon(s)\,\mathrm d\ell_s^\varepsilon
 \le 1+c_Q\kappa B_Q(\ell_T^\varepsilon-\ell_a^\varepsilon).
\end{equation*}
Let $N_T^\varepsilon=\kappa\int_a^T\omega_n^\varepsilon\,\mathrm dM_s^\varepsilon$.
Integration by parts in \eqref{eq:approx-energy-structure} gives
\begin{align}
 &\kappa V(Z_T^\varepsilon)-\kappa e^{-\chi}V(Z_a^\varepsilon)
 +c_D\kappa\int_a^T\omega_n^\varepsilon\mathfrak{F}(Z_s^\varepsilon)\,\mathrm ds
 +c_Q\kappa^2B_Q\int_a^T\omega_n^\varepsilon
 V(Z_s^\varepsilon)\,\mathrm d\ell_s^\varepsilon\notag\\
 &\qquad\le
 C\kappa\int_a^T\omega_n^\varepsilon\,\mathrm ds
 +N_T^\varepsilon
 +C\kappa B_Q\int_a^T\omega_n^\varepsilon\,\mathrm d\ell_s^\varepsilon .
 \label{eq:weighted-energy-before-bracket}
\end{align}
Since $0<\omega_n^\varepsilon\le1$, \eqref{eq:approx-energy-bracket} yields
\begin{align*}
 \langle N^\varepsilon\rangle_T
 &\le C\kappa^2B_Q\int_a^T
 (\omega_n^\varepsilon)^2(1+V(Z_s^\varepsilon))\,\mathrm d\ell_s^\varepsilon\notag\\
 &\le C\kappa^2B_Q\int_a^T\omega_n^\varepsilon\,\mathrm d\ell_s^\varepsilon
 +C\kappa^2B_Q\int_a^T\omega_n^\varepsilon
 V(Z_s^\varepsilon)\,\mathrm d\ell_s^\varepsilon .
\end{align*}
Choose $c_Q$ larger than twice the last structural constant.  Writing
$N_T^\varepsilon=(N_T^\varepsilon-\frac12\langle N^\varepsilon\rangle_T)
+\frac12\langle N^\varepsilon\rangle_T$, the $V\,\mathrm d\ell^\varepsilon$ term in the
second summand is absorbed by at most one half of the corresponding positive term in
\eqref{eq:weighted-energy-before-bracket}.  Moreover,
\begin{equation*}
 C\kappa\int_a^T\omega_n^\varepsilon\,\mathrm ds
 \le C\kappa
 +C\kappa^2B_Q\int_a^T\omega_n^\varepsilon\,\mathrm d\ell_s^\varepsilon
\end{equation*}
by the preceding weight identity.  Taking $\kappa_0\le1$, using
$\omega_n^\varepsilon\le1$, and retaining half of the deterministic dissipation, all remaining
clock terms are bounded by
$\mathscr K\kappa B_Q(\ell_T^\varepsilon-\ell_a^\varepsilon)$.  Consequently,
\begin{align*}
 &\kappa V(Z_T^\varepsilon)-\kappa e^{-\chi}V(Z_a^\varepsilon)
 +c_*\kappa\int_a^T\omega_n^\varepsilon(s)\mathfrak{F}(Z_s^\varepsilon)\,\mathrm ds
 -\mathscr K\kappa B_Q(\ell_T^\varepsilon-\ell_a^\varepsilon)
 \le C\kappa+N_T^\varepsilon-\frac12\langle N^\varepsilon\rangle_T.
\end{align*}
Thus $c_*$ and $\mathscr K$ depend only on
\eqref{eq:approx-energy-structure}--\eqref{eq:approx-energy-bracket}, not on $m$, $\chi$, or
$\kappa$.  Exponentiating the last display, applying the conditional exponential-martingale
inequality, and removing the localization by Fatou's lemma gives, uniformly in $\varepsilon$ and
for almost every fixed $\ell$,
\begin{align*}
 &\mathbb{E}_W\Bigg[\exp\Bigg\{
 \kappa V(Z_{\tau_{n+1}^\varepsilon}^\varepsilon)
 -\kappa e^{-\chi}V(Z_{\tau_n^\varepsilon}^\varepsilon)
 +c_*\kappa\int_{\tau_n^\varepsilon}^{\tau_{n+1}^\varepsilon}
 \omega_n^\varepsilon(s)\mathfrak{F}(Z_s^\varepsilon)\mathrm{d} s
 -\mathscr K\kappa B_Q(\ell_{\tau_{n+1}^\varepsilon}^\varepsilon-
 \ell_{\tau_n^\varepsilon}^\varepsilon)\Bigg\}
 \;\Big|\;\mathcal{F}_{a_n^\varepsilon}^{W,\varepsilon}\Bigg]\le C_{{\rm wt},\kappa}.
\end{align*}
We pass to the limit on each fixed clock fiber. Put $a_n=\ell_{\tau_n}$.
Lemma~\ref{lem:clock-hitting-stability} gives $a_n^\varepsilon\to a_n$.
By Lemma~\ref{lem:clock-stability}, on each finite interval,
\[
 u^\varepsilon\to u\text{ in }C(H^1),\qquad
 y^\varepsilon\to y\text{ in }L^2(H^1),\qquad
 \zeta^\varepsilon\to\zeta\text{ in }L^2(H^1).
\]
Hence $v^\varepsilon\to v$ in $L^2(H^1)$ and
$G(u^\varepsilon)\to G(u)$ in space-time $L^1$.
The endpoints converge at continuity points of $\ell$, and the weights converge
almost everywhere, with values in $[0,1]$. Consequently,
\[
 \int_{\tau_n^\varepsilon}^{\tau_{n+1}^\varepsilon}
 \omega_n^\varepsilon\mathfrak F(Z_r^\varepsilon)\,\mathrm dr
 \longrightarrow
 \int_{\tau_n}^{\tau_{n+1}}\omega_n\mathfrak F(Z_r)\,\mathrm dr.
\]
The terminal energies converge as well. This uses convergence of the shifted
variables, not uniform convergence of continuous paths to a jump path.

Let $H_n^\varepsilon$ denote the exponential in the conditional estimate above
and $H_n$ its limit. For $a_n>0$, take a bounded nonnegative continuous cylinder
$G=\Phi(W_{r_1},\ldots,W_{r_k})$ with $r_i\le a_n-d$, $d>0$.
Then $G^\varepsilon=\Phi(W_{r_1}^\varepsilon,\ldots,W_{r_k}^\varepsilon)$ is measurable
at time $a_n^\varepsilon$ for small $\varepsilon$, and converges boundedly to $G$.
Fatou's lemma gives
\[
 \mathbb E_W[G H_n]
 \le\liminf_{\varepsilon\to0}\mathbb E_W[G^\varepsilon H_n^\varepsilon]
 \le C_{{\rm wt},\kappa}\mathbb E_W G.
\]
A monotone-class argument and continuity of the Brownian filtration, as
$d\downarrow0$, give $\mathbb E_W[H_n\mid\mathcal F^W_{a_n}]\le C_{{\rm wt},\kappa}$.
For $n=0$, take $G$ constant. Integration over the clock and the strong Markov
property at the limiting stopping time $\tau_n$ prove
\eqref{eq:weighted-stopped-energy}.
For the unweighted recursion, we also need the following block-length moment.  Put
$d_\kappa=c_Q\kappa B_Q$.  Assumption \ref{ass:levy} gives
$\Psi_S(\theta)<\infty$ for $0\le\theta\le\zeta_S$.  With the $p_*$ fixed in the lemma, decrease
$\kappa_0$ so that for some $\theta>p_*$,
$\theta d_\kappa\le\zeta_S$ and
\begin{equation*}
 \theta c_V-\Psi_S(\theta d_\kappa)>p_*c_V
\end{equation*}
for all $\kappa\le\kappa_0$.  Since $\{\tau_1>t\}$ implies
$c_Vt-d_\kappa S_t\le\chi$,
\begin{equation*}
 \mathbb{P}(\tau_1>t)
 \le e^{\theta\chi}\mathbb{E}[e^{-\theta(c_Vt-d_\kappa S_t)}]
 =e^{\theta\chi}\exp\{-t[\theta c_V-\Psi_S(\theta d_\kappa)]\}.
\end{equation*}
This proves \eqref{eq:tau-exp}; its finite prefactor grows at most exponentially in the fixed height
$\chi$.

We now derive the unweighted recursion without discarding the clock penalty.  Let
\begin{equation*}
 H_n=\exp\left\{\kappa V(Z_{\tau_{n+1}})-\kappa e^{-\chi}V(Z_{\tau_n})
 -\mathscr K\kappa B_QY_n\right\}.
\end{equation*}
After dropping only the positive dissipation term, the weighted estimate gives
$\mathbb{E}[H_n\mid\mathcal{F}_{\tau_n}]\le C_{{\rm wt},\kappa}$.  Since
$\eta_*=\kappa/2$, conditional Cauchy--Schwarz yields
\begin{align*}
 &\mathbb{E}\left[e^{\eta_*V(Z_{\tau_{n+1}})-\eta_*e^{-\chi}V(Z_{\tau_n})}
 \mid\mathcal{F}_{\tau_n}\right]\le C_{{\rm wt},\kappa}^{1/2}
 \left(\mathbb{E}[e^{\mathscr K\kappa B_QY_n}\mid\mathcal{F}_{\tau_n}]\right)^{1/2}.
\end{align*}
The second factor is finite.  The stopping identity gives
\begin{equation*}
 Y_n\le \frac{c_V}{c_Q\kappa B_Q}(\tau_{n+1}-\tau_n),
\end{equation*}
and therefore
\begin{equation*}
 e^{\mathscr K\kappa B_QY_n}
 \le \exp\left\{\frac{\mathscr Kc_V}{c_Q}(\tau_{n+1}-\tau_n)\right\}.
\end{equation*}
Because $p_*>\mathscr K/c_Q$, \eqref{eq:tau-exp} controls this exponential moment.  By the strong Markov
property the post-$\tau_n$ block length has the same law as $\tau_1$ and is independent of the
pre-$\tau_n$ history.  Absorbing the resulting finite factor into
$\mathscr S_{\kappa,\chi}$ proves \eqref{eq:stopped-recursion}.  This is exactly where the fixed-height
dependence enters; no uniformity in $\chi$ is asserted.

Apply \eqref{eq:maximal-random-horizon} to $T_* =\tau_1\wedge R$. The moment \eqref{eq:tau-exp} bounds its right-hand side independently of $R$. Letting $R\to\infty$ and using Fatou proves \eqref{eq:stopped-sup} for $p\ge2$; lower positive orders follow by H\"older's inequality.
\end{proof}

We now state the exponent reduction and the stopping-time sums that follow.

\begin{lemma}\label{lem:fixed-block-exponent-reduction}
Fix $\kappa\in(0,\kappa_0]$ and the stopping blocks defined by \eqref{eq:stopping-times}.  For every
$\eta\in(0,\eta_*]$, let $\theta=\eta/\eta_*$.  Then we have
\begin{equation}\label{eq:fixed-block-exp-reduction}
 \mathbb{E}\left[e^{\eta V(Z_{\tau_{n+1}})}\mid\mathcal{F}_{\tau_n}\right]
 \le \mathscr S_{\kappa,\chi}^{\theta}
 e^{\eta e^{-\chi}V(Z_{\tau_n})}.
\end{equation}
Set
$\bar\eta=\frac{1-e^{-1}}{2}\eta_*$. Then, for every \(\eta\in(0,\bar\eta]\), there exists \(\bar C_{\eta,\kappa,\chi}<\infty\) such that
\begin{equation}\label{eq:fixed-block-sum-reduction}
 \mathbb{E}_{z_0}\left[\exp\left\{\eta\sum_{j=1}^nV(Z_{\tau_j})
 -\bar C_{\eta,\kappa,\chi}n\right\}\right]
 \le e^{\eta V(z_0)},\qquad n\ge1.
\end{equation}
\end{lemma}

\begin{proof}
Conditional Jensen gives \eqref{eq:fixed-block-exp-reduction}. Put $q=e^{-\chi}$ and
$r_k=\eta(1+q+\cdots+q^{k-1})\le\eta/(1-q)\le\eta_*/2$. Successive conditioning gives
\[
 \mathbb E e^{\eta\sum_{j=1}^nV(Z_{\tau_j})}
 \le \mathscr S_{\kappa,\chi}^{\sum_{k=1}^n r_k/\eta_*}
       e^{q r_n V(z_0)}.
\]
Since $q/(1-q)<1$, one may take
\begin{equation}\label{eq:sum-constant-explicit}
 \bar C_{\eta,\kappa,\chi}
 =\frac{\eta\log(\mathscr S_{\kappa,\chi}\vee1)}{\eta_*(1-e^{-\chi})}.
\end{equation}
This proves \eqref{eq:fixed-block-sum-reduction}.
\end{proof}

\begin{corollary}\label{cor:stopped-exp-uniform}
Fix $\kappa\in(0,\kappa_0]$ and the corresponding stopping blocks.  For every
$\eta\in(0,\eta_*]$, there exists $C_{\eta,\kappa,\chi}<\infty$ such that
\begin{equation*}
 \sup_{n\ge0}\mathbb{E}_{z_0}[e^{\eta V(Z_{\tau_n})}]
 \le C_{\eta,\kappa,\chi} e^{\eta V(z_0)}.
\end{equation*}
\end{corollary}

\begin{proof}
Set $q=e^{-\chi}\in(0,1)$ and iterate \eqref{eq:fixed-block-exp-reduction}.  If
$M_n=\mathbb{E}[e^{\eta V(Z_{\tau_n})}]$, then
\begin{equation*}
 M_{n+1}\le \mathscr S_{\kappa,\chi}^{\eta/\eta_*}M_n^q.
\end{equation*}
Iteration gives a bound independent of $n$, with a finite constant depending on the fixed block height.
\end{proof}

Fix the state exponent in the control iteration so that
\begin{equation}\label{eq:state-exponent-choice}
 0<\tilde\eta\le\frac14\bar\eta.
\end{equation}
All control constants are computed for this choice. In particular,
$2\theta_0\tilde\eta\le\bar\eta$. Allowing a smaller exponent will be useful in
Proposition~\ref{prop:cubic-realization}; the stopping times are unchanged.

The weighted estimate is useful since the linearized wave equation admits growth bounds with powers of \(u\) strictly below the quartic coercive power. With \(Y_n\) as in \eqref{eq:Yn}, set
\begin{equation}\label{eq:XnXi}
 X_n=\int_{\tau_n}^{\tau_{n+1}}\omega_n(s)^{-8}\,\,\mathrm{d} s,
 \qquad
 \Xi_n=X_n+Y_n.
\end{equation}
By
\begin{equation*}
 \omega_n(s)^{-8}\le \exp\{8c_V(\tau_{n+1}-\tau_n)\},
\end{equation*}
we obtain
\begin{equation*}
 X_n\le(\tau_{n+1}-\tau_n)e^{8c_V(\tau_{n+1}-\tau_n)}.
\end{equation*}
Since $p_*>100$, we may choose $p_{\Xi}\in(10,p_*/8)$.  The estimate
\eqref{eq:tau-exp} then yields
\begin{equation}\label{eq:Xi-moment}
 \mathbb{E}[\Xi_0^{p_{\Xi}}]<\infty.
\end{equation}

\begin{lemma}\label{lem:normalized-block-exp}
Let $r_*=32/9$.  Fix $\Lambda>0$, $p\ge1$, and $\kappa\in(0,\kappa_0]$.  There are
constants $C_{\Lambda,p,\kappa},c_{\Lambda,p,\kappa}>0$ such that for each $r\in[1,r_*]$,
each $\theta\in(0,1]$, and each $n\ge0$, 
\begin{equation}\label{eq:normalized-block-exp}
 \mathbb{E}\left[
 \exp\left\{
   p\Lambda\theta\int_{\tau_n}^{\tau_{n+1}}
       (1+\|u_s\|_{L^4}^{r})\,\,\mathrm{d} s
   -c_{\Lambda,p,\kappa}\theta\Xi_n
 \right\}
 \Bigm|\mathcal{F}_{\tau_n}
 \right]
 \le C_{\Lambda,p,\kappa}^{\,\theta}\exp\{\theta\kappa V(Z_{\tau_n})\}.
\end{equation}
The constants in the preceding estimate are independent of $n$ and $\theta$.  Moreover, for every
$\Lambda_*>0$, they can be chosen uniformly on $0<\Lambda\le\Lambda_*$.
\end{lemma}

\begin{proof}
For $r\le r_*$,
\begin{equation*}
 \|u\|_{L^4}^{r}
 =\bigl(\omega_n\|u\|_{L^4}^{4}\bigr)^{r/4}
  \omega_n^{-r/4}.
\end{equation*}
Thanks to \(r/(4-r)\le8\), Young's inequality gives, for any fixed sufficiently small \(\eta>0\),
\begin{equation*}
 p\Lambda(1+\|u\|_{L^4}^{r})
 \le \eta c_*\kappa\,\omega_n\mathfrak{F}(Z)
      +C_{\Lambda,p,\eta}\omega_n^{-8}.
\end{equation*}
Choose the coefficient of \(X_n\) in \eqref{eq:normalized-block-exp} large enough to absorb both the inverse-weight term and the positive \(Y_n\) term appearing when \eqref{eq:weighted-stopped-energy} is rearranged. Raising the exponential in \eqref{eq:weighted-stopped-energy} to the power \(\theta\eta\) and applying conditional Jensen's inequality gives
\begin{equation*}
 \mathbb{E}\left[
 e^{p\Lambda\theta\int(1+\|u\|_{L^4}^{r})\,\mathrm{d} s-c_{\Lambda,p,\kappa}\theta\Xi_n}
 \Bigm|\mathcal{F}_{\tau_n}
 \right]
 \le C_{{\rm wt},\kappa}^{\,\theta\eta}
 e^{\theta\eta\kappa e^{-\chi}V(Z_{\tau_n})}
 \le C_{\Lambda,p,\kappa}^{\,\theta}e^{\theta\kappa V(Z_{\tau_n})},
\end{equation*}
where \(C_{\Lambda,p,\kappa}\ge1\) absorbs the fixed power \(C_{\mathrm{wt},\kappa}^{\eta}\).
By \eqref{eq:V-nonnegative}, \(V\ge0\), so the negative terminal-energy factor produced in this rearrangement is at most one and may be discarded. No exponential integrability of \(X_n\) has been used. If \(0<\Lambda\le\Lambda_*\), the Young constant above is bounded by the one obtained with \(\Lambda_*\); choosing the clock coefficient and prefactor with this upper bound proves the asserted local uniformity in \(\Lambda\).
\end{proof}
\begin{lemma}\label{lem:Xi-iid}
The sequence $\{\Xi_n\}_{n\ge0}$ defined in \eqref{eq:XnXi} is independent and identically
distributed.  In particular,
\begin{equation}\label{eq:Xi-bar}
 \overline\Xi=\mathbb{E}[\Xi_0]<\infty.
\end{equation}
\end{lemma}

\begin{proof}
Put $d_\kappa=c_Q\kappa B_Q$ and $S_r^{(n)}=S_{\tau_n+r}-S_{\tau_n}$. Then
\begin{equation*}
 \sigma_n:=\tau_{n+1}-\tau_n
 =\inf\{r\ge0:c_Vr-d_\kappa S_r^{(n)}\ge\chi\}.
\end{equation*}
By \eqref{eq:clock-block-bound},
\begin{equation*}
 \Xi_n=\int_0^{\sigma_n}e^{8(\chi-c_Vr+d_\kappa S_r^{(n)})}\,\mathrm dr
       +S_{\sigma_n}^{(n)}.
\end{equation*}
Thus $\Xi_n$ is the same functional of the stopped increment path on each block. The strong Markov property of $S$ at $\tau_n$ gives independence from the preceding blocks and equality in law. Finiteness of \eqref{eq:Xi-bar} follows from \eqref{eq:Xi-moment}.
\end{proof}

\begin{proposition}\label{prop:clock-mean-bound}
Let $d_\kappa=c_Q\kappa B_Q$. Assume
$
 8d_\kappa\le\zeta_S$,
 $\Delta_8:=8c_V-\Psi_S(8d_\kappa)>0$.
Then  $c_V-d_\kappa\bar s>0$ and
\begin{equation}\label{eq:Xi-explicit-upper}
 \overline\Xi\le \overline\Xi^{\sharp}(\kappa,\chi,B_Q)
 :=\frac{e^{8\chi}}{\Delta_8}
 +\frac{\bar s\,\chi}{c_V-d_\kappa\bar s}.
\end{equation}
In particular, if a family of forcing operators satisfies $B_Q\le B_\infty$ and
\begin{equation*}
 8c_Q\kappa B_\infty\le\zeta_S,
 \qquad
 \Delta_{8,\infty}:=8c_V-\Psi_S(8c_Q\kappa B_\infty)>0,
\end{equation*}
then
\begin{equation*}
 \overline\Xi\le
 \frac{e^{8\chi}}{\Delta_{8,\infty}}
 +\frac{\bar s\,\chi}{c_V-c_Q\kappa B_\infty\bar s}
\end{equation*}
uniformly over the whole family.
\end{proposition}
\begin{proof}
	Write $T=\tau_1$ and
	$\mathscr{C}_t=c_Vt-d_\kappa S_t$.
We have
	\begin{equation}\label{eq:RT-continuous-crossing}
	\mathscr{C}_T=\chi,\qquad \Delta S_T=0.
	\end{equation}
	For $0\le s\le T$, the block weight satisfies the  identity
	\begin{equation*}
	\omega_0(s)
	=\exp\{-c_V(T-s)+d_\kappa(S_T-S_s)\}
	=\exp\{\mathscr{C}_s-\chi\}.
	\end{equation*}
	Consequently, we have
	\begin{align}
	\mathbb{E}[X_0]
	&=e^{8\chi}\mathbb{E}\left[\int_0^T e^{-8\mathscr{C}_s}\,\,\mathrm{d} s\right]\le e^{8\chi}\int_0^\infty \mathbb{E}[e^{-8\mathscr{C}_s}]\,\,\mathrm{d} s\notag\\
	&=e^{8\chi}\int_0^\infty
	e^{-s[8c_V-\Psi_S(8d_\kappa)]}\,\,\mathrm{d} s
	=\frac{e^{8\chi}}{\Delta_8}.
	\label{eq:X0-occupation-bound}
	\end{align}

Since $\Psi_S(8d_\kappa)\ge8d_\kappa\bar s$, we have $c_V-d_\kappa\bar s>0$, and $T<\infty$ by the strong law. Moreover,
\begin{equation*}
 \mathbb P(T>t)\le e^{8\chi}\mathbb E e^{-8\mathscr C_t}
 =e^{8\chi-\Delta_8t},
\end{equation*}
so $\mathbb E T<\infty$ under the hypotheses of this proposition. To estimate $Y_0=S_T$, use \eqref{eq:RT-continuous-crossing} to write
	\begin{equation*}
	S_T=S_{T-}
	=\int_0^\infty\int_0^\infty
	r\,\mathbf{1}_{\{s\le T\}}N(\,\mathrm{d} s,\,\mathrm{d} r),
	\end{equation*}
	where $N$ is the jump measure of the subordinator. The process \(s\mapsto\mathbf{1}_{\{s\le T\}}\) is adapted and left-continuous, hence predictable; since there is no jump at \(T\) in \eqref{eq:RT-continuous-crossing}, the displayed representation is exact. The compensation formula therefore gives
	\begin{equation}\label{eq:ST-wald}
	\mathbb{E}[S_T]=\bar s\,\mathbb{E}[T].
	\end{equation}
	Taking expectations in $\mathscr{C}_T=\chi$ and using \eqref{eq:ST-wald} yields
	\begin{equation*}
	(c_V-d_\kappa\bar s)\mathbb{E}[T]=\chi.
	\end{equation*}
	Moreover, $\Psi_S(8d_\kappa)\ge8d_\kappa\bar s$, hence
	$\Delta_8>0$ implies $c_V-d_\kappa\bar s>0$.  Thus, we obtain
	\begin{equation*}
	\mathbb{E}[Y_0]=\bar s\mathbb{E}[T]
	=\frac{\bar s\,\chi}{c_V-d_\kappa\bar s}.
	\end{equation*}
	Combining this with \eqref{eq:X0-occupation-bound} proves
	\eqref{eq:Xi-explicit-upper}.  The uniform estimate follows from the monotonicity of
	$\Psi_S$ and $B_Q\le B_\infty$.
\end{proof}

Fix $\widehat\delta>0$ and define
\begin{equation*}
 \widehat\Theta^{(\widehat\delta)}
 =\sup_{N\ge1}\left(
   \sum_{j=0}^{N-1}\Xi_j-(\overline\Xi+\widehat\delta)N
 \right)_+.
\end{equation*}
By the strong law of large numbers,
\begin{equation}\label{eq:Theta-centered-finite}
 \widehat\Theta^{(\widehat\delta)}<\infty\quad\text{almost surely}.
\end{equation}
Hence,
\begin{equation}\label{eq:Theta-centered-tail}
 \mathbb{P}\bigl(\widehat\Theta^{(\widehat\delta)}>L\bigr)\longrightarrow0
 \quad\text{as }L\to\infty.
\end{equation}
On the event $\{\widehat\Theta^{(\widehat\delta)}\le L\}$, one has, for every $N\ge1$,
\begin{equation*}
 \sum_{j=0}^{N-1}\Xi_j
 \le (\overline\Xi+\widehat\delta)N+L.
\end{equation*}
For $b>0$, let
\begin{equation*}
 \widetilde\Theta^{(b)}
 =\sup_{n\ge0}(\Xi_n-bn)_+.
\end{equation*}
For every $q\in(1,p_{\Xi}]$ and $K>1$,
\begin{align}
 \mathbb{P}\bigl(\widetilde\Theta^{(b)}>K\bigr)
 &\le\sum_{n=0}^\infty\mathbb{P}(\Xi_n>K+bn)\notag\\
 &\le C_q\sum_{n=0}^\infty(K+bn)^{-q}
 \longrightarrow0
 \quad\text{as }K\to\infty.
 \label{eq:Theta-local-tail}
\end{align}
Thus $\widetilde\Theta^{(b)}<\infty$ almost surely.

For $L,K>0$, define the two-scale good-clock event
\begin{equation*}
 \mathbb{G}_{L,K}^{\widehat\delta,b}
 =\{\widehat\Theta^{(\widehat\delta)}\le L\}
  \cap\{\widetilde\Theta^{(b)}\le K\}.
\end{equation*}
On $\mathbb{G}_{L,K}^{\widehat\delta,b}$,
\begin{equation*}
 \sum_{j=0}^{N-1}\Xi_j\le(\overline\Xi+\widehat\delta)N+L,
 \qquad
 \Xi_n\le K+bn,
\end{equation*}
for all $N\ge1$ and $n\ge0$, while
\begin{equation}\label{eq:good-clock-tail}
 \mathbb{P}\bigl((\mathbb{G}_{L,K}^{\widehat\delta,b})^c\bigr)\longrightarrow0
 \quad\text{as }L,K\to\infty
\end{equation}
for every fixed $\widehat\delta,b>0$.  Both envelopes depend only on \(S\), so their Brownian Malliavin derivatives vanish.

\subsection{Linearization and conditional Malliavin calculus}
Let $\mathbb{S}(t)$ be the semigroup introduced in
Lemma \ref{lem:compactness}.  For $0\le s\le t$, the derivative of the solution map is defined by the following
mild variational equation:
\begin{equation}\label{eq:linearized}
 J_{s,t}\xi
 =\mathbb{S}(t-s)\xi
 -\int_s^t\mathbb{S}(t-r)
 \left(0,P_0\!\left(g'(u_r)(J_{s,r}\xi)_1\right)\right)\,\mathrm dr,
 \qquad J_{s,s}\xi=\xi.
\end{equation}
For $\eta=(h,k)$ sufficiently regular, we set
\begin{equation*}
 \mathcal{L}_t\eta
 =\bigl(k,-Ak-Ah-P_0(g'(u_t)h)\bigr)
\end{equation*}
and denote $K_{s,t}=J_{s,t}^*$.  The identity \(\partial_tJ_{s,t}\xi=\mathcal{L}_tJ_{s,t}\xi\) holds on smooth vectors.

Let \(J^{(2)}_{s,t}(\xi,\eta)\) denote the second variation with respect to the initial condition. We use
\begin{equation}\label{eq:J2-equation}
 \partial_t J^{(2)}_{s,t}=\mathcal{L}_tJ^{(2)}_{s,t}
 +\left(0,-P_0\!\left(g''(u_t)(J_{s,t}\xi)_1(J_{s,t}\eta)_1\right)\right),
 \qquad J^{(2)}_{s,s}=0,
\end{equation}
We also define the bilinear second nonlinear variation by
\begin{equation}\label{eq:D2F}
 \mathcal{B}_t[(a,b),(c,d)]
 =\left(0,-P_0(g''(u_t)a c)\right).
\end{equation}
After freezing a clock path $S=\ell$, let $D_r^j$ denote Brownian Malliavin differentiation at
 time $r$ and put $\gamma_r=\inf\{q\ge0:\ell_q\ge r\}$.  For $t\ge s\vee\gamma_r$, the
directional Malliavin variation is described by
\begin{equation}\label{eq:DJ-equation}
 \partial_t (D_r^jJ_{s,t}\xi)
 =\mathcal{L}_tD_r^j(J_{s,t}\xi)
 +\mathcal{B}_t[D_r^jZ_t,J_{s,t}\xi],
\end{equation}
with zero initial value at $t=s\vee\gamma_r$; for $s\le t<\gamma_r$ the derivative vanishes.

\begin{proposition}\label{prop:differentiability}
Let Assumptions \ref{ass:g}--\ref{ass:levy} hold and fix $T>0$.  For almost every forcing path,
$z\mapsto Z_t^z$ is twice Fr\'echet differentiable on bounded subsets of $\mathbb{H}$, uniformly for
$0\le t\le T$.  Its first and second derivatives are $J_{0,t}$ and $J_{0,t}^{(2)}$. For the flow restarted at $s$,
they are $J_{s,t}$ and $J_{s,t}^{(2)}$, defined by
\eqref{eq:linearized} and \eqref{eq:J2-equation}.  For every $R>0$, they are bounded on
$\{V(z)\le R,\ 0\le s\le t\le T\}$ by a finite path-dependent constant.

Condition on $S=\ell$. The solution is locally $C^2$ along Cameron--Martin
lines. Moreover, $Z_t$ and $J_{s,t}\xi$ belong to $\mathbb D_{W,\mathrm{loc}}^{1,p}$
for every finite $p$ and deterministic $\xi\in\mathbb H$, with the variational
bounds homogeneous in $\|\xi\|_{\mathbb H}$.  If $\gamma_r=\inf\{q\ge0:\ell_q\ge r\}$, then
\begin{equation}\label{eq:DZ-formula-main}
 D_r^j Z_t
 =\mathbf{1}_{\{r\le \ell_t\}}J_{\gamma_r,t}\mathbf{Q}\mathbf{e}_j^{(m)}.
\end{equation}
For deterministic $\xi\in\mathbb{H}$, $D_r^j(J_{s,t}\xi)=0$ for $t<s\vee\gamma_r$ and satisfies
\eqref{eq:DJ-equation} for $t\ge s\vee\gamma_r$.  The derivative is linear and continuous in $\xi$
under the localized variational estimates. {Clock-measurable endpoints and events have zero Brownian Malliavin derivative. Energy cutoffs are retained in the local Sobolev argument; the weighted bounds needed for the block controls are proved in Lemma~\ref{lem:variational-subcritical}.}
\end{proposition}
\begin{proof}
	We first justify differentiability with respect to the initial condition. After subtracting the additive noise, the resulting equation is a deterministic nonautonomous strongly damped wave equation. On an energy ball 
$\{V\le M\}$, Assumption \ref{ass:g} and
	$H^1(\mathbb{T}^3)\hookrightarrow L^6(\mathbb{T}^3)$ give
	\begin{align}
	\|P_0(g'(u)h)\|_{L^2}
	&\le C_M\|h\|_{H^1},\label{eq:nemytskii-first-local}\\
	\|P_0(g''(u)h_1h_2)\|_{L^2}
	&\le C_M\|h_1\|_{H^1}\|h_2\|_{H^1}.\label{eq:nemytskii-second-local}
	\end{align}
If $u_j\to u$ in $H^1$, the growth bound on $g''$ and Vitali's theorem give
\begin{equation}\label{eq:gsecond-L6-continuity}
 \|g''(u_j)-g''(u)\|_{L^6}\to0.
\end{equation}
Indeed, strong $L^6$ convergence makes $|g''(u_j)|^6$ uniformly integrable. H\"older's inequality then gives continuity of the second derivative as a bilinear map from $H^1\times H^1$ to $L^2$.

Set $\mathcal N(u,v)=(0,-P_0g(u))$. The mild equation is
\begin{equation}\label{eq:pathwise-mild-flow}
 Z_t^z=\mathbb S(t)z+\int_0^t\mathbb S(t-r)\mathcal N(Z_r^z)\,\mathrm dr
       +\int_0^t\mathbb S(t-r)\mathbf Q\,\mathrm dL_r.
\end{equation}
The last term is independent of $z$. On bounded subsets,
\begin{equation}\label{eq:gprime-difference-operator}
 \|[g'(u)-g'(\widetilde u)]h\|_{L^2}
 \le C_M\|u-\widetilde u\|_{H^1}\|h\|_{H^1}.
\end{equation}
Thus $\mathcal N$ is $C^2$ with locally bounded derivatives. The pathwise estimate \eqref{eq:pathwise-galerkin-energy} bounds all trajectories with $V(z)\le R$ on $[0,T]$. Subtracting \eqref{eq:pathwise-mild-flow} for two initial states and applying Gronwall gives local Lipschitz dependence. Taylor's formula and \eqref{eq:gprime-difference-operator} further give
\begin{equation*}
 \sup_{t\le T}\|Z_t^{z+h}-Z_t^z-J_{0,t}h\|_{\mathbb H}
 \le C_{R,T}(\omega)\|h\|_{\mathbb H}^2
\end{equation*}
for sufficiently small $h$. Differentiating the linearized integral equation, using \eqref{eq:gsecond-L6-continuity}, yields \eqref{eq:J2-equation} and
\begin{equation}\label{eq:localized-J-C2-bound}
 \sup_{V(z)\le R}\sup_{0\le s\le t\le T}
 \left(\|J_{s,t}\|_{\mathbb L(\mathbb H)}
 +\|J^{(2)}_{s,t}\|_{\mathcal L^{(2)}(\mathbb H;\mathbb H)}\right)
 \le C_{R,T}(\omega)<\infty.
\end{equation}

For a fixed clock $\ell$, take $h\in L^2(0,\ell_T;\mathbb R^m)$, extended by
zero, and put $H_r=\int_0^r h(q)\,\mathrm dq$.
The shift $W+\varepsilon H$ changes the convolution in
\eqref{eq:pathwise-mild-flow} by
$\varepsilon F_t^h$, where
\[
 F_t^h=\int_0^{\ell_t}\mathbb S(t-\gamma_q)\mathbf Qh(q)\,\mathrm dq,
 \qquad
 \sup_{t\le T}\|F_t^h\|_{\mathbb H}
 \le C_T\|\mathbf Q\|\ell_T^{1/2}\|h\|_{L^2(0,\ell_T)}.
\]
The shifted forcing is bounded on $[0,T]$, uniformly for $\varepsilon$ in a
compact interval. The local $C^2$ bounds for $\mathcal N$ and the preceding
pathwise energy estimate thus permit differentiation of the parameterized
mild equation. If $Z^{\varepsilon,h}$ is the shifted solution, its first
variation satisfies
\[
 X_t^h=F_t^h+\int_0^t\mathbb S(t-r)D\mathcal N(Z_r)X_r^h\,\mathrm dr.
\]
Variation of constants gives
\begin{equation}\label{eq:CM-variation}
 \left.\partial_\varepsilon Z_t^{\varepsilon,h}\right|_{\varepsilon=0}
 =X_t^h
 =\int_0^{\ell_t}J_{\gamma_q,t}\mathbf Qh(q)\,\mathrm dq.
\end{equation}
Its second variation $Y_t^h=\left.\partial_\varepsilon^2
Z_t^{\varepsilon,h}\right|_{\varepsilon=0}$ satisfies
\[
 Y_t^h=\int_0^t\mathbb S(t-r)
 \bigl(D\mathcal N(Z_r)Y_r^h+\mathcal B_r[X_r^h,X_r^h]\bigr)\,\mathrm dr.
\]
Estimate \eqref{eq:nemytskii-second-local} bounds the inhomogeneous term.
Differentiating $J_{s,t}\xi$ along the same line gives \eqref{eq:DJ-equation}
integrated against $h$. Consequently both fields are locally absolutely
continuous along these lines, also when the initial state and direction
are fixed external parameters.

For Galerkin convergence, write $J^N_{s,r}\Pi_N\xi=(h^N_r,k^N_r)$ and
$J_{s,r}\xi=(h_r,k_r)$. On a common energy ball, set
\[
 F_N(r)=P_NP_0\{[g'(u_r^N)-g'(u_r)]h_r\}
       -(I-P_N)P_0(g'(u_r)h_r).
\]
By \eqref{eq:gprime-difference-operator},
\[
 \|F_N(r)\|
 \le C_M\|u_r^N-u_r\|_{H^1}\|h_r\|_{H^1}
       +\|(I-P_N)P_0(g'(u_r)h_r)\|.
\]
Testing the difference equation and applying Gronwall yield
\begin{align*}
 &\sup_{s\le r\le T}\|(h^N_r-h_r,k^N_r-k_r)\|_{\mathbb H}^2
   +\int_s^T\|k^N_r-k_r\|_{H^1}^2\,\mathrm dr\\
 &\quad\le C_{M,T}\left(\|(I-\Pi_N)\xi\|_{\mathbb H}^2
                       +\int_s^T\|F_N(r)\|^2\,\mathrm dr\right)
 \longrightarrow0.
\end{align*}
The first part of $F_N$ vanishes by \eqref{eq:state-galerkin-convergence};
the second by dominated convergence and \eqref{eq:nemytskii-first-local}.
Thus, for each fixed direction,
\begin{equation}\label{eq:J-galerkin-convergence}
 J^N_{s,\cdot}\Pi_N\xi\longrightarrow J_{s,\cdot}\xi
 \quad\text{in }C([s,T];\mathbb H).
\end{equation}

	Fix the subordinator path $S=\ell$. For the Galerkin equation, the driving path is the
	finite-dimensional Gaussian family $W_{\ell_t}$, so that finite-dimensional Malliavin calculus gives,
	for almost every $r$,
	\begin{equation}\label{eq:DZ-galerkin}
	D_r^j Z_t^N
	=\mathbf{1}_{\{r\le\ell_t\}}
	J^N_{\gamma_r,t}\mathbf{Q}\mathbf{e}_j^{(m)}.
	\end{equation}
Apply smooth cutoffs to the Galerkin coefficients before using Malliavin closability. On each energy-localized set, \eqref{eq:localized-J-C2-bound} gives uniform bounds for the cutoff derivatives in $L^p(\Omega_W;L^2(0,\ell_T;\mathbb H))$. The closedness of the Malliavin derivative and \eqref{eq:J-galerkin-convergence} identify the local limit with \eqref{eq:DZ-formula-main}.

Differentiating the Galerkin first variation gives \eqref{eq:DJ-equation}, with forcing
\begin{equation*}
 (0,-P_NP_0[g''(u_t^N)(D_r^ju_t^N)(J_{s,t}^N\xi)_1]).
\end{equation*}
Estimate \eqref{eq:nemytskii-second-local} and the same cutoff argument yield local $L^p$ bounds for this derivative, homogeneous in $\|\xi\|_{\mathbb H}$. Weak closedness identifies the limit, and Mazur's lemma gives strong convergence of convex combinations in the local graph norm. Lemma~\ref{lem:external-random-direction} extends the construction to external random directions. The cutoff sets exhaust the sample space by \eqref{eq:finite-horizon-sup}; this establishes local Sobolev regularity. The weighted global bounds used below follow separately from \eqref{eq:variational-block-moment}.
\end{proof}

For a fixed path $\ell$, use the inverse clock $r\mapsto\gamma_r$ defined above.  For
$h\in L^2([\ell_s,\ell_t];\mathbb{R}^m)$, the Malliavin derivative in the Brownian direction $h$ is
\begin{equation}\label{eq:A-control}
 \mathbb{A}_{s,t}h
 =\int_{\ell_s}^{\ell_t}J_{\gamma_r,t}\mathbf{Q} h(r)\mathrm{d} r.
\end{equation}
Its adjoint is
\begin{equation*}
 (\mathbb{A}_{s,t}^*\phi)(r)=\mathbf{Q}^*K_{\gamma_r,t}\phi,
 \qquad r\in[\ell_s,\ell_t].
\end{equation*}
The Malliavin covariance is
\begin{equation*}
 \mathbb{M}_{s,t}=\mathbb{A}_{s,t}\mathbb{A}_{s,t}^*,
\end{equation*}
and therefore
\begin{equation}\label{eq:M-quadratic}
 \langle\mathbb{M}_{s,t}\phi,\phi\rangle_{\mathbb{H}}
 =\sum_{j=1}^m q_j^2\int_s^t
   \bigl\langle K_{r,t}\phi,(0,e_j)\bigr\rangle_{\mathbb{H}}^2\mathrm{d}\ell_r.
\end{equation}

We next present the variational estimates.
\begin{lemma}\label{lem:variational-subcritical}
Let $J_{s,t}$ be the first variational flow. Then there exists a constant $C>0$ such that
\begin{equation}\label{eq:J-subcritical}
 \|J_{s,t}\xi\|_{\mathbb{H}}^2
 \le C\|\xi\|_{\mathbb{H}}^2
 \exp\left\{C(t-s)+C\int_s^t\|u_r\|_{L^4}^{16/5}\mathrm{d} r\right\}.
\end{equation}
The second variation $J^{(2)}_{s,t}$ is defined by \eqref{eq:J2-equation}.
For the stopping block set
\begin{equation*}
 \mathcal{H}_n=L^2([S_{\tau_n},S_{\tau_{n+1}}];\mathbb{R}^m),
\end{equation*}
let \(D^{[n]}\) be the Malliavin derivative with respect to the Brownian increments on this operational-time interval, after conditioning on the subordinator. Operators defined on a smaller operational subinterval are extended by zero to \(\mathcal{H}_n\). Fix a countable dense \(\mathbb{Q}\)-linear tangent core \(\mathcal{T}\subset\mathbb{H}\), and define the directional seminorm
\begin{equation}\label{eq:directional-DJ-seminorm}
 \mathfrak d_n(J_{s,t})
 :=\sup_{\xi\in\mathcal{T}\setminus\{0\}}
 \frac{\|D^{[n]}(J_{s,t}\xi)\|_{\mathcal{H}_n\otimes\mathbb{H}}}{\|\xi\|_{\mathbb{H}}}.
\end{equation}
The seminorm is finite almost surely and extends by density to deterministic directions.  Malliavin
derivatives of $\mathbb{A}_{s,t}$ and $\mathbb{A}_{s,t}^*$ are taken in the corresponding
Hilbert--Schmidt spaces $\mathbb{L}_2$.  Set
\begin{align}\label{eq:J-block-definition}
 \mathbb{J}_n:=1+\sup_{\tau_n\le s\le t\le\tau_{n+1}}&\Big(
 \|J_{s,t}\|_{\mathbb{L}(\mathbb{H})}
 +\|J^{(2)}_{s,t}\|_{\mathcal{L}^{(2)}(\mathbb{H};\mathbb{H})}
 +\|\mathbb{A}_{s,t}\|_{\mathbb{L}(\mathcal{H}_n;\mathbb{H})}
 +\|\mathbb{A}_{s,t}^*\|_{\mathbb{L}(\mathbb{H};\mathcal{H}_n)}\notag\\
 &+\mathfrak d_n(J_{s,t})
 +\|D^{[n]}\mathbb{A}_{s,t}\|_{\mathcal{H}_n\otimes\mathbb{L}_2(\mathcal{H}_n;\mathbb{H})}+\|D^{[n]}\mathbb{A}_{s,t}^*\|_{\mathcal{H}_n\otimes\mathbb{L}_2(\mathbb{H};\mathcal{H}_n)}\Big).
\end{align}
Here $\mathcal{L}^{(2)}(\mathbb{H};\mathbb{H})$ is the space of bounded bilinear maps.
For every $p\ge1$ and $\eta_V\in(0,\kappa]$, there exist constants
$c_{p,\eta_V,\kappa},C_{p,\eta_V,\kappa}>0$ such that
\begin{equation}\label{eq:variational-block-moment}
 \mathbb{E}\left[
   \mathbb{J}_n^p\exp\{-c_{p,\eta_V,\kappa}(X_n+Y_n)\}
   \,\middle|\,\mathcal{F}_{\tau_n}
 \right]
 \le C_{p,\eta_V,\kappa}\exp\{\eta_V V(Z_{\tau_n})\}.
\end{equation}
The stopping blocks use the original $\kappa$.  Within this lemma, the largest state power requiring inverse-clock
normalization is $\|u\|_{L^4}^{16/5}$; the coupled high--low estimate later uses $32/9$. No larger power is needed for the second or Malliavin variations here.
\end{lemma}

\begin{proof}
Write $J_{s,t}\xi=(h,k)$ and use the energy $\mathcal W$ in \eqref{eq:res-energy}.
The nonlinear velocity pairing satisfies
\begin{align*}
 \int |u|^2|h||k|
 &\le \|u\|_{L^4}^2\|h\|_{L^4}\|k\|_{L^4}\\
 &\le C\|u\|_{L^4}^2\|h\|_{H^1}\|k\|^{1/4}\|k\|_{H^1}^{3/4}\\
 &\le\tfrac14\|k\|_{H^1}^2+C\|u\|_{L^4}^{16/5}\mathcal W(h,k).
\end{align*}
The other pairings are bounded by $C(1+\|u\|_{L^4}^2)\mathcal W$.
Gronwall proves \eqref{eq:J-subcritical}. On a stopping block,
\begin{equation}\label{eq:weighted-young-variational}
 C\|u\|_{L^4}^{16/5}
 \le\tfrac14c_*\kappa\alpha_2\omega_n\|u\|_{L^4}^4
       +C(\kappa)\omega_n^{-4}.
\end{equation}
For the second variation, pair against $\psi\in H^1$. H\"older with exponents $4,6,6,12/5$ and the Sobolev embeddings give
\begin{equation}\label{eq:gsecond-Hminus1-subcritical}
 \|P_0(g''(u)h_1h_2)\|_{H^{-1}}
 \le C(1+\|u\|_{L^4})\|h_1\|_{H^1}\|h_2\|_{H^1}.
\end{equation}

If $w=J^{(2)}_{s,r}(\xi,\eta)$, testing its equation in the modified energy gives
\[
 \frac{\mathrm d}{\mathrm dr}\mathcal W(w)
 +c\|w_2\|_{H^1}^2
 \le C(1+\|u\|_{L^4}^{16/5})\mathcal W(w)
 +C(1+\|u\|_{L^4}^2)\|J_{s,r}\xi\|^2\|J_{s,r}\eta\|^2.
\]
This estimate uses the $H^{-1}$--$H^1$ pairing and therefore requires no
pointwise $H^1$ regularity of the initial velocity.

Set
\[
 \mathcal K_n=\exp\left\{C\int_{\tau_n}^{\tau_{n+1}}
                    (1+\|u_r\|_{L^4}^{16/5})\,\mathrm dr\right\}.
\]
Integration of the last two energy inequalities yields
\begin{align*}
 \sup_{s,t}\|J_{s,t}\|&\le C\mathcal K_n,\\
 \sup_{s,t}\|J^{(2)}_{s,t}\|
 &\le C\mathcal K_n^3
       \left(1+\int_{\tau_n}^{\tau_{n+1}}(1+\|u_r\|_{L^4}^2)\,\mathrm dr\right).
\end{align*}
All suprema in this proof are over $\tau_n\le s\le t\le\tau_{n+1}$.
For the frozen clock,
\begin{equation}\label{eq:DZ-operational}
 D_r^jZ_t=\mathbf1_{\{r\le\ell_t\}}J_{\gamma_r,t}\mathbf Q\mathbf e_j^{(m)}.
\end{equation}
The equation for $D_r^jJ_{s,t}\xi$ has the same inhomogeneous energy estimate,
with the two factors $D_r^jZ$ and $J_{s,\cdot}\xi$. Integrating in $r$ contributes
only powers of $Y_n$. Thus, for every $\xi\in\mathbb H$,
\begin{equation}\label{eq:DJ-simultaneous-bound}
 \sup_{s,t}\|D^{[n]}(J_{s,t}\xi)\|_{\mathcal H_n\otimes\mathbb H}
 \le\mathcal R_n\|\xi\|,\qquad
 \mathcal R_n=C(1+Y_n)^N\mathcal K_n^N
 \left(1+\int_{\tau_n}^{\tau_{n+1}}(1+\|u_r\|_{L^4}^2)\,\mathrm dr\right)^N.
\end{equation}
The directional supremum is justified by this common pathwise bound, followed by
density; it is not obtained by interchanging a supremum and an expectation.

Differentiating the finite-rank kernel in \eqref{eq:A-control} gives
\begin{equation}\label{eq:DA-operational}
 D_r^j(\mathbb A_{s,t}h)
 =\int_{\ell_s}^{\ell_t}D_r^j(J_{\gamma_q,t}\mathbf Qh(q))\,\mathrm dq.
\end{equation}
Consequently,
\begin{equation}\label{eq:DA-norm}
 \|D^{[n]}\mathbb A_{s,t}\|_{\mathcal H_n\otimes\mathbb L_2(\mathcal H_n;\mathbb H)}
 =\|D^{[n]}\mathbb A_{s,t}^*\|_{\mathcal H_n\otimes\mathbb L_2(\mathbb H;\mathcal H_n)}
 \le C Y_n^{1/2}\mathcal R_n.
\end{equation}
These are derivatives in separable Hilbert--Schmidt spaces, obtained first for
Galerkin kernels and then by local closability. The product and resolvent rules give
\begin{align}
 D^{[n]}\mathbb M&=(D^{[n]}\mathbb A)\mathbb A^*
                  +\mathbb A D^{[n]}\mathbb A^*,\label{eq:DM-norm}\\
 D^{[n]}G_\beta&=-G_\beta(D^{[n]}\mathbb M)G_\beta,
 \qquad G_\beta=(\mathbb M+\beta I)^{-1},\quad\|G_\beta\|\le\beta^{-1}.
 \label{eq:DG-resolvent}
\end{align}
The first identity is valid in the trace-class norm, square-integrated over the
operational variable. Since $\|\mathbb A\|+\|\mathbb A^*\|\le
C Y_n^{1/2}\sup_{s,t}\|J_{s,t}\|$, the preceding bounds yield, for a fixed integer $N$,
\begin{equation}\label{eq:block-variational-master}
 \mathbb J_n\le C(1+Y_n)^N\mathcal K_n^N
 \left(1+\int_{\tau_n}^{\tau_{n+1}}(1+\|u_r\|_{L^4}^2)\,\mathrm dr\right)^N.
\end{equation}
Using $Y_n\le\Xi_n$, $\tau_{n+1}-\tau_n\le X_n$ and
$(1+x)^a\le C_{a,\epsilon}e^{\epsilon x}$, we obtain
\begin{equation}\label{eq:block-variational-single-exponential}
 \mathbb J_n^p\le C_{p,\epsilon}e^{\epsilon\Xi_n}
 \exp\left\{A_p\int_{\tau_n}^{\tau_{n+1}}(1+\|u_r\|_{L^4}^{16/5})\,\mathrm dr\right\}.
\end{equation}
Apply Lemma~\ref{lem:normalized-block-exp} with moment parameter $1$,
$\theta=\eta_V/\kappa$, and $\Lambda=A_p/\theta$.
Taking $c_{p,\eta_V,\kappa}=\epsilon+c_{\Lambda,1,\kappa}\theta$ proves
\eqref{eq:variational-block-moment}. The block definition remains fixed.
\end{proof}

\subsection{Terminal observability and Malliavin covariance}
Recall $H_m$ from \eqref{eq:Q} and set
\begin{equation*}
 \mathbb{H}_m=H_m\times H_m,
\end{equation*}
Thus $\Pi_m$ is the orthogonal projection of $\mathbb{H}$ onto $\mathbb{H}_m$.
For $\alpha\in(0,1]$, we define
\begin{equation*}
 \mathbb{U}_{\alpha,m}
 =\left\{\phi\in\mathbb{H}:\|\phi\|_{\mathbb{H}}=1,
 \ \|\Pi_m\phi\|_{\mathbb{H}}\ge\alpha\right\}.
\end{equation*}

\begin{lemma}\label{lem:terminal-observability}
Fix $t>0$ and let $\phi=(\phi_1,\phi_2)\in\mathbb{H}$.  For $1\le j\le m$, set
\begin{equation*}
 f_j(r)=\bigl\langle \phi,J_{r,t}(0,e_j)\bigr\rangle_{\mathbb{H}}
       =\bigl\langle K_{r,t}\phi,(0,e_j)\bigr\rangle_{\mathbb{H}},
 \qquad 0\le r\le t.
\end{equation*}
Then the left derivatives at the time \(t\) satisfy
\begin{equation}\label{eq:terminal-data-f}
 f_j(t)=(\phi_2,e_j),
 \qquad
 f_j'(t)=-\lambda_j\bigl((\phi_1,e_j)-(\phi_2,e_j)\bigr).
\end{equation}
Therefore, there exists a constant $c_m>0$ such that
\begin{equation}\label{eq:finite-observability}
 \sum_{j=1}^m\bigl(|f_j(t)|^2+|f_j'(t)|^2\bigr)
 \ge c_m\|\Pi_m\phi\|_{\mathbb{H}}^2.
\end{equation}
Moreover, for every $h<t/2$ and every path for which
$\sup_{r\in[t-2h,t]}\|J_{r,t}\|_{\mathbb{L}(\mathbb{H})}<\infty$, the functions $f_j$ are $C^2$ on
$[t-2h,t]$ and
\begin{equation}\label{eq:fsecond-mild-identity}
 f_j''(r)
 =\bigl\langle\phi,
 J_{r,t}\mathcal{L}_r(e_j,-Ae_j)\bigr\rangle_{\mathbb{H}},
 \qquad r\in[t-2h,t].
\end{equation}
\end{lemma}

\begin{proof}
The identity is first verified for a spectral Galerkin equation containing the \(j\)th mode. In finite dimension, the evolution property yields, for every smooth deterministic vector \(\eta\),
\begin{equation*}
 \partial_r J_{r,t}^N\eta=-J_{r,t}^N\mathcal{L}_r^N\eta.
\end{equation*}
For $\eta=(0,e_j)$, one has
$\mathcal{L}_r^N(0,e_j)=(e_j,-Ae_j)$, which is independent of $r$.  Thus,
\begin{equation*}
 \partial_r^2J_{r,t}^N(0,e_j)
 =J_{r,t}^N\mathcal{L}_r^N(e_j,-Ae_j).
\end{equation*}
Apply \eqref{eq:J-galerkin-convergence} to the fixed smooth vectors
\((0,e_j)\) and \((e_j,-Ae_j)\). Together with
\eqref{eq:state-galerkin-convergence} and the common local operator bounds,
this permits passage to the limit in the integrated identities; the resulting
right-hand sides are continuous in $r$. Hence \eqref{eq:fsecond-mild-identity} follows without invoking an adjoint of the unbounded operator
\(\mathcal{L}_r\) on the whole phase space.

Since \(J_{t,t}=I\) at \(r=t\), we have \(f_j(t)=(\phi_2,e_j)\). The first initial-time derivative yields
\begin{equation*}
 f_j'(t)
 =-\langle\phi,(e_j,-Ae_j)\rangle_{\mathbb{H}}
 =-\lambda_j\bigl((\phi_1,e_j)-(\phi_2,e_j)\bigr),
\end{equation*}
which shows \eqref{eq:terminal-data-f}.  On $\mathbb{H}_m$, the map
$
 ((a_j),(b_j))\longmapsto ((b_j),(\lambda_j(a_j-b_j)))
$
is invertible, and \eqref{eq:finite-observability} follows by norm equivalence.
\end{proof}

\begin{lemma}\label{lem:strict-clock-increments}
Under Assumption \ref{ass:levy}, there is an event $\widehat\Omega$ of probability one such that for every $\omega\in\widehat\Omega$ and every $0\le a<b<\infty$,
$
 S_b(\omega)-S_a(\omega)>0$.
Consequently, if $[A_k,B_k]$, $1\le k\le N$, is any finite family of random intervals satisfying $A_k<B_k$ pointwise on $\widehat\Omega$, then
\begin{equation*}
 \min_{1\le k\le N}\bigl(S_{B_k}-S_{A_k}\bigr)>0
 \quad\text{on }\widehat\Omega.
\end{equation*}
\end{lemma}

\begin{proof}
For fixed $0\le p<q$, the increment $S_q-S_p$ has Laplace transform
\begin{equation*}
 \mathbb{E}[e^{-\lambda(S_q-S_p)}]
 =\exp\left\{-(q-p)\int_0^\infty(1-e^{-\lambda r})\nu_S(\mathrm dr)\right\}.
\end{equation*}
Since $\nu_S((0,\infty))=\infty$, the exponent $\to-\infty$ as $\lambda\to\infty$, hence $\mathbb{P}(S_q-S_p=0)=0$ for rational $p<q$. A countable intersection over rational pairs yields an a.s.\ event where all rational increments are strictly positive. For arbitrary $a<b$, pick rationals $a<p<q<b$; then
$
S_b-S_a\ge S_q-S_p>0$
by monotonicity, and the finite-family assertion follows pathwise.
\end{proof}

\begin{lemma}\label{lem:terminal-interval}
Fix $h,a,L,M>0$.  Let $f\in C^2([t-2h,t])$ satisfy
\begin{equation*}
 \sup_{r\in[t-2h,t]}|f''(r)|\le L,
 \qquad |f'(t)|\le M,
 \qquad \max\{|f(t)|,|f'(t)|\}\ge a.
\end{equation*}
Then there exist constants $\hat{\textbf{c}}=\hat{\textbf{c}}(h,a,L,M)>0$ and $\ell_*=\ell_*(h,a,L,M)>0$, and an interval $J\subset[t-2h,t]$ with $|J|\ge\ell_*$, such that
\begin{equation*}
 |f(r)|\ge \hat{\textbf{c}},\qquad r\in J.
\end{equation*}
\end{lemma}
\begin{proof}
Let $M_L=M+2hL$. Then $|f'|\le M_L$ on $[t-2h,t]$. If $|f(t)|\ge a$, take
	\begin{equation*}
	\ell_1=\min\left\{h,\frac{a}{2M_L}\right\},
	\end{equation*}
	with $a/(2M_L)=+\infty$ when $M_L=0$.
	Then, for $r\in[t-\ell_1,t]$,
	$|f(r)-f(t)|\le M_L\ell_1\le a/2$, and hence $|f(r)|\ge a/2$.
	
	Assume $|f'(t)|\ge a$. Set $a/(2L)=+\infty$ if $L=0$, and
	\[
	\ell_2=\min\left\{h,\frac{a}{2L}\right\}.
	\]
	Then $f'$ keeps constant sign on $[t-\ell_2,t]$ with $|f'|\ge a/2$, so
	\[
	|f(t)-f(t-\ell_2)|\ge\frac{a\ell_2}{2}.
	\]
	Hence some $r_0\in\{t-\ell_2,t\}$ has $|f(r_0)|\ge a\ell_2/4$.  On the interval adjacent to $r_0$ inside
	$[t-\ell_2,t]$ of length
	\begin{equation*}
	\ell_3=\min\left\{\frac{\ell_2}{2},\frac{a\ell_2}{8M_L}\right\},
	\end{equation*}
	the bound $|f'|\le M_L$ gives $|f(r)|\ge a\ell_2/8$.
	Let
	\begin{equation*}
	\ell_*=\min\{\ell_1,\ell_3\}
	\text{~~~and~~~}
	\hat{\textbf{c}}=\min\left\{\frac a2,\frac{a\ell_2}{8}\right\},
	\end{equation*}
	then the desired conclusion holds.
\end{proof}

\begin{lemma}\label{lem:null-direction}
	Under Assumption \ref{ass:levy} and \eqref{eq:Q}, for fixed $0<s<t$ and almost every subordinator path and corresponding Brownian solution path,
\begin{equation}\label{eq:no-null-S}
 \inf_{\phi\in\mathbb{U}_{\alpha,m}}
 \langle\mathbb{M}_{s,t}\phi,\phi\rangle_{\mathbb{H}}>0
 \quad\text{for every }\alpha\in(0,1].
\end{equation}
In particular, if $\langle\mathbb{M}_{s,t}\phi,\phi\rangle=0$, then $\Pi_m\phi=0$.
\end{lemma}

\begin{proof}
On the probability-one event of Lemma \ref{lem:strict-clock-increments}, if
$\langle\mathbb{M}_{s,t}\phi,\phi\rangle=0$, then \eqref{eq:M-quadratic} together with continuity of the $f_j$ and full support of $\mathrm d\ell$ forces $f_j(r)=0$ for all $r\in[s,t]$ and $j\le m$. Hence $f_j(t)=f_j'(t)=0$, so Lemma \ref{lem:terminal-observability} yields $\Pi_m\phi=0$.

For the uniform bound, choose
$0<h<(t-s)/2$.  For the fixed solution path, the continuity of $u$ in $H^1$, the boundedness of
$J_{r,t}$ on $[t-2h,t]$, and \eqref{eq:fsecond-mild-identity} give a constant
$L=L(\omega,s,t,m)<\infty$ such that
\begin{equation*}
 \max_{j\le m}\sup_{r\in[t-2h,t]}|f_j''(r)|\le L\|\phi\|_{\mathbb{H}}.
\end{equation*}
For unit $\phi$, \eqref{eq:terminal-data-f} also gives 
$|f_j'(t)|\le M_m:=\max_{j\le m}(\lambda_j+\sqrt{\lambda_j})$.  If
$\phi\in\mathbb{U}_{\alpha,m}$, \eqref{eq:finite-observability} supplies an index $j\le m$ with
\begin{equation}\label{eq:a-m-alpha}
 \max\{|f_j(t)|,|f_j'(t)|\}
 \ge a_{m,\alpha}:=\alpha\sqrt{\frac{c_m}{2m}}.
\end{equation}
By Lemma \ref{lem:terminal-interval} with parameters $(h,a_{m,\alpha},L,M_m)$, there exist constants $\hat{\textbf{c}},\ell_*>0$, independent of $\phi$, such that for each $\phi\in\mathbb{U}_{\alpha,m}$, some $j\le m$ and $J_\phi\subset[t-2h,t]$ of length $\ge\ell_*$ satisfy $|f_j|\ge \hat{\textbf{c}}$ on $J_\phi$.

Choose an integer $N$ so large that the equal mesh size $\delta=2h/N$ satisfies
$3\delta<\ell_*$.  Let $I_k=[t-2h+(k-1)\delta,t-2h+k\delta]$, $1\le k\le N$.  Every interval of length $\ge\ell_*$ contains a whole mesh cell $I_k$. Hence for every $\phi\in\mathbb{U}_{\alpha,m}$,
\begin{align*}
 \langle\mathbb{M}_{s,t}\phi,\phi\rangle_{\mathbb{H}}
 &\ge q_{\min}^2\hat{\textbf{c}}^2
 \min_{1\le k\le N}\bigl(\ell_{I_k^+}-\ell_{I_k^-}\bigr).
\end{align*}
The cells are finite in number and each has positive length, so Lemma \ref{lem:strict-clock-increments} makes the minimum strictly positive for the fixed clock path. This lower bound is independent of $\phi$, and the probability-one event is independent of $\alpha$. Equivalently, intersect first over rational $\alpha\in(0,1]$ and use $\mathbb{U}_{\alpha,m}\subset\mathbb{U}_{\alpha',m}$ for $0<\alpha'<\alpha$. Thus \eqref{eq:no-null-S} holds for all $\alpha\in(0,1]$.
\end{proof}

\begin{proposition}\label{prop:covariance}
Fix $R>0$ and $\alpha\in(0,1]$, and let $\tau_1$ be the first stopping time in \eqref{eq:stopping-times}.  Define
\begin{equation}\label{eq:r-eps}
 r(\varepsilon,\alpha,R,m)
 =\sup_{V(z_0)\le R}
 \mathbb{P}\left(
 \inf_{\phi\in\mathbb{U}_{\alpha,m}}
 \langle\mathbb{M}_{0,\tau_1}\phi,\phi\rangle_{\mathbb{H}}<\varepsilon
 \right).
\end{equation}
Then
\begin{equation}\label{eq:r-to-zero}
 \lim_{\varepsilon\to0}r(\varepsilon,\alpha,R,m)=0.
\end{equation}
\end{proposition}
\begin{proof}
Fix $0<h<\chi/(4c_V)$; then $\tau_1\ge\chi/c_V>4h$.
For $t=\tau_1$ and unit $\phi$, Lemma~\ref{lem:terminal-observability} gives
\begin{equation}\label{eq:fsecond-bound}
 \max_{j\le m}\sup_{r\in[t-2h,t]}|f_j''(r)|
 \le C_m\left(1+\sup_{r\le t}V(Z_r)\right)
       \sup_{r\in[t-2h,t]}\|J_{r,t}\|.
\end{equation}

Let $F_M=\{\tau_1\le M,\ \sup_{r\le\tau_1}V(Z_r)\le M\}$.
The moment bounds \eqref{eq:tau-exp} and \eqref{eq:stopped-sup} imply
\[
 \sup_{V(z)\le R}\mathbb P_z(F_M^c)\longrightarrow0.
\]
On $F_M$, \eqref{eq:J-subcritical} bounds the right-hand side of
\eqref{eq:fsecond-bound} by a deterministic $L=L(M,m,h)$.

Also $|f_j'(t)|\le M_m$ and, for every $\phi\in\mathbb U_{\alpha,m}$,
some $j$ satisfies \eqref{eq:a-m-alpha}.
It follows from Lemma~\ref{lem:terminal-interval} that there are deterministic $\hat{\textbf{c}},\ell_*>0$
such that $|f_j|\ge \hat{\textbf{c}}$ on an interval of length $\ell_*$ in $[t-2h,t]$.
Partition $[-2h,0]$ into finitely many cells of mesh less than $\ell_*/3$.
Each such interval contains a translated cell, so on $F_M$,
\begin{equation}\label{eq:cov-clock-lower}
 \inf_{\phi\in\mathbb U_{\alpha,m}}\langle\mathbb M_{0,t}\phi,\phi\rangle
 \ge q_{\min}^2 \hat{\textbf{c}}^2\min_{I\text{ a reference cell}}
             (S_{t+I^+}-S_{t+I^-})=:Z_M.
\end{equation}
The random time $t=\tau_1$ and $Z_M$ depend only on the clock.
Every cell has positive length; Lemma~\ref{lem:strict-clock-increments} gives
$Z_M>0$ almost surely, independently of the initial state. Thus
\[
 r(\varepsilon,\alpha,R,m)
 \le\sup_{V(z)\le R}\mathbb P_z(F_M^c)+\mathbb P(Z_M<\varepsilon).
\]
First let $\varepsilon\downarrow0$ and then $M\to\infty$.
\end{proof}

\section{Residual control and the e-property}\label{sec:control}

\subsection{High- and low-mode estimates}
Let $\Pi_m^\perp=I-\Pi_m$.  For $\rho=(\rho_1,\rho_2)$, we define
\begin{equation}\label{eq:res-energy}
 \mathcal{W}(\rho)=\frac12\bigl(\|A^{1/2}\rho_1\|^2+\|\rho_2\|^2\bigr)
 +\delta_0(\rho_1,\rho_2),
\end{equation}
with $\delta_0>0$ small enough, and write
$\widetilde\rho=\Pi_m\rho$ and $\widehat\rho=\Pi_m^\perp\rho$.

\begin{lemma}\label{lem:high-squeeze}
Let $\rho_r=J_{s,r}\rho_s$ be a solution of the homogeneous linearized wave equation on $[s,t]$. There exist constants $\mathfrak{c},C>0$ and a sequence $a_m\to0$ with
\begin{equation}\label{eq:am-rate}
 a_m\le C\bigl(\lambda_{m+1}^{-1/5}+\lambda_{m+1}^{-1/4}+\lambda_{m+1}^{-1/2}\bigr),
\end{equation}
such that
\begin{align}
 \frac{\,\mathrm{d}}{\,\mathrm{d} r}\mathcal{W}(\widehat\rho_r)
 &+\mathfrak c\mathcal{W}(\widehat\rho_r)
 +\frac14\|A^{1/2}\widehat\rho_{2,r}\|^2\le a_m\bigl(1+\|u_r\|_{L^4}^{16/5}\bigr)\mathcal{W}(\widehat\rho_r)
 +\mathcal{C}_m(u_r)\mathcal{W}(\widetilde\rho_r),
 \label{eq:high-projection-diff}
\end{align}
where $\mathcal{C}_m(u)$ is locally bounded for $u\in H^1$.  In particular, if the low component vanishes on the whole interval, then
\begin{align}
 \mathcal{W}(\widehat\rho_t)
 \le C\mathcal{W}(\widehat\rho_s)
 \exp\left\{-\mathfrak c(t-s)+a_m\int_s^t
  (1+\|u_r\|_{L^4}^{16/5})\mathrm{d} r\right\}.
 \label{eq:high-squeeze}
\end{align}
On a stopping block,
\begin{equation}\label{eq:high-weight-absorb}
 a_m\|u_r\|_{L^4}^{16/5}
 \le \frac{c_*\kappa\alpha_2}{8}\,\omega_n(r)\|u_r\|_{L^4}^4
 +C(\kappa) a_m^5\omega_n(r)^{-4}.
\end{equation}
\end{lemma}

\begin{proof}
For the high--high contribution, by \eqref{eq:g-prime-remainder},
\begin{equation*}
 g'(u)=g'(0)+\bigl(g'(u)-g'(0)\bigr),
 \qquad |g'(u)-g'(0)|\le C(|u|+|u|^2).
\end{equation*}
Setting $h=\widehat\rho_1$, $k=\widehat\rho_2$, the quadratic part satisfies
\begin{equation*}
 \int_{\mathbb{T}^3}|u|^2|h||k|\mathrm{d} x
 \le \|u\|_{L^4}^2\|h\|_{L^4}\|k\|_{L^4}.
\end{equation*}
For the high displacement, 
\begin{equation*}
 \|h\|_{L^4}
 \le C\|h\|^{1/4}\|A^{1/2}h\|^{3/4}
 \le C\lambda_{m+1}^{-1/8}\|A^{1/2}h\|,
\end{equation*}
and for the velocity retain the energy--dissipation interpolation,
\begin{equation*}
 \|k\|_{L^4}\le C\|k\|^{1/4}\|A^{1/2}k\|^{3/4}.
\end{equation*}
By Young's inequality with exponents $8/3$ and $8/5$, we can derive
\begin{align*}
 \int |u|^2|h||k|
 &\le \frac18\|A^{1/2}k\|^2
 +C\lambda_{m+1}^{-1/5}\|u\|_{L^4}^{16/5}
   \|A^{1/2}h\|^{8/5}\|k\|^{2/5}\\
 &\le \frac18\|A^{1/2}k\|^2
 +C\lambda_{m+1}^{-1/5}\|u\|_{L^4}^{16/5}
   \mathcal{W}(\widehat\rho).
\end{align*}
For the linear part, we have
\begin{align*}
 \int_{\mathbb{T}^3}|u||h||k|\mathrm{d} x
 &\le \|u\|_{L^4}\|h\|\|k\|_{L^4}\\
 &\le C\lambda_{m+1}^{-1/2}\|u\|_{L^4}\|A^{1/2}h\|
       \|k\|^{1/4}\|A^{1/2}k\|^{3/4}\\
 &\le \frac18\|A^{1/2}k\|^2
 +C\lambda_{m+1}^{-4/5}(1+\|u\|_{L^4}^{8/5})\mathcal{W}(\widehat\rho),
\end{align*}
which is absorbed by the $\lambda_{m+1}^{-1/5}$ term in \eqref{eq:am-rate} when $\lambda_{m+1}\ge1$.  The constant part $g'(0)$ contributes $C\lambda_{m+1}^{-1/2}\mathcal{W}(\widehat\rho)$. The term from $g'(u)-g'(0)$ involves linear and quadratic powers of $u$, and interpolation bounds it by $C\lambda_{m+1}^{-1/4}(1+\|u\|_{L^4}^2)\mathcal{W}(\widehat\rho)$. Collecting the $\widetilde\rho$-terms in $\mathcal{C}_m(u)\mathcal{W}(\widetilde\rho)$ yields \eqref{eq:high-projection-diff} and \eqref{eq:am-rate}.

For \eqref{eq:high-weight-absorb}, write
\begin{equation*}
 a_m\|u\|_{L^4}^{16/5}
 =\bigl(\omega_n\|u\|_{L^4}^{4}\bigr)^{4/5}
  \bigl(a_m\omega_n^{-4/5}\bigr)
\end{equation*}
and apply Young's inequality with exponents $5/4$ and $5$. This yields the quartic term and the remainder $C(\kappa)a_m^5\omega_n^{-4}$. Since $1\le\omega_n^{-8}$, the constant $a_m$ integrates to at most $a_mX_n$, and the $\omega_n^{-4}$ remainder is also bounded by $X_n$.
\end{proof}

\subsection{Tikhonov control and residual contraction}\label{subsec:residual-control}

Fix a unit initial tangent direction $\xi\in\mathbb{H}$, $\|\xi\|_{\mathbb{H}}=1$, and set $\rho_0=\xi$.  
For any $n\ge0$, define
\begin{equation*}
 R_{n,\beta}=\beta(\mathbb{M}_{\tau_n,\tau_{n+1}}+\beta I)^{-1}.
\end{equation*}
For even $n$, define the controlled Brownian direction on $[S_{\tau_n},S_{\tau_{n+1}}]$ by
\begin{equation*}
 h_n
 =\mathbb{A}_{\tau_n,\tau_{n+1}}^*
   (\mathbb{M}_{\tau_n,\tau_{n+1}}+\beta I)^{-1}
   J_{\tau_n,\tau_{n+1}}\rho_{\tau_n}.
\end{equation*}

For $t\in[\tau_n,\tau_{n+1}]$, let
\begin{equation*}
 \rho_t
 =J_{\tau_n,t}\rho_{\tau_n}
 -\mathbb{A}_{\tau_n,t}
   \bigl(h_n\mathbf{1}_{[S_{\tau_n},S_t]}\bigr),
\end{equation*}
while on the following uncontrolled block $[\tau_{n+1},\tau_{n+2}]$, set
\begin{equation*}
 \rho_t=J_{\tau_{n+1},t}\rho_{\tau_{n+1}}.
\end{equation*}
At the controlled endpoint, by 
$I-\mathbb{M}(\mathbb{M}+\beta I)^{-1}=\beta(\mathbb{M}+\beta I)^{-1}$ we have
$
 \rho_{\tau_{n+1}}
 =R_{n,\beta}J_{\tau_n,\tau_{n+1}}\rho_{\tau_n}$,
and hence
\begin{equation}\label{eq:res-recursion}
 \rho_{\tau_{n+2}}
 =J_{\tau_{n+1},\tau_{n+2}}R_{n,\beta}
  J_{\tau_n,\tau_{n+1}}\rho_{\tau_n},
 \qquad n\in2\mathbb{N}_0.
\end{equation}
The finite control through pair $N$ is
\begin{equation*}
 h^{[N]}=\sum_{j=0}^{N}h_{2j}
 \mathbf{1}_{[S_{\tau_{2j}},S_{\tau_{2j+1}}]},
 \qquad
 \rho_t^{[N]}=J_{0,t}\xi-\mathbb{A}_{0,t}
 \bigl(h^{[N]}\mathbf{1}_{[0,S_t]}\bigr).
\end{equation*}
The recursive residual agrees with $\rho^{[N]}$ up to $\tau_{2N+2}$.

\begin{lemma}\label{lem:regularized-low}
Let $p\ge2$, $\kappa\in(0,\kappa_0]$, $\chi>1$, and $\delta_1\in(0,1)$.  For fixed $m$, there exists
$\beta=\beta(p,\kappa,\chi,\delta_1,m)>0$ such that
\begin{equation}\label{eq:regularized-low}
 \mathbb{E}\left[
 \|\Pi_mR_{n,\beta}\|_{\mathbb{L}(\mathbb{H})}^p
 \mid\mathcal{F}_{\tau_n}
 \right]
 \le \delta_1 e^{\tilde{\eta} V(Z_{\tau_n})},
 \qquad n\ge0.
\end{equation}
\end{lemma}

\begin{proof}
For $\varepsilon,\alpha\in(0,1]$, let
\begin{equation*}
 A_{\varepsilon,\alpha}
 =\left\{\inf_{\phi\in\mathbb{U}_{\alpha,m}}
 \langle\mathbb{M}_{\tau_n,\tau_{n+1}}\phi,\phi\rangle\ge\varepsilon\right\}.
\end{equation*}
Write
$\mathbb{M}=\mathbb{M}_{\tau_n,\tau_{n+1}}$,
 $R_\beta=\beta(\mathbb{M}+\beta I)^{-1}$,
$\eta=R_\beta\xi$.
Clearly the case $\eta=0$ is trivial. If $\|\Pi_m\eta\|<\alpha\|\eta\|$, then $\|R_\beta\|\le1$ gives
\begin{equation*}
 \|\Pi_m\eta\|\le\alpha\|\xi\|.
\end{equation*}
Otherwise, $\eta/\|\eta\|\in\mathbb{U}_{\alpha,m}$ and, on
$A_{\varepsilon,\alpha}$,
\begin{equation*}
 \varepsilon\|\eta\|^2\le\langle\mathbb{M}\eta,\eta\rangle
 =\|\mathbb{M}^{1/2}R_\beta\xi\|^2.
\end{equation*}
The spectral theorem gives
\begin{equation*}
 \|\mathbb{M}^{1/2}R_\beta\|
 =\sup_{\lambda\ge0}\frac{\sqrt\lambda\,\beta}{\lambda+\beta}
 \le\frac{\sqrt\beta}{2},
\end{equation*}
and hence
\begin{equation*}
 \|\Pi_m\eta\|\le\|\eta\|
 \le\frac12\sqrt{\frac\beta\varepsilon}\,\|\xi\|.
\end{equation*}
Therefore, on $A_{\varepsilon,\alpha}$, 
\begin{equation}\label{eq:spectral-regularization}
 \|\Pi_mR_{n,\beta}\xi\|
 \le \left(\alpha+\frac12\sqrt{\frac\beta\varepsilon}\right)\|\xi\|,
\end{equation}
while $\|R_{n,\beta}\|\le1$ everywhere. Let $R_*=\frac1{\tilde{\eta}}\log\frac{2}{\delta_1}$. On $\{V(Z_{\tau_n})>R_*\}$, the trivial bound yields
\begin{equation*}
 \mathbb{E}\left[\|\Pi_mR_{n,\beta}\|^p\mid\mathcal{F}_{\tau_n}\right]
 \le1\le\delta_1e^{\tilde{\eta} V(Z_{\tau_n})}.
\end{equation*}

It remains to verify \eqref{eq:regularized-low} on $\{V(Z_{\tau_n})\le R_*\}$.  By the strong Markov property at $\tau_n$ for the future Brownian and subordinator increments, Proposition \ref{prop:covariance} gives
\begin{equation*}
 \mathbb{P}\left(A_{\varepsilon,\alpha}^c\mid\mathcal{F}_{\tau_n}\right)
 \le r(\varepsilon,\alpha,R_*,m)
 \qquad\text{on }\{V(Z_{\tau_n})\le R_*\},
\end{equation*}
where $r$ is as in \eqref{eq:r-eps} and $r\to0$ as $\varepsilon\to0$ for each fixed $\alpha$. Thus, applying \eqref{eq:spectral-regularization} on $A_{\varepsilon,\alpha}$ and $\|R_{n,\beta}\|\le1$ on the complement,
\begin{align*}
 \mathbb{E}\left[\|\Pi_mR_{n,\beta}\|^p\mid\mathcal{F}_{\tau_n}\right]
 &\le C^p\left(\alpha+\sqrt{\frac\beta\varepsilon}\right)^p
 +r(\varepsilon,\alpha,R_*,m)
\quad\text{on }\{V(Z_{\tau_n})\le R_*\}.
\end{align*}
Choose $\alpha>0$ with $C^p(2\alpha)^p<\delta_1/4$, then $\varepsilon>0$ with $r(\varepsilon,\alpha,R_*,m)<\delta_1/4$, and finally $\beta>0$ with $\sqrt{\beta/\varepsilon}\le\alpha$. The last display is then $\le\delta_1/2$, hence $\le\delta_1e^{\tilde{\eta}V(Z_{\tau_n})}$. Combining the two energy regions proves \eqref{eq:regularized-low}.
\end{proof}

For an even $n$, set
\begin{equation*}
 \Gamma_n=\Xi_n+\Xi_{n+1}.
\end{equation*}
Here $\Gamma_n$ denotes the two-block cost, while $\Xi_n$ in \eqref{eq:XnXi} is the one-block quantity. We now present the high-frequency branch estimate.

\begin{proposition}\label{prop:high-branch}
There exist $p_0\ge40$, constants $\mathscr H,\mathfrak a>0$, and $\chi_0>1$, all independent of $m,n$, such that for every $\chi\ge\chi_0$ one may set
\begin{equation}\label{eq:qh-explicit}
 \widehat q(\chi)=\mathscr H e^{-\mathfrak a\chi},
 \qquad \mathfrak a=\frac{p_0\mathfrak c}{2c_V},
\end{equation}
with $\widehat q(\chi)<1$.  For each fixed forcing operator $Q_m$ below, there are constants $0<\widehat c,\widehat C<\infty$, independent of $n$. There are positive sequences $\varepsilon_m\to0$ and $d_m\to0$ such that, for all large $m$ and every even $n$, there exist nonnegative random factors $\widehat U_{n+1},\widehat V_{n+1}$, independent of $\eta$, such that for every $\eta\in\Pi_m^\perp\mathbb{H}$,
\begin{equation}\label{eq:high-branch-closure}
 \|J_{\tau_{n+1},\tau_{n+2}}\eta\|_{\mathbb{H}}^{p_0}
 \le
 \left[
 \widehat q(\chi)\widehat U_{n+1} e^{\varepsilon_m \widehat c\Xi_{n+1}}
 +d_m\widehat V_{n+1} e^{\widehat c\Xi_{n+1}}
 \right]\|\eta\|_{\mathbb{H}}^{p_0},
\end{equation}
 and
\begin{align}
 \mathbb{E}[(\widehat U_{n+1})^4\mid\mathcal{F}_{\tau_{n+1}}]
 &\le \widehat C e^{\varepsilon_m\tilde{\eta} V(Z_{\tau_{n+1}})},\label{eq:Uh-moment}\\
 \mathbb{E}[(\widehat V_{n+1})^4\mid\mathcal{F}_{\tau_{n+1}}]
 &\le \widehat C e^{\tilde{\eta} V(Z_{\tau_{n+1}})}.\label{eq:Vh-moment}
\end{align}
Furthermore, there is a deterministic constant $C$, independent of $m$ and $n$, such that
\begin{equation}\label{eq:eps-dm-rates}
 \varepsilon_m+d_m
 \le C\bigl(\lambda_{m+1}^{-1/9}
 +\lambda_{m+1}^{-1/8}+\lambda_{m+1}^{-1/5}\bigr).
\end{equation}
\end{proposition}

\begin{proof}
Let $s=\tau_{n+1}$, let $\eta\in\Pi_m^\perp\mathbb{H}$, and put
$\rho_r=J_{s,r}\eta$.  Write $\widehat E(r)=\mathcal{W}(\widehat\rho_r)$ and $\widetilde E(r)=\mathcal{W}(\widetilde\rho_r)$.  Set
\begin{equation*}
 \Phi(u)=1+\|u\|_{L^4}^{32/9}.
\end{equation*}
Since $16/5<32/9<4$, all lower powers below are bounded by $C\Phi(u)$.  We again use
\eqref{eq:g-prime-remainder} and separate $g'(0)$ from the remainder.  The constant part preserves the spectral decomposition, while the remainder is bounded by $C(|u|+|u|^2)$.

Put $P=P_m$, $P^\perp=I-P_m$, $\rho^0=\widetilde\rho$, $\rho^1=\widehat\rho$,
and $P^0=P$, $P^1=P^\perp$.  If $\rho^i=(h^i,k^i)$, then
\begin{equation*}
 \partial_r h^i=k^i,\qquad
 \partial_r k^i=-Ak^i-Ah^i-P^iP_0\bigl(g'(u_r)(h^0+h^1)\bigr),\qquad i=0,1.
\end{equation*}
Hence, at the Galerkin level,
\begin{align}\label{eq:projected-modified-energy}
 \frac{\mathrm d}{\mathrm dr}\mathcal{W}(\rho^i)
 &+\|A^{1/2}k^i\|^2+\delta_0\|A^{1/2}h^i\|^2
 -\delta_0\|k^i\|^2+\delta_0(A^{1/2}h^i,A^{1/2}k^i)\notag\\
 &=-\bigl(P^iP_0(g'(u)(h^0+h^1)),k^i\bigr)
   -\delta_0\bigl(P^iP_0(g'(u)(h^0+h^1)),h^i\bigr),\qquad i=0,1.
\end{align}
The left-hand side controls $\mathcal{W}(\rho^i)$ and a fixed fraction of $\|A^{1/2}k^i\|^2$.

For the quadratic part of the low-to-high transfer,
\begin{equation*}
 \|\widehat\rho_2\|_{L^4}
 \le C\|\widehat\rho_2\|^{1/4}\|A^{1/2}\widehat\rho_2\|^{3/4}
 \le C\lambda_{m+1}^{-1/16}
 \|\widehat\rho_2\|^{1/8}\|A^{1/2}\widehat\rho_2\|^{7/8},
\end{equation*}
so Young's inequality with exponents $16/7$ and $16/9$ yields
\begin{align*}
 \int_{\mathbb{T}^3}|u|^2|\widetilde\rho_1||\widehat\rho_2|\mathrm{d} x&\le \frac18\|A^{1/2}\widehat\rho_2\|^2
 +C\lambda_{m+1}^{-1/9}\|u\|_{L^4}^{32/9}
   \|A^{1/2}\widetilde\rho_1\|^{16/9}\|\widehat\rho_2\|^{2/9}\notag\\
 &\le \frac18\|A^{1/2}\widehat\rho_2\|^2
 +C\lambda_{m+1}^{-1/9}\Phi(u)(\widehat E+\widetilde E).
\end{align*}
The linear part of the same transfer is controlled with the same enhanced high-velocity interpolation:
\begin{align*}
 \int_{\mathbb{T}^3}|u||\widetilde\rho_1||\widehat\rho_2|\mathrm{d} x\
 &\le C\lambda_{m+1}^{-1/16}\|u\|_{L^4}\|\widetilde\rho_1\|
   \|\widehat\rho_2\|^{1/8}\|A^{1/2}\widehat\rho_2\|^{7/8}\notag\\
 &\le \frac18\|A^{1/2}\widehat\rho_2\|^2
 +C\lambda_{m+1}^{-1/9}\bigl(1+\|u\|_{L^4}^{16/9}\bigr)(\widehat E+\widetilde E)\notag\\
 &\le \frac18\|A^{1/2}\widehat\rho_2\|^2
 +C\lambda_{m+1}^{-1/9}\Phi(u)(\widehat E+\widetilde E).
\end{align*}
For the quadratic part of the reverse transfer,
$\|\widehat\rho_1\|_{L^4}\le C\lambda_{m+1}^{-1/8}\|A^{1/2}\widehat\rho_1\|$ gives
\begin{align*}
 \int_{\mathbb{T}^3}|u|^2|\widehat\rho_1||\widetilde\rho_2|\mathrm{d} x
 &\le \frac18\|A^{1/2}\widetilde\rho_2\|^2
 +C\lambda_{m+1}^{-1/5}\|u\|_{L^4}^{16/5}\widehat E^{4/5}\widetilde E^{1/5}\notag\\
 &\le \frac18\|A^{1/2}\widetilde\rho_2\|^2
 +C\lambda_{m+1}^{-1/5}\Phi(u)(\widehat E+\widetilde E).
\end{align*}
The linear reverse-transfer term satisfies
\begin{align*}
 \int_{\mathbb{T}^3}|u||\widehat\rho_1||\widetilde\rho_2|\mathrm{d} x
 &\le \|u\|_{L^4}\|\widehat\rho_1\|_{L^4}\|\widetilde\rho_2\|\notag\\
 &\le C\lambda_{m+1}^{-1/8}(1+\|u\|_{L^4}^2)(\widehat E+\widetilde E)
 \le C\lambda_{m+1}^{-1/8}\Phi(u)(\widehat E+\widetilde E).
\end{align*}
The $\delta_0$-pairings in \eqref{eq:projected-modified-energy} produce no worse transfer.  Indeed, the linear and quadratic pieces of \eqref{eq:g-prime-remainder}, with the $L^4$--$L^4$--$L^2$ and $L^2$--$L^4$--$L^4$ H\"older splittings, give
\begin{align*}
 &\left|\bigl(PP_0((g'(u)-g'(0))\widehat\rho_1),\widetilde\rho_1\bigr)\right|
 +\left|\bigl(P^\perp P_0((g'(u)-g'(0))\widetilde\rho_1),\widehat\rho_1\bigr)\right|\notag\\
 &\qquad\le C\Bigl(
 \|u\|_{L^4}\|\widehat\rho_1\|_{L^4}\|\widetilde\rho_1\|
 +\|u\|_{L^4}^2\|\widehat\rho_1\|_{L^4}\|\widetilde\rho_1\|_{L^4}\Bigr)\notag\\
 &\qquad\le C\lambda_{m+1}^{-1/8}\Phi(u)(\widehat E+\widetilde E).
\end{align*}
The same estimate, with the roles of $P$ and $P^\perp$ interchanged, covers all mixed terms created by the cross part of $\mathcal{W}$.
The constant part $g'(0)$ is diagonal in the spectral basis and therefore creates no low--high
transfer.  The high--high terms are exactly those estimated in Lemma \ref{lem:high-squeeze}; after
absorbing the velocity dissipation they contribute
$-\mathfrak c\widehat E+\varepsilon_m\Phi(u)\widehat E$, with $\mathfrak c$ independent of $m$.  Each mixed term above carries
one of the vanishing powers $\lambda_{m+1}^{-1/9}$, $\lambda_{m+1}^{-1/8}$, or
$\lambda_{m+1}^{-1/5}$, so their maximum may be absorbed into one coefficient $d_m\to0$.
The remaining low--low terms require no spectral smallness and are bounded by
$C\Phi(u)\widetilde E$.  Consequently,
\begin{align}
 \frac{\,\mathrm{d}}{\,\mathrm{d} r}\widehat E(r)
 &\le -\mathfrak c\widehat E(r)+\varepsilon_m\Phi(u_r)\widehat E(r)
       +d_m\Phi(u_r)\widetilde E(r),\label{eq:coupled-high}\\
 \frac{\,\mathrm{d}}{\,\mathrm{d} r}\widetilde E(r)
 &\le C\Phi(u_r)\widetilde E(r)+d_m\Phi(u_r)\widehat E(r),\label{eq:coupled-low}
\end{align}
where \eqref{eq:eps-dm-rates} holds and $\widetilde E(s)=0$.

To pass to the energy solution, take $N>m$ and solve the Galerkin variational
equation from $\Pi_N\eta$.  The preceding identities hold with constants independent of $N$ and, after time integration, control
\begin{equation*}
 \sup_{q\in[s,r]}\|\rho_q^N\|_{\mathbb{H}}^2
 +\int_s^r\!\|A^{1/2}\rho_{2,q}^N\|^2\,\mathrm dq.
\end{equation*}
Using \eqref{eq:J-galerkin-convergence}, strong convergence in $C([s,r];\mathbb{H})$ and weak
$L^2(s,r;H^1)$ convergence of the velocities, lower semicontinuity passes the integrated inequalities to the limit; removing the energy localization gives
\eqref{eq:coupled-high}--\eqref{eq:coupled-low}.  Only the time integral of
$\|A^{1/2}\rho_2\|^2$ occurs, so no $A^{1/2}\rho_2(s)$ regularity beyond
$\mathbb{H}=H^1\times H_0$ is used.

Put $R_{a,b}=\int_a^b\Phi(u_r)\mathrm{d} r$.  Since $\varepsilon_m+d_m\le1$ after increasing the minimal rank if necessary, adding \eqref{eq:coupled-high} and \eqref{eq:coupled-low} and discarding the negative term gives
\begin{equation}\label{eq:total-high-low-rough}
 \widehat E(r)+\widetilde E(r)
 \le C\exp\{CR_{s,r}\}\widehat E(s),\qquad r\in[s,t],
\end{equation}
because $\widetilde E(s)=0$.

Variation
of constants in \eqref{eq:coupled-low} gives, for $r\in[s,t]$,
\begin{align}
 \widetilde E(r)
 &\le d_m\int_s^r
 \exp\left\{C\int_q^r\Phi(u_a)\,\mathrm da\right\}
 \Phi(u_q)\widehat E(q)\,\mathrm dq \notag\\
 &\le Cd_m \widehat E(s)e^{CR_{s,r}}\int_s^r\Phi(u_q)\,\mathrm dq \notag\\
 &\le Cd_m R_{s,r}e^{CR_{s,r}}\widehat E(s)
 \le Cd_me^{C_1R_{s,r}}\widehat E(s),
 \label{eq:leakage-energy}
\end{align}
where \eqref{eq:total-high-low-rough} was used in the second line and
$x e^{Cx}\le C'e^{C_1x}$ for $x\ge0$ in the last one.

Apply variation of constants also to \eqref{eq:coupled-high}.  We obtain
\begin{align*}
 \widehat E(t)
 &\le e^{-\mathfrak c(t-s)+\varepsilon_mR_{s,t}}\widehat E(s)+d_m\int_s^t
 e^{-\mathfrak c(t-r)+\varepsilon_mR_{r,t}}
 \Phi(u_r)\widetilde E(r)\,\mathrm dr.
\end{align*}
Using \eqref{eq:leakage-energy}, dropping the favorable factor $e^{-\mathfrak c(t-r)}$, and using
$R_{s,r}+R_{r,t}=R_{s,t}$ yields
\begin{equation*}
 d_m\int_s^t
 e^{-\mathfrak c(t-r)+\varepsilon_mR_{r,t}}
 \Phi(u_r)\widetilde E(r)\,\mathrm dr
 \le Cd_m^2R_{s,t}e^{C_2R_{s,t}}\widehat E(s).
\end{equation*}
For all sufficiently large $m$, $d_m\le1$, so the last term is bounded by
$Cd_me^{C_3R_{s,t}}\widehat E(s)$.  Combining this estimate with \eqref{eq:leakage-energy} at $r=t$ gives
\begin{equation}\label{eq:full-high-energy}
 \widehat E(t)+\widetilde E(t)
 \le C\left[e^{-\mathfrak c(t-s)+\varepsilon_mR_{s,t}}+d_me^{C R_{s,t}}\right]\widehat E(s).
\end{equation}
Raise \eqref{eq:full-high-energy} to the power $p_0/2$ and use norm equivalence.  Since
$\tau_{n+2}-\tau_{n+1}\ge\chi/c_V$, the deterministic part of the first branch is bounded by
$\mathscr H e^{-p_0\mathfrak c\chi/(2c_V)}$, which is \eqref{eq:qh-explicit}.  Increase a deterministic $\chi_0>1$, if necessary, so that $\mathscr H e^{-\mathfrak a\chi_0}<1$.  Then $\widehat q(\chi)<1$ for every $\chi\ge\chi_0$.  Write
\begin{equation*}
 R_{n+1}=\int_{\tau_{n+1}}^{\tau_{n+2}}\Phi(u_r)\mathrm{d} r
\end{equation*}
and choose constants $c_1,c_2>0$ from Lemma \ref{lem:normalized-block-exp}.  Up to fixed norm-equivalence constants, one may take
\begin{align*}
 \widehat U_{n+1}
 &=\exp\left\{\frac{p_0}{2}\varepsilon_mR_{n+1}-c_1\varepsilon_m\Xi_{n+1}\right\},\\
 \widehat V_{n+1}
 &=\exp\left\{\frac{p_0C}{2}R_{n+1}-c_2\Xi_{n+1}\right\}.
\end{align*}
Choose these constants from Lemma \ref{lem:normalized-block-exp} using the smaller state exponent $\tilde{\eta}$.  For $\widehat U_{n+1}$, take
\begin{equation*}
 p=4,\qquad r=32/9,\qquad
 \theta=\frac{\varepsilon_m\tilde{\eta}}{\kappa},\qquad
 \Lambda=\frac{p_0\kappa}{2\tilde{\eta}},
\end{equation*}
and for $\widehat V_{n+1}$ take
\begin{equation*}
 p=4,\qquad r=32/9,\qquad
 \theta=\frac{\tilde{\eta}}{\kappa},\qquad
 \Lambda=\frac{p_0C\kappa}{2\tilde{\eta}}.
\end{equation*}
Let $\widehat c$ and $\widehat C$ dominate the corresponding normalization coefficients and moment constants.  Then
\begin{equation*}
 \mathbb{E}[(\widehat U_{n+1})^4\mid\mathcal{F}_{\tau_{n+1}}]
 \le \widehat C e^{\varepsilon_m\tilde{\eta} V(Z_{\tau_{n+1}})},\qquad
 \mathbb{E}[(\widehat V_{n+1})^4\mid\mathcal{F}_{\tau_{n+1}}]
 \le \widehat C e^{\tilde{\eta} V(Z_{\tau_{n+1}})},
\end{equation*}
which are exactly \eqref{eq:Uh-moment}--\eqref{eq:Vh-moment}.  For the leakage branch, the factor produced directly by
\eqref{eq:full-high-energy} is $d_m^{p_0/2}$; since $d_m\le1$ for large $m$, it may be replaced by the larger quantity $d_m$.  This proves \eqref{eq:high-branch-closure}.  The operator bounds are uniform in $\eta\in\Pi_m^\perp\mathbb{H}$.
\end{proof}

\begin{lemma}\label{lem:controlled-variational-factor}
For the same $p_0$ as in Proposition \ref{prop:high-branch}, there are nonnegative random variables $U_n^c$ such that
\begin{equation}\label{eq:controlled-variational-factor}
 \|J_{\tau_n,\tau_{n+1}}\|_{\mathbb{L}(\mathbb{H})}^{p_0}
 \le U_n^c e^{\widehat c\Xi_n},
\end{equation}
and
\begin{equation}\label{eq:controlled-variational-factor-moment}
 \mathbb{E}[(U_n^c)^4\mid\mathcal{F}_{\tau_n}]
 \le \widehat C e^{\tilde{\eta} V(Z_{\tau_n})}.
\end{equation}
\end{lemma}

\begin{proof}
Use only the first-variation bound \eqref{eq:J-subcritical}.
Lemma~\ref{lem:normalized-block-exp}, with $r=16/5$ and the state exponent
$\tilde\eta$, gives constants $c,C$ such that
\[
 \mathbb E[\|J_{\tau_n,\tau_{n+1}}\|^{4p_0}e^{-c\Xi_n}
                    \mid\mathcal F_{\tau_n}]
 \le C e^{\tilde\eta V(Z_{\tau_n})}.
\]
Increase $\widehat c$ and $\widehat C$ to dominate $c/4$ and $C$, respectively,
and set $U_n^c=\|J_{\tau_n,\tau_{n+1}}\|^{p_0}e^{-\widehat c\Xi_n}$.
Both assertions follow. Constants for second variations and
derivatives of the controls are chosen separately and do not enlarge $\widehat c$.
\end{proof}

\begin{lemma}\label{lem:one-pair}
Fix $\kappa\in(0,\kappa_0]$ and $\chi>1$, and let the factors be those appearing in
Proposition \ref{prop:high-branch} and Lemma \ref{lem:controlled-variational-factor}. There exists a constant $\mathcal{Q}=\mathcal{Q}(g,Q_m,S,p_0,\kappa,\chi,\tilde\eta)<\infty$, independent of $n,\delta_2$ and $\beta$, such that for every
$\delta_2\in(0,1)$, one can choose
$\beta=\beta(p_0,\kappa,\chi,\delta_2,m)>0$ so that for every even $n$, there are nonnegative random variables
$\vartheta_n^{(0)},\vartheta_n^{(1)},\vartheta_n^{(2)}$ satisfying
\begin{align}
 &\|\rho_{\tau_{n+2}}\|_{\mathbb{H}}^{p_0}
 \le C e^{\widehat c\Gamma_n}
 \bigl[\vartheta_n^{(0)}+\vartheta_n^{(1)}+\vartheta_n^{(2)}\bigr]
 \|\rho_{\tau_n}\|_{\mathbb{H}}^{p_0},\label{eq:one-pair-path}\\
 &\mathbb{E}[\vartheta_n^{(0)}\mid\mathcal{F}_{\tau_n}]
 \le\delta_2e^{\tilde{\eta} V(Z_{\tau_n})},\label{eq:one-pair-low}\\
 &\mathbb{E}[\vartheta_n^{(1)}\mid\mathcal{F}_{\tau_n}]
 \le \widehat q(\chi)\mathcal{Q} e^{\tilde{\eta} V(Z_{\tau_n})},\\
 &\mathbb{E}[\vartheta_n^{(2)}\mid\mathcal{F}_{\tau_n}]
 \le d_m\mathcal{Q} e^{\tilde{\eta} V(Z_{\tau_n})}.\label{eq:one-pair-high2}
\end{align}
\end{lemma}
\begin{proof}
	Throughout this proof, let $n\in2\mathbb{N}_0$.  Write
	\begin{equation*}
	J_n=J_{\tau_n,\tau_{n+1}},\qquad
	J_{n+1}=J_{\tau_{n+1},\tau_{n+2}},\qquad
	\eta_n=R_{n,\beta}J_n\rho_{\tau_n},
	\end{equation*}
	and decompose $\eta_n=\widetilde\eta_n+\widehat\eta_n$ with
	$\widetilde\eta_n=\Pi_m\eta_n$ and $\widehat\eta_n=\Pi_m^\perp\eta_n$.  By
	\eqref{eq:res-recursion} we have
	\begin{equation*}
	\rho_{\tau_{n+2}}=J_{n+1}\widetilde\eta_n+J_{n+1}\widehat\eta_n.
	\end{equation*}
	
	We first treat the low branch.  Applying Lemma \ref{lem:controlled-variational-factor} to both blocks yields normalized factors $U_n^c,U_{n+1}^c$ such that
	\begin{equation*}
	\|J_{n+1}\widetilde\eta_n\|^{p_0}
	\le e^{\widehat c\Gamma_n}U_{n+1}^cU_n^c
	\|\Pi_mR_{n,\beta}\|^{p_0}\|\rho_{\tau_n}\|^{p_0}.
	\end{equation*}
	Take $p=2p_0$ in Lemma~\ref{lem:regularized-low}.  By \eqref{eq:controlled-variational-factor-moment} we obtain
	\begin{equation*}
	\mathbb{E}[U_{n+1}^c\mid\mathcal{F}_{\tau_{n+1}}]
	\le \widehat C^{1/4} e^{\tilde{\eta} V(Z_{\tau_{n+1}})/4}.
	\end{equation*}
	The remaining factors are $\mathcal{F}_{\tau_{n+1}}$-measurable. Applying conditional H\"older at
	$\tau_n$, together with \eqref{eq:regularized-low} and \eqref{eq:stopped-recursion}, yields
	\begin{align*}
	&~~\mathbb{E}\left[
	U_{n+1}^cU_n^c\|\Pi_mR_{n,\beta}\|^{p_0}
	\mid\mathcal{F}_{\tau_n}\right]\notag\\
	&\le C
	\mathbb{E}[(U_n^c)^4\mid\mathcal{F}_{\tau_n}]^{1/4}
	\mathbb{E}[\|\Pi_mR_{n,\beta}\|^{2p_0}\mid\mathcal{F}_{\tau_n}]^{1/2}
	\mathbb{E}[e^{\tilde{\eta} V(Z_{\tau_{n+1}})}\mid\mathcal{F}_{\tau_n}]^{1/4}\notag\\
	&\le C\delta_1^{1/2}
	\exp\left\{\tilde{\eta}\left(\frac34+\frac{e^{-\chi}}4\right)
	V(Z_{\tau_n})\right\}
	\le C\delta_1^{1/2}e^{\tilde{\eta} V(Z_{\tau_n})}.
	\end{align*}
	Choose $\delta_1$ so that $C\delta_1^{1/2}\le\delta_2$, and define
	\begin{equation*}
	\vartheta_n^{(0)}
	=C U_{n+1}^cU_n^c\|\Pi_mR_{n,\beta}\|^{p_0}
	\end{equation*}
	which shows \eqref{eq:one-pair-low}.
	
	For the high branch, Proposition \ref{prop:high-branch} gives
	\begin{equation*}
	\|J_{n+1}\widehat\eta_n\|^{p_0}
	\le\left[\widehat q(\chi)\widehat U_{n+1}e^{\varepsilon_m \widehat c\Xi_{n+1}}
	+d_m\widehat V_{n+1}e^{\widehat c\Xi_{n+1}}\right]\|\widehat\eta_n\|^{p_0}.
	\end{equation*}
	Since $\|R_{n,\beta}\|\le1$, Lemma \ref{lem:controlled-variational-factor} yields
	\begin{equation*}
	\|\widehat\eta_n\|^{p_0}
	\le U_n^ce^{\widehat c\Xi_n}\|\rho_{\tau_n}\|^{p_0}.
	\end{equation*}
	Thus, since $\varepsilon_m\le1$ for large $m$, the two high contributions are bounded by
	$e^{\widehat c\Gamma_n}$ times
	\begin{equation*}
	\vartheta_n^{(1)}=C\widehat q(\chi)U_n^c\widehat U_{n+1},
	\qquad
	\vartheta_n^{(2)}=Cd_mU_n^c\widehat V_{n+1}.
	\end{equation*}
	It follows from
	\eqref{eq:Uh-moment}--\eqref{eq:Vh-moment} that
	\begin{equation*}
	\mathbb{E}[\widehat U_{n+1}\mid\mathcal{F}_{\tau_{n+1}}]
	\le \widehat C^{1/4} e^{\varepsilon_m\tilde{\eta} V(Z_{\tau_{n+1}})/4},
	\qquad
	\mathbb{E}[\widehat V_{n+1}\mid\mathcal{F}_{\tau_{n+1}}]
	\le \widehat C^{1/4} e^{\tilde{\eta} V(Z_{\tau_{n+1}})/4}.
	\end{equation*}
	A second conditional H\"older inequality at $\tau_n$, together with
	\eqref{eq:controlled-variational-factor-moment} and the fixed-block exponent reduction
	\eqref{eq:fixed-block-exp-reduction}, the latter applied at exponents $\varepsilon_m\tilde{\eta}/3$ and $\tilde{\eta}/3$
	without changing the stopping times, gives
	\begin{equation*}
	\mathbb{E}[U_n^c\widehat U_{n+1}\mid\mathcal{F}_{\tau_n}]
	+\mathbb{E}[U_n^c\widehat V_{n+1}\mid\mathcal{F}_{\tau_n}]
	\le \mathcal{Q} e^{\tilde{\eta} V(Z_{\tau_n})},
	\end{equation*}
	after enlarging $\mathcal{Q}(g,Q_m,S,p_0,\kappa,\chi,\tilde\eta)$ so as to absorb the stopped-recursion prefactor
	$\mathscr S_{\kappa,\chi}$ and the height-independent variational constants. Consequently,
	\begin{equation*}
	\mathbb{E}[\vartheta_n^{(1)}\mid\mathcal{F}_{\tau_n}]
	\le \widehat q(\chi)\mathcal{Q} e^{\tilde{\eta} V(Z_{\tau_n})},
	\qquad
	\mathbb{E}[\vartheta_n^{(2)}\mid\mathcal{F}_{\tau_n}]
	\le d_m\mathcal{Q} e^{\tilde{\eta} V(Z_{\tau_n})}.
	\end{equation*}
Finally, \eqref{eq:one-pair-path}--\eqref{eq:one-pair-high2} follow from the inequality
$\|x+y\|^{p_0}\le2^{p_0-1}(\|x\|^{p_0}+\|y\|^{p_0})$.
\end{proof}

Since $x\mapsto x^{\theta_0}$ is concave, \eqref{eq:one-pair-path} and conditional Jensen's inequality imply
\begin{align}
 &\|\rho_{\tau_{n+2}}\|_{\mathbb{H}}^{4}
 \le C_1e^{\theta_0\widehat c\Gamma_n}
 \bigl[(\vartheta_n^{(0)})^{\theta_0}+(\vartheta_n^{(1)})^{\theta_0}+(\vartheta_n^{(2)})^{\theta_0}\bigr]
 \|\rho_{\tau_n}\|_{\mathbb{H}}^{4},\label{eq:one-pair-fourth}\\
 &\mathbb{E}[(\vartheta_n^{(0)})^{\theta_0}\mid\mathcal{F}_{\tau_n}]
 \le\delta_2^{\theta_0}e^{\theta_0\tilde{\eta} V(Z_{\tau_n})},\notag\\
 &\mathbb{E}[(\vartheta_n^{(1)})^{\theta_0}\mid\mathcal{F}_{\tau_n}]
 \le(\widehat q(\chi)\mathcal{Q})^{\theta_0}e^{\theta_0\tilde{\eta} V(Z_{\tau_n})},\notag\\
 &\mathbb{E}[(\vartheta_n^{(2)})^{\theta_0}\mid\mathcal{F}_{\tau_n}]
 \le(d_m\mathcal{Q})^{\theta_0}e^{\theta_0\tilde{\eta} V(Z_{\tau_n})}.\label{eq:one-pair-fourth-moments}
\end{align}
Here $C_1$ is independent of $n$, the localization intercept, and the initial state. To iterate \eqref{eq:one-pair-fourth}, use second moments rather than multiplying conditional first moments. Since $p_0\ge40$, we have $2\theta_0\le1$. If
\begin{equation*}
 B_{n,0}=(\vartheta_n^{(0)})^{\theta_0},\qquad
 B_{n,1}=(\vartheta_n^{(1)})^{\theta_0},\qquad
 B_{n,2}=(\vartheta_n^{(2)})^{\theta_0},
\end{equation*}
then conditional Jensen's inequality strengthens \eqref{eq:one-pair-fourth-moments} to
\begin{equation}\label{eq:reduced-second-moments}
 \mathbb{E}[B_{n,\varsigma}^2\mid\mathcal{F}_{\tau_n}]
 \le \widehat a_\varsigma^2e^{2\theta_0\tilde{\eta} V(Z_{\tau_n})},
 \qquad \varsigma\in\{0,1,2\},
\end{equation}
where, by conditional Jensen, there is a deterministic constant $C_3\ge1$, independent of $n$ and of the numerical tolerance $\delta_2$, such that
\begin{equation*}
 \widehat a_0\le C_3\delta_2^{\theta_0},\qquad
\widehat a_1\le C_3\widehat q(\chi)^{\theta_0},\qquad
 \widehat a_2\le C_3d_m^{\theta_0}.
\end{equation*}
For $\varsigma=(\varsigma_0,\dots,\varsigma_{n-1})$, define
$\widehat B_{2j,\varsigma_j}=B_{2j,\varsigma_j}
 e^{-\theta_0\tilde{\eta} V(Z_{\tau_{2j}})}$.  By \eqref{eq:reduced-second-moments} and successive conditioning,
\begin{equation*}
 \mathbb{E}\left[\prod_{j=0}^{n-1}\widehat B_{2j,\varsigma_j}^{\,2}\right]
 \le\prod_{j=0}^{n-1}\widehat a_{\varsigma_j}^{2}.
\end{equation*}
By Cauchy--Schwarz, for a deterministic constant $C_2\ge1$ independent of $n$, we obtain
\begin{align}
 \mathbb{E}\left[\prod_{j=0}^{n-1}B_{2j,\varsigma_j}\right]
 &\le \prod_{j=0}^{n-1}\widehat a_{\varsigma_j}
 \left(\mathbb{E}\left[\exp\left\{2\theta_0\tilde{\eta}
 \sum_{j=0}^{n-1}V(Z_{\tau_{2j}})\right\}\right]\right)^{1/2}\notag\\
 &\le C_2^{\,n}e^{C\tilde{\eta} V(z_0)}
 \prod_{j=0}^{n-1}\widehat a_{\varsigma_j}.\label{eq:weighted-word-product}
\end{align}
The last step is \eqref{eq:fixed-block-sum-reduction} with
$\eta=2\theta_0\tilde{\eta}\le\bar\eta$, using
$\sum_{j<n}V(Z_{\tau_{2j}})\le V(z_0)+\sum_{r=1}^{2n}V(Z_{\tau_r})$. 

Set $\widehat\mu=\overline\Xi+\widehat\delta$ and $\mathscr C=C_1C_2C_3\ge1$. With $g,Q_m,S,p_0,\kappa,\chi,\tilde\eta$ fixed, these constants and
$\widehat c$ are independent of $n,L,K,\delta_2,\beta$. They are evaluated before
the covariance tolerance and the Tikhonov parameter are chosen.
The constant $C_2$ comes from \eqref{eq:fixed-block-sum-reduction} at exponent
$2\theta_0\tilde\eta$.

\begin{proposition}\label{prop:block-contraction}
Fix $\widehat\delta>0$ and assume the hypotheses of Proposition \ref{prop:high-branch}.  Let $m$ be sufficiently large and suppose
\begin{equation}\label{eq:fixed-rank-admissible}
 \mathfrak q_m(\kappa,\chi,\tilde\eta,\widehat\delta)
 =\mathscr Ce^{2\theta_0\widehat c\widehat\mu}
 \bigl(\widehat q(\chi)^{\theta_0}+d_m^{\theta_0}\bigr)<\frac14.
\end{equation}
Then $\beta>0$ may be chosen independently of $L$ so that, for some $\widehat\gamma>0$ and every $L>0$,
\begin{equation}\label{eq:block-contraction}
 \mathbb{E}_{z_0}\left[\|\rho_{\tau_{2n}}\|_{\mathbb{H}}^{4}
 \mathbf{1}_{\{\widehat\Theta^{(\widehat\delta)}\le L\}}\right]
 \le C_Le^{C\kappa V(z_0)-\widehat\gamma n},\qquad n\ge0.
\end{equation}
Moreover, there exists $\widetilde c>0$ such that, if
$0<b<\widehat\gamma/(8\widetilde c)$, then for every $K>0$, $n\ge0$, and
deterministic $t\ge0$,
\begin{equation}\label{eq:rho-between-blocks}
 \mathbb{E}_{z_0}\left[\|\rho_t\|_{\mathbb{H}}^{4}
 \mathbf{1}_{\mathbb{G}_{L,K}^{\widehat\delta,b}}
 \mathbf{1}_{\{t\in[\tau_{2n},\tau_{2n+2})\}}\right]
 \le C_{L,K}e^{C\kappa V(z_0)-\widehat\gamma n/2}.
\end{equation}
The constants $\widehat\gamma$, $\widetilde c$, and $\beta$ are independent of $L$ and $K$.
\end{proposition}

\begin{proof}
On $\{\widehat\Theta^{(\widehat\delta)}\le L\}$,
\begin{equation*}
 \sum_{j=0}^{n-1}\Gamma_{2j}=\sum_{r=0}^{2n-1}\Xi_r
 \le2\widehat\mu n+L.
\end{equation*}
Choose $\delta_2$ so that
\begin{equation*}
 \mathscr Ce^{2\theta_0\widehat c\widehat\mu}\delta_2^{\theta_0}<\frac14.
\end{equation*}
This choice, and hence the corresponding $\beta$, is independent of $L$.  Iterating
\eqref{eq:one-pair-fourth} gives
\begin{equation*}
 \|\rho_{\tau_{2n}}\|_{\mathbb{H}}^{4}
 \le C_1^{n}e^{\theta_0\widehat c\sum_{j=0}^{n-1}\Gamma_{2j}}
 \sum_{\varsigma\in\{0,1,2\}^n}
 \prod_{j=0}^{n-1}B_{2j,\varsigma_j}.
\end{equation*}
Using the preceding clock bound and \eqref{eq:weighted-word-product},
\begin{align*}
 &\mathbb{E}\left[\|\rho_{\tau_{2n}}\|_{\mathbb{H}}^{4}
 \mathbf{1}_{\{\widehat\Theta^{(\widehat\delta)}\le L\}}\right]\le C_Le^{C\kappa V(z_0)}
 \left[C_1C_2e^{2\theta_0\widehat c\widehat\mu}(\widehat a_0+\widehat a_1+\widehat a_2)\right]^n.
\end{align*}
Since
\begin{equation*}
 C_1C_2e^{2\theta_0\widehat c\widehat\mu}(\widehat a_0+\widehat a_1+\widehat a_2)
 \le \mathscr Ce^{2\theta_0\widehat c\widehat\mu}
 \bigl(\delta_2^{\theta_0}+\widehat q(\chi)^{\theta_0}+d_m^{\theta_0}\bigr)
 =:\bar q<\frac12,
\end{equation*}
\eqref{eq:block-contraction} follows with $\widehat\gamma=-\log\bar q$.

For the estimate inside a pair, retain a terminal energy weight. Define the past-measurable event
\begin{equation*}
 E_n^L=\left\{\sum_{j=0}^{2n-1}\Xi_j\le2\widehat\mu n+L\right\},
 \qquad \eta_0=\theta_0\tilde\eta.
\end{equation*}
The same word-product estimate gives
\begin{equation}\label{eq:weighted-endpoint-residual}
 \mathbb E\left[\|\rho_{\tau_{2n}}\|^4e^{\eta_0V(Z_{\tau_{2n}})}
                   \mathbf1_{E_n^L}\right]
 \le C_Le^{C\kappa V(z_0)-\widehat\gamma n}.
\end{equation}
Indeed, after Cauchy--Schwarz the energy factor is bounded by
\begin{equation*}
 \mathbb E\exp\left\{2\theta_0\tilde\eta
 \left(\sum_{j=0}^{n-1}V(Z_{\tau_{2j}})+V(Z_{\tau_{2n}})\right)\right\}
 \le e^{2\theta_0\tilde\eta V(z_0)}
 \mathbb E e^{2\theta_0\tilde\eta\sum_{r=1}^{2n}V(Z_{\tau_r})}.
\end{equation*}
This is precisely the bound used to choose $C_2$ in \eqref{eq:weighted-word-product}, so the contraction factor is unchanged.

By \eqref{eq:h-L2} and $\|R_{2n,\beta}\|\le1$, throughout the current pair,
\begin{equation}\label{eq:current-pair-pathwise}
 \|\rho_t\|^4\le C_\beta H_n\|\rho_{\tau_{2n}}\|^4,
 \qquad H_n=\mathbb J_{2n}^8\mathbb J_{2n+1}^4.
\end{equation}

Apply \eqref{eq:variational-block-moment} to the odd block with
$(p,\eta_V)=(8,\eta_0)$ and to the even block with $(p,\eta_V)=(32,\eta_0)$.
Choose $\widetilde c$ at least the odd normalization coefficient and half the even one.
Conditioning first at $\tau_{2n+1}$ and then using conditional Cauchy--Schwarz gives
\begin{align}
 &~~\mathbb E[H_n^2e^{-\widetilde c(\Xi_{2n}+\Xi_{2n+1})}
                  \mid\mathcal F_{\tau_{2n}}]\notag\\
 &\le C\mathbb E[\mathbb J_{2n}^{32}e^{-2\widetilde c\Xi_{2n}}
                     \mid\mathcal F_{\tau_{2n}}]^{1/2}
       \mathbb E[e^{2\eta_0V(Z_{\tau_{2n+1}})}
                     \mid\mathcal F_{\tau_{2n}}]^{1/2}\notag\\
 &\le\widetilde C
       e^{(\eta_0/2+\eta_0e^{-\chi})V(Z_{\tau_{2n}})}
 \le\widetilde C e^{2\eta_0V(Z_{\tau_{2n}})}.
 \label{eq:current-two-block-normalized}
\end{align}
Here $2\eta_0\le\eta_*$, so the fixed-block exponent reduction applies.

On $\mathbb G_{L,K}^{\widehat\delta,b}$, the event $E_n^L$ holds and
$\Xi_{2n}+\Xi_{2n+1}\le2K+(4n+1)b$. Since $E_n^L$ and $\rho_{\tau_{2n}}$ are $\mathcal F_{\tau_{2n}}$-measurable, conditional Cauchy--Schwarz yields
\begin{align*}
 &~~\mathbb E\left[\|\rho_t\|^4\mathbf1_{\mathbb G_{L,K}^{\widehat\delta,b}}
               \mathbf1_{\{t\in[\tau_{2n},\tau_{2n+2})\}}\right]\\
 &\le C_{K,\beta}e^{\widetilde c(4n+1)b/2}
 \mathbb E\left[\|\rho_{\tau_{2n}}\|^4e^{\eta_0V(Z_{\tau_{2n}})}
                  \mathbf1_{E_n^L}\right]\\
 &\le C_{L,K}e^{C\kappa V(z_0)-\widehat\gamma n+\widetilde c(4n+1)b/2}.
\end{align*}
The stated choice $b<\widehat\gamma/(8\widetilde c)$ implies \eqref{eq:rho-between-blocks}.
\end{proof}

\begin{remark}\label{rem:fixed-rank-interpretation}
The constants in \eqref{eq:fixed-rank-admissible} are those fixed by the
preceding estimates. Although $d_m\to0$, the noise size, the admissible
$\kappa$, and the normalized moment constants can change with $Q_m$.
Thus the spectral limit alone does not verify the condition for an arbitrary
family of forcing operators. Under the hypotheses of
Proposition~\ref{prop:clock-mean-bound}, a sufficient condition is
\[
 \mathscr C\exp\!\left\{2\theta_0\widehat c
       (\overline\Xi^{\sharp}(\kappa,\chi,B_Q)+\widehat\delta)\right\}
       (\widehat q(\chi)^{\theta_0}+d_m^{\theta_0})<\tfrac14.
\]
This replaces the clock expectation by its explicit upper bound; all other
constants must still be evaluated for the selected equation and forcing.
\end{remark}

\subsection{Blockwise Skorohod estimates and the e-property}

After conditioning on a clock path $S=\ell$, the controlled operational intervals
$ I_n^c(\ell)=[\ell_{\tau_{2n}},\ell_{\tau_{2n+1}}]$ are deterministic and pairwise disjoint.  We work on the fixed Cameron--Martin space
$\mathcal{H}=L^2(\mathbb{R}_+;\mathbb{R}^m)$ and identify each block space by zero extension.  Let
\begin{equation*}
 (\mathsf P_n(\ell)h)(r)=\mathbf{1}_{I_n^c(\ell)}(r)h(r),\qquad h\in\mathcal{H}.
\end{equation*}
For finite $N$,
\begin{equation}\label{eq:partial-wiener-decomposition}
 \mathcal{H}=\left(\bigoplus_{n=0}^{N}\mathcal{H}_n^c(\ell)\right)
 \oplus\mathcal{H}_{N,\mathrm{rest}}(\ell).
\end{equation}
Fix a countable dense set $\mathcal{H}_{\mathbb{Q}}\subset\mathcal{H}$ and define the external-factor
sigma-field for the $n$th controlled block by
\begin{equation}\label{eq:external-factor-sigma}
 \mathcal{E}_n
 =\sigma(S)\vee
 \sigma\left\{W\bigl((I-\mathsf P_n(S))h\bigr):h\in\mathcal{H}_{\mathbb{Q}}\right\}.
\end{equation}
Conditionally on $S=\ell$, $\mathcal{E}_n$ contains the complementary Gaussian factor
$\mathcal{H}_n^c(\ell)^\perp$ and is therefore constant along the current Gaussian factor
$\mathcal{H}_n^c(\ell)$.  In particular, bounded $\mathcal{E}_n$-measurable multipliers have zero
partial derivative $D^{(n)}$; this includes all clock events and all random variables measurable
with respect to Brownian increments preceding the current controlled block.

For the pair index used in this subsection we set $D^{(n)}:=D^{[2n]}$ and identify $\mathcal{H}_n^c=\mathcal{H}_{2n}$ by zero extension to $\mathcal{H}$.

On jointly measurable Brownian cylindrical functionals set
\begin{equation*}
 D^{(n)}F=\mathsf P_n(S)D_WF,\qquad
 \|F\|_{1,2;(n)}^2=\mathbb{E}[|F|^2]+\mathbb{E}[\|D^{(n)}F\|_{\mathcal{H}}^2].
\end{equation*}
Its closure is denoted by $\mathbb{D}_{(n)}^{1,2}$.  For any separable Hilbert target space $E$,
the space $\mathbb{D}_{(n)}^{1,2}(E)$ is defined by closing finite-rank cylindrical $E$-valued
fields in the corresponding graph norm.  The localized space
$\mathbb{D}_{(n),\mathrm{loc}}^{1,2}(E)$ consists of fields $H$ for which there are
$\mathcal{E}_n$-measurable sets $A_k\uparrow\Omega$ such that
$\mathbf{1}_{A_k}H\in\mathbb{D}_{(n)}^{1,2}(E)$ for every $k$.  We use this with
$E=\mathcal{H}$ and $E=\mathbb{H}$.  For fixed $\ell$ this is the usual Malliavin derivative on the
Gaussian factor $\mathcal{H}_n^c(\ell)$, with the complementary Gaussian factor treated as a
parameter; the corresponding divergence on $\mathcal{H}_n^c(\ell)$ is $\delta_n^\ell$.  The resulting closed divergence operator is used in the blockwise integration-by-parts argument below.

The field of projections is strongly measurable because
\[
 \langle\mathsf P_n(\ell)h,k\rangle
 =\int_{\ell_{\tau_{2n}}}^{\ell_{\tau_{2n+1}}}h(r)\cdot k(r)\,\mathrm dr
\]
is measurable and $\mathcal H$ is separable. Closing the cylindrical core
$a(S)\varphi(W(h_1),\ldots,W(h_k))$ and using Fubini identifies $D^{(n)}$
with the direct integral of the Gaussian derivatives on
$\mathcal H_n^c(\ell)$. Its adjoint is denoted by $\delta_n$.

\begin{lemma}\label{lem:ray-sobolev}
Let $E$ be a separable Hilbert space. Suppose an $E$-valued field $U$ is locally
absolutely continuous along Cameron--Martin lines in the current Gaussian factor,
with directional derivative $\langle F,h\rangle_{\mathcal H_n^c}$, where
$F\in\mathcal H_n^c\otimes E$ is measurable. If
\[
 \mathbb E(\|U\|_E^2+\|F\|_{\mathcal H_n^c\otimes E}^2)<\infty,
\]
then $U\in\mathbb D_{(n)}^{1,2}(E)$ and $D^{(n)}U=F$.
It suffices to check the line property on a countable dense Cameron--Martin core.
\end{lemma}
\begin{proof}
Freeze the clock and complementary factor. Choose an orthonormal basis of the
current factor and disintegrate along each Gaussian coordinate. One-dimensional
integration by parts, first against compactly supported smooth tests, identifies
the weak derivative with the corresponding coordinate of $F$.
Square integrability permits removal of the test cutoffs. Applying this to a
countable dense set of scalar products in $E$, then summing the squared coordinate
bounds and integrating over the external factors, gives the closed Gaussian
Sobolev derivative; see \cite[Chapter~1]{Nualart2006}.
\end{proof}

\begin{lemma}\label{lem:conditional-block-divergence}
Fix $n$ and let $\mathcal{E}_n$ be defined by \eqref{eq:external-factor-sigma}.  If
$\Lambda$ is bounded and $\mathcal{E}_n$-measurable, then
\begin{equation*}
 \Lambda\in\mathbb{D}_{(n)}^{1,2},\qquad D^{(n)}\Lambda=0.
\end{equation*}
Let $H\in\mathbb{D}_{(n),\mathrm{loc}}^{1,2}(\mathcal{H})$ satisfy
$H=\mathsf P_n(S)H$, and let $\delta_n^{\mathrm{loc}}(H)$ be the local divergence obtained on the
current Gaussian factor after conditioning on the clock and the complementary factors.  If $F\in\mathbb D_{(n)}^{1,2}$ is bounded and
\begin{equation}\label{eq:localized-divergence-integrability}
 \mathbb{E}\!\left[|\Lambda|^2
 \left(\|H\|_{\mathcal{H}}^2+
 \|D^{(n)}H\|_{\mathcal{H}\otimes\mathcal{H}}^2\right)\right]<\infty,
\end{equation}
then $\Lambda H\in\operatorname{Dom}\delta_n$ and
\begin{equation}\label{eq:clock-multiplier-divergence}
 \delta_n(\Lambda H)=\Lambda\delta_n^{\mathrm{loc}}(H)
\end{equation}
in $L^2$.  Consequently,
\begin{equation}\label{eq:conditional-disintegration-ibp}
 \mathbb{E}\!\left[\Lambda\langle D^{(n)}F,H\rangle_{\mathcal{H}}\right]
 =\mathbb{E}\!\left[\Lambda F\,\delta_n^{\mathrm{loc}}(H)\right],
\end{equation}
and
\begin{equation}\label{eq:partial-divergence-L2-estimate}
 \mathbb{E}[|\Lambda\delta_n^{\mathrm{loc}}(H)|^2]
 \le \mathbb{E}\!\left[|\Lambda|^2
 \left(\|H\|_{\mathcal{H}}^2+
 \|D^{(n)}H\|_{\mathcal{H}\otimes\mathcal{H}}^2\right)\right].
\end{equation}
\end{lemma}

\begin{proof}
Fix the clock path $S=\ell$.  Then
\begin{equation*}
 \mathcal H=\mathcal H_n^c(\ell)\oplus\mathcal H_n^c(\ell)^\perp,
\end{equation*}
and $D^{(n)}$ is the Malliavin derivative on the first factor.  Hence every bounded
$\mathcal E_n$-measurable multiplier $\Lambda$ is constant on
$\mathcal H_n^c(\ell)$ and
\begin{equation*}
 D^{(n)}\Lambda=0.
\end{equation*}
Let $A_k\uparrow\Omega$ be $\mathcal E_n$-measurable with
$\mathbf 1_{A_k}H\in\mathbb D_{(n)}^{1,2}(\mathcal H)$.  Since
$H=\mathsf P_n(S)H$, the Gaussian divergence estimate on the current factor gives
\begin{equation*}
 \mathbb E_W\!\left[|\delta_n^\ell(\mathbf 1_{A_k}H)|^2\mid S=\ell\right]
 \le
 \mathbb E_W\!\left[\mathbf 1_{A_k}
 \left(\|H\|_{\mathcal H}^2+
 \|D^{(n)}H\|_{\mathcal H\otimes\mathcal H}^2\right)\middle|S=\ell\right].
\end{equation*}
The divergences agree on $A_k\cap A_j$ by uniqueness of the adjoint of $D^{(n)}$; thus they define
$\delta_n^{\mathrm{loc}}(H)$.  Under \eqref{eq:localized-divergence-integrability}, integration over
the clock and the complementary Gaussian factor yields
\begin{equation*}
 \mathbb E[|\Lambda\delta_n^{\mathrm{loc}}(H)|^2]
 \le \mathbb E\!\left[|\Lambda|^2
 \left(\|H\|_{\mathcal H}^2+
 \|D^{(n)}H\|_{\mathcal H\otimes\mathcal H}^2\right)\right].
\end{equation*}
Therefore $\Lambda H\in\operatorname{Dom}\delta_n$ and
\eqref{eq:clock-multiplier-divergence} holds.  
For bounded $F\in\mathbb D_{(n)}^{1,2}$, partial duality gives
\begin{align*}
 \mathbb E[\Lambda\langle D^{(n)}F,H\rangle_{\mathcal H}]
 &=\mathbb E[\langle D^{(n)}F,\Lambda H\rangle_{\mathcal H}]\\
 &=\mathbb E[F\delta_n(\Lambda H)]
 =\mathbb E[\Lambda F\delta_n^{\mathrm{loc}}(H)],
\end{align*}
which proves \eqref{eq:conditional-disintegration-ibp} and
\eqref{eq:partial-divergence-L2-estimate}.
\end{proof}

For a pair index $n$ set
\begin{equation*}
 J_n^c=J_{\tau_{2n},\tau_{2n+1}},\qquad
 \mathbb{A}_n^c=\mathbb{A}_{\tau_{2n},\tau_{2n+1}},\qquad
 \mathbb{M}_n^c=\mathbb{M}_{\tau_{2n},\tau_{2n+1}},
\end{equation*}
\begin{equation*}
 \rho_n^c=\rho_{\tau_{2n}},\qquad
 G_n^c=(\mathbb{M}_n^c+\beta I)^{-1},\qquad
 h_n^c=h_{2n}=(\mathbb{A}_n^c)^*G_n^cJ_n^c\rho_n^c,
\end{equation*}
\begin{equation*}
 \mathcal{H}_n^c=L^2([S_{\tau_{2n}},S_{\tau_{2n+1}}];\mathbb{R}^m).
\end{equation*}
For the approximation on this block, start the Galerkin equation at
$\Pi_NZ_{\tau_{2n}}$ and keep the external direction $\rho_n^c$ fixed. Set
\begin{equation*}
 h_n^{c,N}=(\mathbb A_n^{c,N})^*
 (\mathbb M_n^{c,N}+\beta I)^{-1}J_n^{c,N}\Pi_N\rho_n^c.
\end{equation*}
Define the external localizers
\begin{equation}\label{eq:external-control-localizers}
 A_{n,k}=\left\{V(Z_{\tau_{2n}})\le k,\ \Xi_{2n}\le k,
 \ \tau_{2n+1}-\tau_{2n}\le k,\ \|\rho_n^c\|_{\mathbb{H}}\le k\right\},\qquad k\ge1.
\end{equation}
Each $A_{n,k}$ is $\mathcal{E}_n$-measurable and $A_{n,k}\uparrow\Omega$.

\begin{lemma}\label{lem:external-random-direction}
Fix a controlled block $n$.  Let $R$ be an $\mathbb{H}$-valued,
$\mathcal{E}_n$-measurable random variable.  Assume that for some
$\mathcal{E}_n$-measurable sets $B_k\uparrow\Omega$,
$\|R\|_{\mathbb{H}}\le k$ on $B_k$ and
\begin{equation*}
 \mathbb{E}\!\left[\mathbf{1}_{B_k}
 \bigl(\|J_n^c\|_{\mathbb{L}(\mathbb{H})}^2+
 \mathfrak d_{2n}(J_n^c)^2\bigr)\right]<\infty.
\end{equation*}
Then
\begin{equation*}
 J_n^cR\in\mathbb{D}_{(n),\mathrm{loc}}^{1,2}(\mathbb{H}),
 \qquad D^{(n)}R=0,
\end{equation*}
and the closed directional derivative
\begin{equation}\label{eq:external-random-direction-derivative}
 \mathfrak D_nJ_n^c[R]:=D^{(n)}(J_n^cR)
\end{equation}
satisfies
\begin{equation*}
 \|\mathfrak D_nJ_n^c[R]\|_{\mathcal{H}_n^c\otimes\mathbb{H}}
 \le \mathfrak d_{2n}(J_n^c)\|R\|_{\mathbb{H}}.
\end{equation*}
\end{lemma}

\begin{proof}
Fix $k$.  Since $R$ is $\mathcal{E}_n$-measurable, it is constant along the current Gaussian
factor.  Choose $\mathcal{E}_n$-measurable simple variables $R_N$, with values in the countable dense
set $\mathcal{T}$, such that $\mathbf1_{B_k}R_N\to\mathbf1_{B_k}R$ almost surely and $\|\mathbf1_{B_k}R_N\|\le k+1$. Dominated convergence also gives convergence in every finite $L^q$.  Then
\begin{equation*}
 D^{(n)}R_N=0,
 \qquad D^{(n)}(J_n^cR_N)=\mathfrak D_nJ_n^c[R_N].
\end{equation*}
Moreover,
\begin{align*}
 \|J_n^c(R_N-R_M)\|_{\mathbb{H}}
 &\le\|J_n^c\|\,\|R_N-R_M\|_{\mathbb{H}},\\
 \|D^{(n)}(J_n^c(R_N-R_M))\|_{\mathcal{H}_n^c\otimes\mathbb{H}}
 &\le\mathfrak d_{2n}(J_n^c)\|R_N-R_M\|_{\mathbb{H}}.
\end{align*}
The two estimates above, the bound $\|\mathbf1_{B_k}(R_N-R_M)\|\le2(k+1)$, and dominated convergence show that $\mathbf1_{B_k}J_n^cR_N$ is Cauchy in the partial Malliavin graph norm.  By the closedness of
$D^{(n)}$,
\begin{equation*}
 \mathbf{1}_{B_k}J_n^cR\in\mathbb{D}_{(n)}^{1,2}(\mathbb{H}),
 \qquad
 \|D^{(n)}(J_n^cR)\|_{\mathcal{H}_n^c\otimes\mathbb{H}}
 \le\mathfrak d_{2n}(J_n^c)\|R\|_{\mathbb{H}}.
\end{equation*}
Patching over $B_k$ proves the claim.  The same approximation gives $D^{(n)}R=0$.
\end{proof}

In particular,
$R=\rho_n^c$ satisfies the lemma after localization by \eqref{eq:external-control-localizers}, because
$\rho_n^c$ depends only on the clock and on Brownian factors preceding the current controlled block.

\begin{lemma}\label{lem:block-control-sobolev-stability}
Assume the hypotheses of Proposition \ref{prop:block-contraction}.  For every $n\ge0$ and fixed
$\beta>0$,
\begin{equation*}
 h_n^c\in\mathbb{D}_{(n),\mathrm{loc}}^{1,2}(\mathcal{H}),
 \qquad h_n^c=\mathsf P_n(S)h_n^c.
\end{equation*}
For each $k$,
$\mathbf{1}_{A_{n,k}}h_n^{c,N}\to\mathbf{1}_{A_{n,k}}h_n^c$ strongly in
$L^2(\Omega;\mathcal{H})$, and, along a subsequence, the corresponding derivatives converge weakly
in $L^2(\Omega;\mathcal{H}\otimes\mathcal{H})$ to $D^{(n)}h_n^c$.  Moreover,
\begin{align}\label{eq:Dh-local}
 D^{(n)}h_n^c
 &=(D^{(n)}(\mathbb{A}_n^c)^*)G_n^cJ_n^c\rho_n^c
 -(\mathbb{A}_n^c)^*G_n^c(D^{(n)}\mathbb{M}_n^c)G_n^cJ_n^c\rho_n^c\notag\\
 &\quad+(\mathbb{A}_n^c)^*G_n^c\mathfrak D_nJ_n^c[\rho_n^c],
\end{align}
where
\begin{equation*}
 D^{(n)}\mathbb{M}_n^c
 =(D^{(n)}\mathbb{A}_n^c)(\mathbb{A}_n^c)^*
 +\mathbb{A}_n^cD^{(n)}(\mathbb{A}_n^c)^*
\end{equation*}
in $L^2([S_{\tau_{2n}},S_{\tau_{2n+1}}];\mathbb{L}_1(\mathbb{H})^m)$, and
\begin{equation}\label{eq:Dh-beta-summary}
 \|h_n^c\|_{\mathcal{H}_n^c}^2+
 \|D^{(n)}h_n^c\|_{\mathcal{H}_n^c\otimes\mathcal{H}_n^c}^2
 \le C(\beta^{-1}+\beta^{-2}+\beta^{-3})
 \mathbb{J}_{2n}^8\|\rho_{\tau_{2n}}\|_{\mathbb{H}}^2.
\end{equation}
\end{lemma}

\begin{proof}
Write $D^{(n)}=D^{[2n]}$.  Since $\rho_n^c$ is $\mathcal{E}_n$-measurable,
$D^{(n)}\rho_n^c=0$.  Lemma \ref{lem:external-random-direction} therefore gives, on
$A_{n,k}$,
\begin{equation*}
 D^{(n)}(J_n^c\rho_n^c)=\mathfrak D_nJ_n^c[\rho_n^c].
\end{equation*}
Differentiating
$h_n^c=(\mathbb{A}_n^c)^*G_n^cJ_n^c\rho_n^c$ and using
\eqref{eq:DG-resolvent} yields \eqref{eq:Dh-local}.  The Tikhonov bounds are
\begin{equation}\label{eq:h-L2}
 \|h_n^c\|_{\mathcal{H}_n^c}\le\beta^{-1/2}\|J_n^c\rho_n^c\|_{\mathbb{H}},
 \qquad
 \|G_n^c\|\le\beta^{-1},
 \qquad
 \|(\mathbb{A}_n^c)^*G_n^c\|\le\beta^{-1/2}.
\end{equation}
Together with \eqref{eq:DM-norm}, the three terms in \eqref{eq:Dh-local} are bounded by
\begin{equation*}
 C\beta^{-1}\mathbb{J}_{2n}^3\|\rho_n^c\|,
 \qquad
 C\beta^{-3/2}\mathbb{J}_{2n}^4\|\rho_n^c\|,
 \qquad
 C\beta^{-1/2}\mathbb{J}_{2n}^2\|\rho_n^c\|,
\end{equation*}
which implies \eqref{eq:Dh-beta-summary}.

On $A_{n,k}$, the initial state, direction, and clock cost are bounded, and are
constant along the current Gaussian factor. The weighted moment estimate
\eqref{eq:variational-block-moment} implies, for every finite $p$,
\[
 \mathbb E[\mathbf1_{A_{n,k}}\mathbb J_{2n}^p]<\infty.
\]
Cameron--Martin shifts of this factor give bounded finite-rank forcing paths.
Proposition~\ref{prop:differentiability}, \eqref{eq:CM-variation}, and the
resolvent identity give continuous line differentiability of $h_n^c$,
with derivative \eqref{eq:Dh-local}; the initial state and $\rho_n^c$
remain fixed external parameters. Hence \eqref{eq:Dh-beta-summary} and
Lemma~\ref{lem:ray-sobolev} give
$\mathbf1_{A_{n,k}}h_n^c\in\mathbb D_{(n)}^{1,2}(\mathcal H)$.
This step uses absolute continuity along Gaussian lines in addition to the
integrable derivative bound.

For the restarted Galerkin controls,
\eqref{eq:state-galerkin-convergence} and \eqref{eq:J-galerkin-convergence}
give convergence of the kernels at almost every operational time.
The common local variational bounds and dominated integration yield
$\mathbb A_n^{c,N}\to\mathbb A_n^c$ in Hilbert--Schmidt norm, hence
$\mathbb M_n^{c,N}\to\mathbb M_n^c$ in operator norm.
On $A_{n,k}$, $V(\Pi_NZ_{\tau_{2n}})\le C_k$ uniformly in $N$.
The weighted estimates therefore give uniform moments of order greater than
two for the controls and uniform $L^2$ bounds for their derivatives.
Uniform integrability yields strong $L^2$ convergence of the controls;
weak closedness identifies every derivative limit with \eqref{eq:Dh-local}.
Letting $k\to\infty$ proves the local statement.

\end{proof}
\begin{lemma}\label{lem:skorohod-control}
Assume the hypotheses of Proposition~\ref{prop:block-contraction}. There exists $\widetilde q>0$ such that, for every
\begin{equation}\label{eq:b-choice}
 0<b<\min\left\{\frac{\widehat\gamma}{8\widetilde{q}},
                    \frac{\widehat\gamma}{8\widetilde c}\right\}
\end{equation}
and every $L,K>0$, the block control
$
 h_{n,t}^{\bullet}
 =h_{2n}\,\mathbf{1}_{[S_{\tau_{2n}},(S_{\tau_{2n+1}}\wedge S_t)\vee S_{\tau_{2n}}]}
$
belongs locally to the conditional block divergence construction of Lemma
\ref{lem:conditional-block-divergence}.  On the good-clock multiplier used below,
\eqref{eq:localized-divergence-integrability} holds, so the localized field belongs to the genuine
$L^2$ divergence domain.  Define
\begin{equation*}
 \Delta_{n,t}^{L,K}
 :=\delta_n\!\left(\mathbf{1}_{\mathbb{G}_{L,K}^{\widehat\delta,b}}h_{n,t}^{\bullet}\right)
 =\mathbf{1}_{\mathbb{G}_{L,K}^{\widehat\delta,b}}\,\delta_n^{\mathrm{loc}}(h_{n,t}^{\bullet}),
\end{equation*}
where the second identity follows from \eqref{eq:clock-multiplier-divergence}.  Then
\begin{equation}\label{eq:control-bound}
 \mathbb{E}_{z_0}[|\Delta_{n,t}^{L,K}|^2]
 \le C_{\kappa,L,K,\beta}
 e^{C\kappa V(z_0)}e^{-\widehat\gamma n/4}
\end{equation}
uniformly in $t\ge0$.  Consequently, the localized series of block divergences converges in $L^2$ uniformly in $t$:
\begin{equation}\label{eq:total-control-bound}
 \sup_{t\ge0}
 \left\|\sum_{n\ge0}\Delta_{n,t}^{L,K}\right\|_{L^2(\Omega)}
 \le C_{\kappa,L,K,\beta}e^{C\kappa V(z_0)}.
\end{equation}
\end{lemma}

\begin{proof}
Condition on the full subordinator path.  The current Gaussian factor has Cameron--Martin space $\mathcal{H}_n^c$, while the complementary Wiener factors are parameters.  Lemma \ref{lem:block-control-sobolev-stability} supplies the local Sobolev domain, \eqref{eq:Dh-local}, and \eqref{eq:Dh-beta-summary}.  Choose $\eta_V\in(0,\tilde{\eta}/4]$ and apply the strengthened estimate
\eqref{eq:variational-block-moment} with $p=16$.  For a corresponding $\widetilde{q}>0$,
\begin{equation}\label{eq:weighted-current-variation}
 \mathbb{E}\left[\mathbb{J}_{2n}^{16}e^{-\widetilde{q}\Xi_{2n}}
 \mid\mathcal{F}_{\tau_{2n}}\right]
 \le C_{\kappa,\eta_V}e^{\eta_V V(Z_{\tau_{2n}})}.
\end{equation}
The choice $4\eta_V\le\tilde{\eta}$ leaves enough room for the Cauchy--Schwarz and stopped-energy
exponents below.  On $\mathbb{G}_{L,K}^{\widehat\delta,b}$,
$\Xi_{2n}\le K+2bn$.  Since
\begin{equation*}
 \mathbb{J}_{2n}^{8}\|\rho_{\tau_{2n}}\|^2
 =\bigl(\mathbb{J}_{2n}^{16}e^{-\widetilde{q}\Xi_{2n}}\bigr)^{1/2}
  \|\rho_{\tau_{2n}}\|^2e^{\widetilde{q}\Xi_{2n}/2},
\end{equation*}
the current-clock bound gives
\begin{align*}
 &\mathbb{E}\left[\mathbb{J}_{2n}^{8}\|\rho_{\tau_{2n}}\|^2
 \mathbf{1}_{\mathbb{G}_{L,K}^{\widehat\delta,b}}\right]\le e^{\widetilde{q} K/2+\widetilde{q} bn}
 \left(\mathbb{E}[\mathbb{J}_{2n}^{16}e^{-\widetilde{q}\Xi_{2n}}]\right)^{1/2}
 \left(\mathbb{E}[\|\rho_{\tau_{2n}}\|^4
 \mathbf{1}_{\mathbb{G}_{L,K}^{\widehat\delta,b}}]\right)^{1/2}.
\end{align*}
By \eqref{eq:weighted-current-variation}, conditioning at $\tau_{2n}$ and then using Corollary
\ref{cor:stopped-exp-uniform} bounds the first square-root by
$C(\kappa) e^{C\kappa V(z_0)}$.  Proposition \ref{prop:block-contraction} bounds the second by
$C_Le^{C\kappa V(z_0)-\widehat\gamma n/2}$.  Hence
\begin{equation}\label{eq:weighted-current-times-residual}
 \mathbb{E}\left[\mathbb{J}_{2n}^{8}\|\rho_{\tau_{2n}}\|^2
 \mathbf{1}_{\mathbb{G}_{L,K}^{\widehat\delta,b}}\right]
 \le C_{\kappa,L,K}e^{C\kappa V(z_0)}
 \exp\left\{\widetilde{q} bn-\frac{\widehat\gamma}2n\right\}.
\end{equation}
By Lemma \ref{lem:block-control-sobolev-stability}, \eqref{eq:Dh-beta-summary}, and
\eqref{eq:weighted-current-times-residual}, the local-domain and integrability hypotheses of Lemma
\ref{lem:conditional-block-divergence} hold with
$\Lambda=\mathbf{1}_{\mathbb{G}_{L,K}^{\widehat\delta,b}}$, also after the clock-measurable truncation defining
$h_{n,t}^{\bullet}$.  Hence the measurable conditional divergence is well defined.  On each fixed clock fiber the standard estimate
(Nualart \cite[Proposition~1.3.1]{Nualart2006}) gives
\begin{equation*}
 \mathbb{E}_W\!\left[|\delta_n^\ell(H)|^2\mid S=\ell\right]
 \le \mathbb{E}_W\!\left[\|H\|_{\mathcal{H}_n^c}^2
 +\|D^{(n)}H\|_{\mathcal{H}_n^c\otimes\mathcal{H}_n^c}^2\mid S=\ell\right].
\end{equation*}
The truncation in $h_{n,t}^{\bullet}$ and the event $\mathbb{G}_{L,K}^{\widehat\delta,b}$ are clock
measurable and hence $\mathcal{E}_n$-measurable, so their $D^{(n)}$ derivatives vanish and
\eqref{eq:clock-multiplier-divergence} applies.  Integrating the fiberwise estimate over the clock
and complementary Gaussian factors and using
\eqref{eq:Dh-beta-summary} and \eqref{eq:weighted-current-times-residual}, we obtain uniformly in
$t$,
\begin{align}
 \mathbb{E}[|\Delta_{n,t}^{L,K}|^2]\le C_{\beta,\kappa,L,K}e^{C\kappa V(z_0)}
 \exp\left\{-\frac{\widehat\gamma}2n+\widetilde{q} bn\right\}.
 \label{eq:divergence-decay-before-b}
\end{align}
If $b<\widehat\gamma/(8\widetilde{q})$, then
$-\widehat\gamma/2+\widetilde{q} b\le-3\widehat\gamma/8<-\widehat\gamma/4$, and therefore
\eqref{eq:divergence-decay-before-b} implies exactly
\eqref{eq:control-bound}.  Only $D^{(n)}h_n^c$ enters the fiberwise divergence estimate.  By Lemma~\ref{lem:conditional-block-divergence}, no cross-block trace term occurs.  By
Minkowski's inequality and \eqref{eq:control-bound},
\begin{equation*}
 \left\|\sum_{n=N_1}^{N_2}\Delta_{n,t}^{L,K}\right\|_{L^2}
 \le C\sum_{n=N_1}^{N_2}e^{-\widehat\gamma n/8},
\end{equation*}
uniformly in $t$.  Hence the localized divergence series is Cauchy in $L^2$, and the geometric series gives \eqref{eq:total-control-bound}.  \end{proof}

\begin{proposition}\label{prop:localized-gradient}
Under the hypotheses above, define
\begin{equation*}
 P_t^{L,K,\widehat\delta,b}\varphi(z)
 =\mathbb{E}_z[\varphi(Z_t)\mathbf{1}_{\mathbb{G}_{L,K}^{\widehat\delta,b}}].
\end{equation*}
For every $R>0$, there exists $C_{R,L,K,\widehat\delta,b}<\infty$, independent of $t$, such that for every $\varphi\in C_b^1(\mathbb{H})$,
\begin{equation}\label{eq:localized-gradient}
 \|\nabla P_t^{L,K,\widehat\delta,b}\varphi(z)\|_{\mathbb{H}}
 \le C_{R,L,K,\widehat\delta,b}
 (\|\varphi\|_\infty+\|\nabla\varphi\|_\infty),
 \qquad V(z)\le R.
\end{equation}
\end{proposition}

\begin{proof}
Fix $\|\xi\|_{\mathbb{H}}=1$, let
$h_N^{(t)}=\sum_{n=0}^{N}h_{n,t}^{\bullet}$ and
$\rho_t^N=J_{0,t}\xi-\mathbb{A}_{0,t}h_N^{(t)}$.  

Fix $t$ and put $M=\lfloor c_Vt/\chi\rfloor+2$ and
$D=K+bM$. On $\mathbb G=\mathbb G_{L,K}^{\widehat\delta,b}$,
all blocks meeting $[0,t]$ have index less than $M$, $\Xi_j\le D$, and
$S_t\le MD$. For any $p<\infty$, choose $0<\eta\le\eta_*$ and use
\eqref{eq:variational-block-moment} at exponent $pM$, followed by
Corollary~\ref{cor:stopped-exp-uniform}. H\"older gives
\begin{align*}
 \mathbb E\left[\mathbf1_{\mathbb G}\prod_{j=0}^{M-1}\mathbb J_j^p\right]
 &\le e^{cD}\prod_{j=0}^{M-1}
       \left(\mathbb E[\mathbb J_j^{pM}e^{-c\Xi_j}]\right)^{1/M}\le C_{t,L,K,p}e^{\eta V(z)}.
\end{align*}
The flow property bounds $J_{0,t}$ by this product. Formula
\eqref{eq:DZ-formula-main} also bounds $\|D_WZ_t\|$ by
$C(MD)^{1/2}$ times the same product. The constants are uniform on bounded
energy sets. Pathwise continuity of $J$ and this bound with $p>2$ justify
differentiation under the expectation by uniform integrability, also along a
segment of initial states.

The Cameron--Martin line differentiability from
Proposition~\ref{prop:differentiability} and \eqref{eq:CM-variation}, together
with the preceding integrable derivative bound, permits application of
Lemma~\ref{lem:ray-sobolev}. Thus
\[
 \mathbf1_{\mathbb G}\varphi(Z_t)\in\mathbb D_{(n)}^{1,2},\qquad
 D^{(n)}(\mathbf1_{\mathbb G}\varphi(Z_t))
 =\mathbf1_{\mathbb G}D^{(n)}\varphi(Z_t).
\]

For finite $N$, condition on the clock and use \eqref{eq:partial-wiener-decomposition}. In the partial duality formula take $F=\mathbf1_{\mathbb G_{L,K}^{\widehat\delta,b}}\varphi(Z_t)$.
 Lemma
\ref{lem:conditional-block-divergence}, with
$\Lambda=\mathbf{1}_{\mathbb{G}_{L,K}^{\widehat\delta,b}}$,
gives for every controlled block
\begin{equation*}
 \mathbb{E}\!\left[
 \big\langle \nabla\varphi(Z_t),\mathbb{A}_{0,t}h_{n,t}^{\bullet}\big\rangle
 \mathbf{1}_{\mathbb{G}_{L,K}^{\widehat\delta,b}}\right]
 =\mathbb{E}\!\left[
 \varphi(Z_t)\,\Delta_{n,t}^{L,K}\right].
\end{equation*}
Summing over $n\le N$ and using
$J_{0,t}\xi=\mathbb{A}_{0,t}h_N^{(t)}+\rho_t^N$ yields
\begin{align*}
 |\nabla_\xi P_t^{L,K,\widehat\delta,b}\varphi(z)|
 &\le \|\varphi\|_\infty
 \left\|\sum_{n=0}^{N}\Delta_{n,t}^{L,K}\right\|_{L^1}+\|\nabla\varphi\|_\infty
 \mathbb{E}\left[\|\rho_t^N\|\mathbf{1}_{\mathbb{G}_{L,K}^{\widehat\delta,b}}\right].
\end{align*}
By Lemma \ref{lem:skorohod-control} the first term is uniformly
bounded, so we may take $N\to\infty$.  Blocks after the pair
containing $t$ give zero; \eqref{eq:rho-between-blocks} gives
\begin{equation*}
 \mathbb{E}[\|\rho_t\|\mathbf 1_{\mathbb{G}_{L,K}^{\widehat\delta,b}}]
 \le \sum_{n\ge0}
 \mathbb{E}[\|\rho_t\|^4\mathbf 1_{\mathbb{G}_{L,K}^{\widehat\delta,b}}
 \mathbf 1_{\{t\in[\tau_{2n},\tau_{2n+2})\}}]^{1/4}
 \le C_{R,L,K}\sum_{n\ge0}e^{-\widehat\gamma n/8}<\infty.
\end{equation*}
Hence \eqref{eq:localized-gradient} follows.
\end{proof}

\begin{lemma}\label{lem:smooth-lipschitz-approx}
Let $\varphi:\mathbb{H}\to\mathbb{R}$ be bounded and Lipschitz.  There exists
$\varphi_N\in C_b^\infty(\mathbb{H})$ such that, for every $z\in\mathbb{H}$,
\begin{equation}\label{eq:smooth-lipschitz-approx}
 \varphi_N(z)\longrightarrow\varphi(z),\qquad
 \|\varphi_N\|_\infty\le\|\varphi\|_\infty,
 \qquad
 \|\nabla\varphi_N\|_\infty\le\operatorname{Lip}(\varphi).
\end{equation}
\end{lemma}
\begin{proof}
Use the phase-space projections $\Pi_N$ from Section~\ref{sec:setting}
and set $E_N=\Pi_N\mathbb H$. On $E_N$, choose a nonnegative
$C_c^\infty$ mollifier $\varrho_N$ of unit mass and a scale
$\varepsilon_N>0$ so small that
	\begin{equation*}
	\varepsilon_N\int_{E_N}|y|\varrho_N(y)\mathrm{d} y\le \frac1N.
	\end{equation*}
	Define the function
	\begin{equation*}
	\varphi_N(z)=\int_{E_N}\varphi(\Pi_Nz-\varepsilon_Ny)\varrho_N(y)\mathrm{d} y.
	\end{equation*}
	It belongs to $C_b^\infty(\mathbb{H})$ and convolution preserves both
	the sup norm and the Lipschitz constant.  Moreover,
	\begin{equation*}
	|\varphi_N(z)-\varphi(z)|
	\le \operatorname{Lip}(\varphi)\|\Pi_Nz-z\|+\frac{\operatorname{Lip}(\varphi)}{N},
	\end{equation*}
	which implies \eqref{eq:smooth-lipschitz-approx}.
\end{proof}

\begin{proposition}\label{prop:e-property}
Let \eqref{eq:fixed-rank-admissible} hold.  Then for every $R>0$ and every bounded Lipschitz function $\varphi$,
\begin{equation}\label{eq:e-property}
 \lim_{r\to0}
 \sup_{t\ge0}
 \sup_{\substack{V(z)\le R,\;V(z')\le R\\
                  \|z-z'\|_{\mathbb{H}}\le r}}
 |P_t\varphi(z)-P_t\varphi(z')|=0.
\end{equation}
\end{proposition}

\begin{proof}
Fix $R>0$ and $\widehat\epsilon>0$.  Choose $b$ by
\eqref{eq:b-choice}.  By \eqref{eq:Theta-centered-tail},
\eqref{eq:Theta-local-tail}, and \eqref{eq:good-clock-tail}, choose $L$ and $K$ so large that
\begin{equation*}
 2\|\varphi\|_\infty
 \mathbb{P}\bigl((\mathbb{G}_{L,K}^{\widehat\delta,b})^c\bigr)<\frac{\widehat\epsilon}2.
\end{equation*}
For $z_\theta=(1-\theta)z+\theta z'$, $0\le\theta\le1$, convexity of the squared norm and the fourth power gives, by \eqref{eq:V-equivalence},
\[
V(z_\theta)\le C_0(1+R/c_0)=:R'.
\]
Apply Proposition \ref{prop:localized-gradient} on $\{V\le R'\}$ to the approximants from Lemma \ref{lem:smooth-lipschitz-approx}, and use dominated convergence. Renaming the constant, uniformly in $t$,
\begin{equation*}
 |P_t^{L,K,\widehat\delta,b}\varphi(z)-P_t^{L,K,\widehat\delta,b}\varphi(z')|
 \le C_{R,L,K,\widehat\delta,b}
 (\|\varphi\|_\infty+\operatorname{Lip}(\varphi))\|z-z'\|_{\mathbb{H}}.
\end{equation*}
The bad-clock event contributes at most
$2\|\varphi\|_\infty\mathbb{P}((\mathbb{G}_{L,K}^{\widehat\delta,b})^c)$.
Taking $r$ small enough yields \eqref{eq:e-property}.
\end{proof}

\subsection{Cubic realization of the fixed-rank condition}

\begin{proposition}\label{prop:cubic-realization}
There exists $m_0$ with the following property. Fix $m\ge m_0$, any operator $Q_m$
as in \eqref{eq:Q}, and a subordinator satisfying Assumption~\ref{ass:levy}.
For
\[
 g_a(s)=a s^3,\qquad 0<a\le1,
\]
there exists $a_*=a_*(m,Q_m,S)\in(0,1]$ such that the constants and parameters in
\eqref{eq:fixed-rank-admissible} can be chosen to satisfy that condition whenever
$0<a<a_*$. No uniformity in the amplitudes of $Q_m$ is asserted.
\end{proposition}
\begin{proof}
Write $V_a$ and $\mathfrak F_a$ for the energy and dissipation associated with
$G_a(s)=as^4/4$. Their quadratic parts and the constants
$c_V,c_D,c_*,c_Q,\mathscr K,\mathfrak c$ may be chosen independently of
$0<a\le1$, since $g_a(s)s=4G_a(s)$.
Set $w=a^{1/4}u$ and
\[
 \Phi_a(u)=a^{1/2}(1+\|w\|_{L^4}^{32/9}),\qquad
 R_a(s,t)=\int_s^t\Phi_a(u_r)\,\mathrm dr.
\]
The first-variation energy calculation, retaining the coefficient of $g_a'$, gives
\begin{align*}
 3a\int |u|^2|h||k|
 &\le\epsilon\|k\|_{H^1}^2
     +C_\epsilon a^{4/5}\|w\|_{L^4}^{16/5}\mathcal W(h,k),\\
 3\delta_0a\int |u|^2|h|^2
 &\le C a^{1/2}\|w\|_{L^4}^2\mathcal W(h,k).
\end{align*}
Both state factors are bounded by $C\Phi_a$. The enhanced high-velocity
interpolation used in Proposition~\ref{prop:high-branch} produces
$a^{8/9}\|w\|_{L^4}^{32/9}$, also bounded by $\Phi_a$.
Thus its two coupled inequalities hold with $\Phi$ replaced by $\Phi_a$,
with the same spectral rates for $\varepsilon_m,d_m$ and the same linear damping.
In particular, for constants independent of $a,m,s,t$,
\begin{align}
 \|J_{s,t}\|^{p_0}&\le \mathscr H_0 e^{\Lambda_0 R_a(s,t)},\label{eq:cubic-first-factor}\\
 \|J_{s,t}\Pi_m^\perp\|^{p_0}
 &\le \mathscr H e^{-p_0\mathfrak c(t-s)/2}
          e^{\Lambda_0\varepsilon_m R_a(s,t)}
       +C d_m e^{\Lambda_0 R_a(s,t)}.\label{eq:cubic-high-factor}
\end{align}
To obtain the second line, integrate the coupled inequalities as in
\eqref{eq:leakage-energy}--\eqref{eq:full-high-energy}, use
$x\le e^x$ for $x\ge0$, and raise the energy bound to $p_0/2$.
The full-flow bound in the first line follows by discarding only its negative
linear damping; there is no positive constant multiple of $t-s$ independent of $a$.
The factor $C$ in the leakage term is absorbed into the geometric constants below.

Fix $\kappa,\chi$ for the moment and choose
\[
 \tilde\eta_a=a^{1/4}\bar\eta/4,\qquad
 \lambda_a=\tilde\eta_a/\kappa.
\]
Since $\mathfrak F_a\ge\|w\|_{L^4}^4/4$, weighted Young's inequality with
exponents $9/8$ and $9$ gives, for every fixed $\Lambda_0>0$,
\[
 \Lambda_0\Phi_a(u)
 \le \lambda_a c_*\kappa\omega_n\mathfrak F_a
    +C_{\Lambda_0}\left(a^{1/2}+a^{9/2}\tilde\eta_a^{-8}\right)\omega_n^{-8}.
\]
Raise \eqref{eq:weighted-stopped-energy} to $\lambda_a\le1$, drop its negative
terminal-energy term, and apply conditional Jensen. This yields
\begin{equation}\label{eq:cubic-normalization}
 \mathbb E\left[e^{\Lambda_0 R_a(\tau_n,\tau_{n+1})-c_{\Lambda_0,a}\Xi_n}
                      \mid\mathcal F_{\tau_n}\right]
 \le C_{{\rm wt},\kappa}^{\lambda_a}
      e^{\tilde\eta_a V_a(Z_{\tau_n})},
\end{equation}
where one may take
\[
 c_{\Lambda_0,a}=C_{\Lambda_0}\left(a^{1/2}+a^{9/2}\tilde\eta_a^{-8}\right)
                  +\mathscr K\tilde\eta_a B_Q
 \longrightarrow0\quad(a\downarrow0).
\]
Indeed, $a^{9/2}\tilde\eta_a^{-8}=C_\kappa a^{5/2}$.
The prefactor in \eqref{eq:cubic-normalization} tends to one.
Conditional Jensen supplies the same bound with every term in the exponent
multiplied by $\varepsilon_m\in(0,1]$.

Use \eqref{eq:cubic-normalization} at the finitely many coefficients $\Lambda_0$
required for fourth moments of the factors in
\eqref{eq:cubic-first-factor}--\eqref{eq:cubic-high-factor}.
It follows that those factors can be chosen with
\[
 \widehat c_a\longrightarrow0,
\]
while their moment prefactors tend to constants depending only on $p_0$ and
norm equivalence. The stopped-recursion factors in Lemma~\ref{lem:one-pair}
occur as powers proportional to $\tilde\eta_a/\eta_*$ and hence tend to one.
Moreover, \eqref{eq:sum-constant-explicit} shows that the word-product constant
$C_2$ tends to one. Consequently there is a geometric constant $C_g$, independent
of $m,Q_m,\kappa,\chi$, such that, for every fixed choice of these parameters,
\begin{equation}\label{eq:cubic-geometric-limit}
 \limsup_{a\downarrow0}\mathscr C_a\le C_g.
\end{equation}
Here the constants are those constructed from the preceding first-variation
bounds; the later Sobolev constants for the controls are not involved.

Choose $m_0$ to cover the rank threshold in Proposition~\ref{prop:high-branch},
to ensure $\varepsilon_m+d_m\le1$, and to satisfy
$C_gd_m^{\theta_0}<1/16$ for $m\ge m_0$. The cubic estimates above make this
choice uniform in $0<a\le1$. Choose $\chi\ge\chi_0$ so that
$C_g\widehat q(\chi)^{\theta_0}<1/16$.
For the fixed $Q_m$, choose $\kappa$ satisfying Lemma~\ref{lem:stopped-energy}
and fix any $\widehat\delta>0$. Then $\widehat\mu=\overline\Xi+\widehat\delta<\infty$
is independent of $a$. By \eqref{eq:cubic-geometric-limit},
\[
 \limsup_{a\downarrow0}
 \mathscr C_a e^{2\theta_0\widehat c_a\widehat\mu}
       (\widehat q(\chi)^{\theta_0}+d_m^{\theta_0})<\tfrac18.
\]
Thus \eqref{eq:fixed-rank-admissible} holds for all sufficiently small $a>0$.
The covariance tolerance and $\beta$ are chosen only afterwards.
\end{proof}

For example, a gamma subordinator has
$\nu_S(\mathrm dr)=c r^{-1}e^{-br}\,\mathrm dr$, $b,c>0$.
Its L\'evy measure is infinite and
\[
 \int_0^\infty(e^{\zeta r}-1)\nu_S(\mathrm dr)
   =c\log\frac b{b-\zeta}<\infty\qquad(0<\zeta<b),
\]
so Proposition~\ref{prop:cubic-realization} gives a concrete pure-jump family.

\section{Unique ergodicity}\label{sec:ergodicity}

\subsection{Weak irreducibility}\label{sec:irreducibility}
Let
$\zeta_t=QW_{S_t}$
and set
$ y_t=v_t-\zeta_t$.
Then, 
\begin{equation}\label{eq:shifted}
 \partial_t u(t)=y_t+\zeta_t,
 \qquad
 \partial_t y(t)=-Ay_t-A\zeta_t-Au_t-P_0g(u_t).
\end{equation}

\begin{lemma}\label{lem:small-forcing}
For every $R,\widehat r>0$, there exist $T=T(R,\widehat r)>0$ and
$\varepsilon=\varepsilon(R,\widehat r)>0$ such that, if
\begin{equation}\label{eq:small-force-event}
 \sup_{0\le t\le T}\|\zeta_t\|_{H^1}<\varepsilon,
\end{equation}
then
\begin{equation}\label{eq:reach-small-ball}
 \|Z_T^{z_0}\|_{\mathbb{H}}<\widehat r
 \qquad\text{for every }V(z_0)\le R.
\end{equation}
\end{lemma}
\begin{proof}
Apply the modified energy $V(u,y)$ from \eqref{eq:V} to the system \eqref{eq:shifted}.
Since $\zeta_t$ always belongs to the fixed finite-dimensional space $H_m$, all Sobolev norms of $\zeta_t$ are equivalent.  In particular, the terms $(A\zeta,y)$ and $(P_0g(u),\zeta)=(g(u),\zeta)$ are controlled by the $H^1$ size of $\zeta$.  Using \eqref{eq:g-diss} and Young's inequality, every term containing $\zeta$ is bounded by
\begin{equation*}
 \eta\mathfrak{F}(u,y)+C_{\eta,m}\bigl(\|\zeta\|_{H^1}^2+\|\zeta\|_{H^1}^4\bigr).
\end{equation*}
Taking $\eta$ small yields
\begin{equation*}
 \frac{\,\mathrm{d}}{\,\mathrm{d} t}V(u_t,y_t)+cV(u_t,y_t)
 \le C\bigl(\varepsilon^2+\varepsilon^4\bigr)
\end{equation*}
on the event \eqref{eq:small-force-event}.  Hence
\begin{equation*}
 V(u_T,y_T)
 \le e^{-cT}V(u_0,v_0)+C(\varepsilon^2+\varepsilon^4).
\end{equation*}
Choose $T$ sufficiently large and then $\varepsilon$ sufficiently small.  Since $v_T=y_T+\zeta_T$,
\eqref{eq:reach-small-ball} follows.
\end{proof}

\begin{lemma}\label{lem:small-ball}
For every $T,\varepsilon>0$,
\begin{equation*}
 \mathbb{P}\left(
   \sup_{0\le t\le T}\|QW_{S_t}\|_{H^1}<\varepsilon
 \right)>0.
\end{equation*}
\end{lemma}

\begin{proof}
Choose $M>0$ such that $\mathbb{P}(S_T\le M)>0$.  By independence of $S$ and $W$,
\begin{align*}
 &\mathbb{P}\left(
   \sup_{0\le t\le T}\|QW_{S_t}\|_{H^1}<\varepsilon
 \right)\ge
 \mathbb{P}(S_T\le M)
 \mathbb{P}\left(
   \sup_{0\le r\le M}\|QW_r\|_{H^1}<\varepsilon
 \right)>0.
\end{align*}
The last probability is strictly positive because $QW$ is a finite-dimensional continuous Gaussian process.
\end{proof}

\begin{proposition}\label{prop:weak-irr}
For every $R,\widehat r>0$ there exists $T=T(R,\widehat r)>0$ such that
\begin{equation*}
 \inf_{V(z_0)\le R}P_T(z_0,B_{\mathbb{H}}(0,\widehat r))>0.
\end{equation*}
\end{proposition}

\begin{proof}
	Let $T,\varepsilon$ be as in Lemma \ref{lem:small-forcing}.  On the
	event of Lemma \ref{lem:small-ball}, every solution with $V(z_0)\le R$ reaches $B_{\mathbb H}(0,\widehat r)$ at time $T$. Thus
\begin{equation*}
 \inf_{V(z_0)\le R}P_T(z_0,B_{\mathbb{H}}(0,\widehat r))
 \ge
 \mathbb{P}\left(\sup_{0\le t\le T}\|QW_{S_t}\|_{H^1}<\varepsilon\right)>0.
\end{equation*}
This completes the proof.
\end{proof}

\subsection{Proof of the main results}
This subsection is devoted to the proofs of the main results
presented in Section~\ref{DeSubintro}.

\begin{proof}[Proof of Theorem \ref{thm:main}]
The existence of an invariant measure and \eqref{eq:mu-V} follow from Proposition \ref{prop:existence}, while weak irreducibility is given by Proposition \ref{prop:weak-irr}.
\end{proof}

We use the following standard support consequence of the e-property.
\begin{lemma}\label{lem:unique-criterion}
Let $\{P_t\}$ be a stochastically continuous Feller Markov semigroup on a Polish space with the
e-property.  Then the supports of two distinct ergodic invariant
measures are disjoint.
\end{lemma}

\begin{proof}

Suppose $x$ lies in the supports of distinct ergodic invariant measures
$\mu_1,\mu_2$. Choose $f\in\operatorname{Lip}_b$ with $\mu_1(f)\ne\mu_2(f)$,
and write $Q_Tf=T^{-1}\int_0^TP_tf\,\mathrm dt$.
The mean ergodic theorem on $L^2(\mu_i)$ gives
$Q_Tf\to\mu_i(f)$ in $L^2(\mu_i)$; along one common subsequence $T_k\to\infty$,
the convergence holds $\mu_i$-almost everywhere, $i=1,2$.
For $\epsilon>0$, the e-property at $x$ provides $r>0$ such that
\[
 \sup_{t\ge0}|P_tf(y)-P_tf(x)|<\epsilon\qquad(y\in B(x,r)).
\]
Each ball has positive $\mu_i$-mass, so choose convergence points $y_i\in B(x,r)$.
Then $|Q_{T_k}f(y_1)-Q_{T_k}f(y_2)|\le2\epsilon$.
Letting $k\to\infty$ and then $\epsilon\downarrow0$ is a contradiction.
This is the support-separation criterion used in
\cite{KapicaSzarekSleczka2012,KomorowskiPeszatSzarek2010,GongLiu2015}.

\end{proof}

\begin{proof}[Proof of Theorem \ref{thm:conditional-ergodicity}]
By Theorem \ref{thm:main}, invariant measures exist.  Under \eqref{eq:fixed-rank-admissible}, Proposition \ref{prop:e-property} gives the e-property, while Proposition \ref{prop:wellposed} gives stochastic continuity.  It remains to show that the origin belongs to the support of every invariant measure.

Let $\mu$ be invariant, and assume that $\widehat r>0$.  Since $V$ is finite on $\mathbb{H}$ and
$\bigcup_{R>0}\{V\le R\}=\mathbb{H}$, we may choose $R$ such that
\begin{equation*}
 \mu(\{V\le R\})>0.
\end{equation*}
By invariance and Proposition \ref{prop:weak-irr}, we obtain
\begin{align*}
 \mu(B_{\mathbb{H}}(0,\widehat r))
 &=\int_{\mathbb{H}}P_T(z,B_{\mathbb{H}}(0,\widehat r))\,\mu(\,\mathrm{d} z)\\
 &\ge \int_{\{V\le R\}}P_T(z,B_{\mathbb{H}}(0,\widehat r))\,\mu(\,\mathrm{d} z)>0.
\end{align*}
Thus $0\in\operatorname{supp}\mu$ for every invariant measure, and the same holds for every ergodic one.  By the standard ergodic decomposition theorem for invariant measures on Polish spaces \cite{DaPratoZabczyk1996}, nonuniqueness would give two distinct ergodic invariant measures, whose supports both contain the origin, contradicting Lemma \ref{lem:unique-criterion}.  Hence the invariant measure is unique and therefore ergodic.

Finally, fix $z\in\mathbb{H}$ and set
$\mu_T^z=\frac{1}{T}\int_0^T P_t^*\delta_z\,\mathrm{d}t$.  By Lemma
\ref{lem:compactness}, $(\mu_T^z)$ is tight as $T\to\infty$; the
Feller--Krylov--Bogoliubov argument shows every weak limit point is invariant.  Uniqueness forces every such limit point to equal
$\mu_*$.  A tight family with a single possible subsequential limit converges as a whole, proving \eqref{eq:cesaro-convergence}.
\end{proof}

\begin{proof}[Proof of Corollary~\ref{cor:cubic-ergodicity}]
For $g(s)=as^3$, Assumption~\ref{ass:g} holds with $G(s)=as^4/4$ and $\iota=4$.
Proposition~\ref{prop:cubic-realization} verifies the fixed-rank condition;
Theorem~\ref{thm:conditional-ergodicity} gives the conclusions.
\end{proof}

\subsection{Synchronous coupling for a small cubic coefficient}\label{sec:synchronous}

For a sufficiently small cubic coefficient, an independent argument gives
uniqueness without the covariance and control estimates. Both solutions are
driven by the same realization of $L$. This is a synchronous coupling, a
special case of asymptotic coupling; see \cite{HairerMattinglyScheutzow2011}.
The smallness condition below is separate from \eqref{eq:fixed-rank-admissible}.

\begin{proposition}\label{prop:synchronous}
Let $g(s)=as^3$, $0<a\le1$, and suppose Assumption~\ref{ass:levy} holds.
In this proposition, $Q:\mathbb R^m\to H_0$ may be any operator with finite
spectral range; neither \eqref{eq:Q} nor a lower bound on its rank is required.
There are structural constants $c_s,C_s>0$, independent of $a,Q,S$, such that,
if
\begin{equation}\label{eq:synchronous-smallness}
 \epsilon_a\left(2+\frac{\mathcal E_L+2}{a_0}\right)<c_s,
 \qquad \epsilon_a=C_sa^{4/5},
\end{equation}
then the semigroup has the bounded-set-uniform e-property, a unique invariant
measure $\mu_a$, and
\begin{equation}\label{eq:synchronous-convergence}
 P_t^*\delta_z\Longrightarrow\mu_a\quad(t\to\infty),
 \qquad z\in\mathbb H.
\end{equation}
Here $a_0$ is the dissipation constant in Lemma~\ref{lem:basic-energy}, chosen
with $\iota=4$, so it is independent of $a$. In particular, the condition
holds for every fixed $Q,S$ when $a>0$ is sufficiently small.
\end{proposition}

\begin{proof}
Write $V_a,\mathfrak F_a$ for the cubic energy and dissipation. The energy
calculation in Lemma~\ref{lem:basic-energy}, before taking expectations, gives
\begin{equation}\label{eq:synchronous-energy}
 V_a(Z_t)+a_0\int_0^t\mathfrak F_a(Z_r)\,\mathrm dr
 \le V_a(z)+\tfrac12\mathcal E_Lt+M_t^z.
\end{equation}
The martingale starts at zero and has compensated jump integrand
$(v+\delta u,Qy)+\|Qy\|^2/2$. Consequently,
\[
 \mathbb E_z|M_t^z|^2
 \le C_Q\int_0^t(1+\mathbb E_zV_a(Z_r))\,\mathrm dr
 \le C_Rt,\qquad V_a(z)\le R.
\]
The fourth jump moment is finite by Assumption~\ref{ass:levy}; the quadratic
coercivity and mean energy constants can be chosen uniformly for $0<a\le1$.
Set $H_z=\sup_{t\ge0}(M_t^z-t)_+$. For $K\ge1$, Doob's inequality on
$[0,1]$ and $[0,2^{j+1}]$ gives
\begin{align}\label{eq:martingale-envelope}
 \sup_{V_a(z)\le R}\mathbb P(H_z>K)
 &\le C_R\left(K^{-2}
       +\sum_{j\ge0}\frac{2^j}{(K+2^j)^2}\right)
 \le\frac{C_R}{K}.
\end{align}
For the last inequality, split the sum at $2^j\le K$. In particular, $H_z$
is finite almost surely and
\begin{equation}\label{eq:synchronous-occupation}
 \int_0^t\mathfrak F_a(Z_r^z)\,\mathrm dr
 \le a_0^{-1}\left[V_a(z)+H_z+(\mathcal E_L/2+1)t\right].
\end{equation}

Let $Z^{(i)}=(u^{(i)},v^{(i)})$, $i=1,2$, share the same noise and put
$h=u^{(1)}-u^{(2)}$, $k=v^{(1)}-v^{(2)}$. Then
\[
 \partial_t h=k,\qquad
 \partial_t k=-Ak-Ah-aP_0\bigl((u^{(1)})^3-(u^{(2)})^3\bigr).
\]
Use $\mathcal W(h,k)$ from \eqref{eq:res-energy}, with a fixed small
$\delta_0$. Its linear dissipation controls $\mathcal W$ and a positive
fraction of $\|k\|_{H^1}^2$. Moreover,
\[
 \bigl((u^{(1)})^3-(u^{(2)})^3,h\bigr)
 =\int h^2\bigl((u^{(1)})^2+u^{(1)}u^{(2)}+(u^{(2)})^2\bigr)\,\mathrm dx
 \ge0.
\]
Writing $U_i=\|u^{(i)}\|_{L^4}$, the remaining nonlinear pairing satisfies
\begin{align*}
 a\left|\bigl((u^{(1)})^3-(u^{(2)})^3,k\bigr)\right|
 &\le Ca(U_1^2+U_2^2)\|h\|_{H^1}\|k\|^{1/4}\|k\|_{H^1}^{3/4}\\
 &\le\eta\|k\|_{H^1}^2
   +C_\eta a^{8/5}(U_1^{16/5}+U_2^{16/5})\mathcal W(h,k).
\end{align*}
Choose $\eta$ below the available velocity dissipation. Since
$\mathfrak F_a(Z^{(i)})\ge aU_i^4/4$ and $x^{4/5}\le1+x$ for $x\ge0$,
there are $c_s,C_s>0$ as in the statement such that
\begin{equation}\label{eq:synchronous-difference}
 \frac{\mathrm d}{\mathrm dt}\mathcal W(h,k)
 \le\left[-c_s+\epsilon_a
   \bigl(2+\mathfrak F_a(Z^{(1)})+\mathfrak F_a(Z^{(2)})\bigr)\right]
   \mathcal W(h,k).
\end{equation}
These calculations are valid by Galerkin approximation; the common additive
noise cancels, and the difference velocity belongs to $L^2(0,T;H^1)$.
Let
\[
 \gamma_s=c_s-\epsilon_a\left(2+\frac{\mathcal E_L+2}{a_0}\right)>0.
\]
Combining \eqref{eq:synchronous-occupation} and
\eqref{eq:synchronous-difference} gives
\begin{equation}\label{eq:synchronous-path-contraction}
 \mathcal W(Z_t^z-Z_t^{z'})
 \le\mathcal W(z-z')
 \exp\left\{-\gamma_st+\frac{\epsilon_a}{a_0}
   \bigl(V_a(z)+V_a(z')+H_z+H_{z'}\bigr)\right\}.
\end{equation}
Thus every fixed pair of initial states synchronizes almost surely.

For $V_a(z),V_a(z')\le R$, split according to
$\{H_z\le K,H_{z'}\le K\}$. Norm equivalence and
\eqref{eq:martingale-envelope} yield, for $\varphi\in\operatorname{Lip}_b(\mathbb H)$,
\[
 \sup_{t\ge0}|P_t\varphi(z)-P_t\varphi(z')|
 \le C\operatorname{Lip}(\varphi)
       e^{\epsilon_a(R+K)/a_0}\|z-z'\|_{\mathbb H}
       +\frac{C_R\|\varphi\|_\infty}{K}.
\]
First choose $K$ large and then $\|z-z'\|$ small. This proves
\eqref{eq:e-property} independently of Sections~\ref{sec:stopping}--\ref{sec:control}.

Existence follows from Proposition~\ref{prop:existence}: its proof requires
only finite spectral range of $Q$. If $\mu_1,\mu_2$ are invariant, couple
initial states with law $\mu_1\otimes\mu_2$, independently of the common noise.
For each deterministic pair, \eqref{eq:synchronous-path-contraction} gives
almost sure convergence of the difference to zero. Fubini's theorem and
bounded convergence imply
\[
 |\mu_1(\varphi)-\mu_2(\varphi)|
 \le\mathbb E\min\{2\|\varphi\|_\infty,
       \operatorname{Lip}(\varphi)\|Z_t^{(1)}-Z_t^{(2)}\|\}\longrightarrow0.
\]
Bounded Lipschitz functions separate probability measures, so $\mu_1=\mu_2$.
For \eqref{eq:synchronous-convergence}, couple $z$ with an independent initial
state of law $\mu_a$ and repeat this estimate uniformly over functions with
sup norm and Lipschitz constant at most one. The bounded Lipschitz distance
to $\mu_a$ tends to zero. No exponential mixing rate follows from
\eqref{eq:synchronous-path-contraction} alone, since an exponential moment
of $H_z$ has not been established.
\end{proof}

\begin{remark}\label{rem:comparison-cubic}
For fixed $Q,S$, an explicit sufficient interval in
\eqref{eq:synchronous-smallness} is $0<a<a_{\rm syn}$, where
\[
 a_{\rm syn}=\min\left\{1,
 \left[\frac{c_s}{C_s\left(2+(\mathcal E_L+2)/a_0\right)}\right]^{5/4}\right\}.
\]
Proposition~\ref{prop:cubic-realization} instead constructs
$a_*(m,Q_m,S)$ through the fixed-rank leakage--clock balance.
Neither construction establishes an ordering of these thresholds.
The fixed-rank criterion gives unique ergodicity and Ces\`aro convergence.
For the cubic equation, synchronous coupling gives convergence of each
transition law under the separate explicit condition
\eqref{eq:synchronous-smallness}, with no minimum forcing rank.
\end{remark}

\subsection*{Declaration of competing interest}
The authors declare that they have no known competing financial interest or personal relationship that could have appeared to influence the work reported in this paper.

\subsection*{Data availability}
No datasets were generated or analysed in this theoretical study.

\subsection*{Declaration of AI use}
During the preparation of this manuscript, the authors used
AI to improve the language and clarity of expression. The authors reviewed and edited the
AI-assisted content and take full responsibility for the final manuscript.

\end{document}